\documentclass[reqno]{amsart}

\usepackage{amsfonts}
\usepackage{amssymb}
\usepackage{amsmath}

\usepackage{enumitem}
\usepackage{xcolor}
\usepackage{todonotes}

\setenumerate{label={\rm (\alph{*})}}
\usepackage{bbm,bm,euscript,mathrsfs}
\usepackage{paper_diening}
\usepackage{graphicx}
\usepackage[top=1in, bottom=1.25in, left=1.10in, right=1.10in]{geometry}
\usepackage{cases, color,amsthm} 

\usepackage{relsize}
\usepackage{hyperref}

\usepackage[nameinlink, noabbrev, capitalize]{cleveref}

\hypersetup{linkcolor=blue, colorlinks=true,citecolor = red}

\allowdisplaybreaks

\numberwithin{equation}{section}

\usepackage{tikz-cd}
\usetikzlibrary{lindenmayersystems}
\usetikzlibrary{decorations.pathreplacing}

\DeclareMathOperator{\Div}{div}

\newcommand{\R}{\mathbb R}

\newcommand{\dd}{\mathrm d}
\newcommand{\dt}{\,\mathrm{d} t}

\newcommand{\bx}{\mathbf{x}}
\newcommand{\by}{\mathbf{y}}

\newcommand{\bq}{\mathbf{q}}
\newcommand{\bu}{\mathbf{u}}

\newcommand{\bv}{\mathbf{v}}
\newcommand{\bn}{\mathbf{n}}

\newcommand{\dy}{\, \mathrm{d}\mathbf{y}}

\newcommand{\dx}{\, \mathrm{d} \mathbf{x}}

\newcommand{\divx}{\mathrm{div} }

\newcommand{\nabx}{\nabla }
\newcommand{\naby}{\nabla_{\mathbf{y}}}

\newcommand{\Delx}{\Delta }
\newcommand{\Dely}{\Delta_{\mathbf{y}}}

\newcommand{\dH}{\,\mathrm{d}\mathbf{\mathcal{H}}^2}
\newcommand{\dxt}{\,\mathrm{d}x\,\mathrm{d}t}
\newcommand{\ds}{\,\mathrm{d}s}

\newcommand{\Oeta}{\Omega_\eta}

\newtheorem{theorem}{Theorem}[section]
\newtheorem{lemma}{Lemma}[section]
\newtheorem{proposition}{Proposition}[section]
\newtheorem{corollary}{Corollary}[section]
\newtheorem{remark}{Remark}[section]

\newtheorem{definition}{Definition}[section]
\theoremstyle{definition}

\begin{document}

\title[Thermodynamics of compressible FSI]{Well-posedness and regularity in thermodynamics of compressible fluid-structure interactions}

\author{D. Breit}
\author{P. M. Ngougoue Ngougoue}
\address{Faculty of Mathematics, University of Duisburg-Essen, Thea-Leymann-Straße 9, 45127 Essen, Germany}
\email{dominic.breit@uni-due.de}
\email{pierre.ngougouengougoue@uni-due.de}


\date{\today}


\begin{abstract}
We consider the interaction of a general viscous compressible and heat-conducting fluid with an elastic shell located at the boundary of the fluid's domain. The fluid is described by the compressible Navier--Stokes--Fourier equations and the shell evolves in accordance with a viscoelastic beam equation. Both are coupled through kinematic boundary conditions and the balance of forces. 
Our first result is the local-in-time well-posedness of the underlying coupled system of nonlinear PDEs in smooth function spaces. 

Eventually, we prove a conditional regularity criterion which is in the spirit of the Beale--Kato--Majda condition for the incompressible Euler equations combined with the boundedness of density and temperature. It rests upon a weak-strong uniqueness result for the underlying system.
\end{abstract}

\maketitle

\section{Introduction}
 
The interactions of fluids with elastic structures are important for many applications
ranging from Hydroelasticity and Aeroelasticity to Biomechanics and Hydrodynamics. We are interested in the case, where a viscous fluid interacts with a flexible shell which is located at a part of the boundary (or even describes the complete boundary) of the underlying domain $\Omega\subset\R^3$ denoted by $\omega$. The shell, described by a function $\eta:(0,T)\times\omega\rightarrow\R$ for some $T>0$, reacts to the surface forces induced by the fluid and deforms the domain $\Omega$ to $\Omega_{\eta(t)}$. From a mathematical point of view the first step is to analyse the well-posedness of the underlying system of nonlinear PDEs. This opens the door for numerical approximations and, subsequently, engineering applications. Typically, strong solutions only exist in short-time such that one is also interested in conditions ensuring that a weak solution (which typically exists globally in time) does not blow up: one is looking for conditional regularity criteria.

\subsection{The fluid-structure interaction problem}
We consider the interaction of a viscous compressible and heat-conducting fluid and a flexible shell. The latter is located at the boundary of the fluid's reference domain $\Omega\subset \mathbb{R}^{3}$.
Accordingly, the shell function $\eta:(t, \by)\in I \times \omega \mapsto   \eta(t,\by)\in \mathbb{R}$ with $I=(0,T)$ solves\footnote{We use nondimensional variables and set the shell’s surface mass density, structural damping coefficient, and bending stiffness to 1, so no explicit constants appear in  \eqref{eq:ShellEq}.}
\begin{equation}\label{eq:ShellEq}
\left\{\begin{aligned}
& \partial_t^2\eta - \partial_t\Dely \eta + \Dely^2\eta=-\bn^\intercal(\bm{\tau}\circ\bm{\varphi}_\eta)\bn_\eta
\mathrm{det}(\naby \bm{\varphi}_{\eta}) 
&\text{ for all }  (t,\by)\in I\times\omega
,\\
&\eta(0,\by)=\eta_0(\by), \quad (\partial_t\eta)(0, \by)=\eta_*(\by)
&\text{ for all } \by\in\omega,
\end{aligned}\right.
\end{equation}
with periodic boundary conditions in space. Here, $\omega \subset \mathbb{R}^2$ is such that there is $\bfvarphi_\eta :\omega\to \partial \Omega_\eta$ that parametrises the boundary of the deformed domain $\Omega_\eta$, see Figure 2. We denote the fluid velocity and density  by 
\[\mathbf{v}:(t, \mathbf{x})\in I \times \Oeta \mapsto  \mathbf{v}(t, \mathbf{x}) \in \mathbb{R}^3  \quad \text{and} \quad  \rho:(t, \mathbf{x})\in I \times \Oeta \mapsto  \rho(t, \mathbf{x}) \in \mathbb{R}.  \]
The vectors $\bn$ and $\bn_\eta$ are the normal vectors  of the reference boundary and of the deformed boundary, respectively, whereas $\bm{\tau}$ denotes the Cauchy stress of the fluid given by {\em Newton's rheological law}, that is
$$\bm{\tau} =\mathbb{S}(\nabla\bv)-p(\rho,\vartheta)\mathbb I_{3\times 3}$$
with the viscous stress tensor
$$\mathbb{S}(\nabla\bv)=2\mu\left(\frac{1}{2}\big(\nabla\bv+(\nabla\bv)^\intercal\big)-\frac{1}{3}\Div \bv\mathbb{I}_{3\times 3}\right)+\left(\lambda+\frac{2}{3}\mu\right)\Div\bv\mathbb{I}_{3\times 3},$$
where the pressure $p=p(\rho,\vartheta)$ is described by the standard Boyle-Mariott equation of state
$$p(\rho,\vartheta)=\rho\vartheta.$$
The viscosity coefficients $\mu$ and $\lambda$ satisfy
$$\mu>0, \quad \lambda+\frac{2}{3}\mu\geq0. $$ 

The motion of the fluid in $\Omega_\eta$ is governed by the complete Navier--Stokes--Fourier equations:
\begin{equation}\label{eq:ContMomentEq}
\left\{\begin{aligned}
&\partial_t \rho  + \divx(\rho\mathbf{v} \big)
= 
0 &\text{ for all }(t,\bx)\in I\times\Omega_\eta,\\
&\partial_t(\rho \bv)+\Div(\rho\bv\otimes\bv)
= 
\mu\Delx \bv +(\lambda+\mu)\nabx\Div \bv-\nabx p(\rho,\vartheta) &\text{ for all }(t,\bx)\in I\times\Omega_\eta,\\
&c_v\partial_t(\rho \vartheta)+\Div(c_v\rho\vartheta\bv)
= 
\kappa\Delx \vartheta +\mathbb S(\nabla\bfv):\nabx\bv- p(\rho,\vartheta)\divx \bfv &\text{ for all }(t,\bx)\in I\times\Omega_\eta,\\
&\rho(0,\bx)=\rho_0(\bx), \quad
(\rho\bv)(0,\bx)=\bq_0(\bx), \quad
\vartheta(0,\bx)=\vartheta_0(\bx)&\text{ for all } \bx\in \Omega_{\eta_0}.
\end{aligned}\right.
\end{equation}
Here $c_v=\frac{1}{\gamma-1}>0$ with the adiabatic exponent $\gamma>1$ and $\kappa>0$ is the heat conductivity coefficient.
The equations \eqref{eq:ShellEq} and \eqref{eq:ContMomentEq} are coupled through the kinematic boundary condition
\begin{align}
\label{interfaceCond}
\bv\circ \bfvarphi_\eta= \big(\partial_t\eta\big)\bn \quad\text{ for all } (t,\by)\in I\times \omega,
\end{align} 
while the internal energy equation is complemented by
\begin{align}
\label{interfaceCond2}
\nabla\vartheta\cdot \bfn_\eta\circ\bfvarphi_\eta^{-1}=0 \quad\text{ for all } (t,\by)\in I\times \partial\Omega_\eta.
\end{align} 
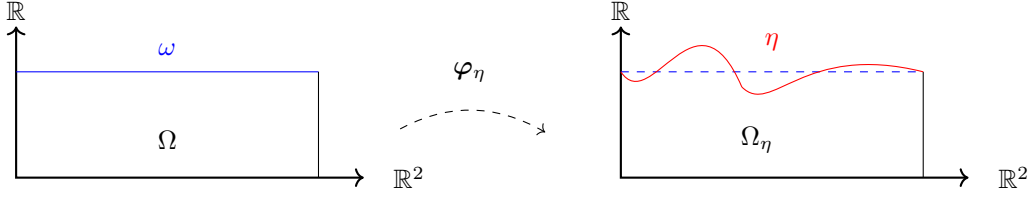
\begin{figure}\label{fig:1}
\begin{center}
\begin{tikzpicture}[scale=2]
  \begin{scope}     \draw [thick, <->] (0.5,0.5) -- (0.5,-0.5) -- (2.8,-0.5);
        \node at (0.5,0.6) {$\R$};
        \node at (1.5,-0.25) {$\Omega$};
        \node [blue] at (1.5,0.35) {$\omega$};
       \node at (3.1,-0.5) {$\R^2$};
        \draw [blue] (0.5,0.2) -- (2.5,0.2);
                \draw (2.5,0.2) -- (2.5,-0.5);
                   \node at (3.5,0.2) {$\bfvarphi_\eta$};
        \draw [thick, <->] (4.5,0.5) -- (4.5,-0.5) -- (6.8,-0.5);
        \node at (4.5,0.6) {$\R$};
       \node at (7.1,-0.5) {$\R^2$};
       \draw [blue,dashed] (4.5,0.2) -- (6.5,0.2);
                \draw (6.5,0.2) -- (6.5,-0.5);
                \draw [red] (4.5,0.2) .. controls (4.7,-0.1) and (5,0.8) .. (5.3,0.1);
                \draw [red] (5.3,0.1) .. controls (5.5,-0.1) and  (5.7,0.4) .. (6.5,0.2);
                        \node at (5.4,-0.25) {$\Omega_\eta$};
                         \node [red] at (5.5,0.4) {$\eta$};
          \draw [<-,dashed] (4,-0.2) to [out=150,in=30] (3,-0.2);
  \end{scope}
\end{tikzpicture}
\caption{Domain transformation in the simplified set-up, cf. \cite[fig. 1]{breit2024regularity}.}
\end{center}
\end{figure}
\begin{figure}\label{fig:2}
\begin{center}
\begin{tikzpicture}[scale=2]
  \begin{scope}     \draw [thick, <->] (0.5,0.5) -- (0.5,-0.5) -- (2.8,-0.5);
        \node at (0.5,0.6) {$\R$};
        \node at (1.3,-0.25) {$\Omega$};
          \node [blue] at (1.5,0.2) {$\omega$};
       \node at (3.1,-0.5) {$\R^2$};
       \draw (0.5,-0.5) .. controls (0,0) and (0.5,0.2) .. (0.7,0.3);
     \draw [blue] (0.5,0.2) .. controls (0.7,0.3) and (0.7,0.3) .. (0.7,0.3);
        \draw [blue] (0.7,0.3) .. controls (0.8,0.4) and (0.9,0.3) .. (1.1,0.2);
       \draw [blue] (1.1,0.2) .. controls (1.3,-0.3) and (1.5,0.3) .. (2.5,0);
       \draw (2.5,0) .. controls (2.8,-0.2) and (2.5,-0.9) .. (0.5,-0.5);
                   \node at (3.5,0.2) {$\bfvarphi_\eta$};
        \draw [thick, <->] (4.5,0.5) -- (4.5,-0.5) -- (6.8,-0.5);
        \node at (4.5,0.6) {$\R$};
       \node at (7.1,-0.5) {$\R^2$};
                     \draw (4.5,-0.5) .. controls (4,0) and (4.5,0.2) .. (4.5,0.2);
       \draw (6.5,0) .. controls (6.8,-0.2) and (6.5,-0.9) .. (4.5,-0.5);
        \draw [blue,dashed] (4.5,0.2) .. controls (4.7,0.3) and (4.7,0.3) .. (4.7,0.3);
        \draw [blue,dashed] (4.7,0.3) .. controls (4.8,0.4) and (4.9,0.3) .. (5.1,0.2);
       \draw [blue,dashed] (5.1,0.2) .. controls (5.3,-0.3) and (5.5,0.3) .. (6.5,0);
                \draw [red] (4.5,0.2) .. controls (4.7,-0.1) and (5,0.8) .. (5.3,0.1);
                \draw [red] (5.3,0.1) .. controls (5.5,-0.1) and  (5.7,0.4) .. (6.5,0);
                        \node at (5.4,-0.25) {$\Omega_\eta$};
                         \node [red] at (5.5,0.4) {$\eta$};
          \draw [<-,dashed] (4,-0.2) to [out=150,in=30] (3,-0.2);
  \end{scope}
\end{tikzpicture}
\caption{Domain transformation in the general set-up, cf. \cite[fig. 2]{breit2024regularity}.}
\end{center}
\end{figure}
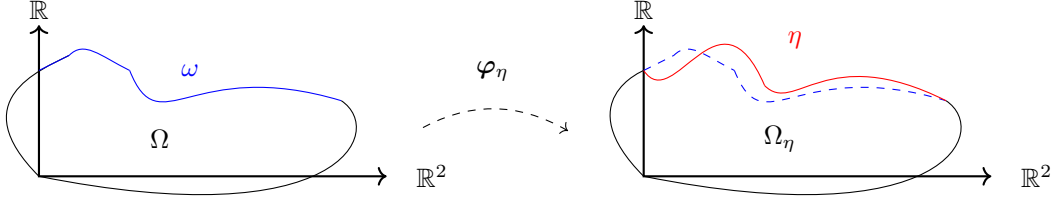

\subsection{The state of the art}
\textbf{Incompressible fluids:}
In the case of a homogeneous incompressible fluid the density $\rho$ is a positive constant. The first equation of \eqref{eq:ContMomentEq} reduces to $\Div\bfv=0$ and the pressure function $p$ is an unknown
(the equations for the temperature in \eqref{eq:ContMomentEq} together with \eqref{interfaceCond2} are decoupled and thus removed from the set of equations).
 The existence of weak solutions to \eqref{eq:ShellEq}--\eqref{interfaceCond} (even for the case of a general reference geometry as in Figure 2) is shown in \cite{LeRu,MuCa,MuSc}. Weak solution can be constructed to satisfy an energy inequality with the fluid's energy $\frac{1}{2}\int_{\Omega_\eta}|\bfv|^2\dx$. 

In the case of elastic plates there exists several results concerning the existence of local-in-time strong solutions, see \cite{DeSa,GraHilLe,Le,DRR,maity2021existence2}. In the 2D case these solutions exist globally in time, cf. \cite{GraHil}. The study of well-posedness for fluid-structure interactions with general reference geometries as in Figure 2 (the case of elastic shells) has only been started very recently. The existence of strong solutions to \eqref{eq:ShellEq}--\eqref{interfaceCond} has been shown in~\cite{breit2024regularity} (in 2D, globally in time) and \cite{breit2023ladyzhenskaya} (in 3D, locally in time).

Recently, a version of the classical Ladyzhenskaya--Prodi--Serrin condition for \eqref{eq:ShellEq}--\eqref{interfaceCond} has been proved in \cite{breit2023ladyzhenskaya}: under additional integrability conditions on the velocity field, the weak solution must be regular as well as unique in the class of weak solutions.
This is a consequence of an acceleration estimate
and a weak-strong uniqueness result for \eqref{eq:ShellEq}--\eqref{interfaceCond}. The latter generalises a result from \cite{ScSr} which only applies to elastic plates (Figure 1). 

\textbf{Isentropic compressible fluids:}
In the isentropic compressible Navier--Stokes equations the density $\rho:(0,T)\times\Omega_\eta\rightarrow [0,\infty)$ is an unknown function and the pressure relates to it via the adiabatic law
\begin{align*}
p=p(\rho)=\tfrac{1}{\mathrm{Ma}^2}\rho^\gamma,
\end{align*}
where $\mathrm{Ma}>0$ is the Mach-number and $\gamma>1$ is the adiabatic exponent. The existence of weak solutions to \eqref{eq:ShellEq}--\eqref{interfaceCond} in this case was shown in
\cite{breit2018compressible} (in the case of a general reference geometry as in Figure 2). These solutions satisfy an energy inequality, where the energy of the fluid system \eqref{eq:ContMomentEq} is given by \begin{align*}
\tfrac{1}{2}\int_{\Omega_\eta} \rho|\bfv|^2\dx+\tfrac{1}{(\gamma-1)\mathrm{Ma}^2}\int_{\Omega_\eta} \rho^\gamma\dx.
\end{align*}
The result from \cite{breit2018compressible} gives a counterpart to the celebrated theory by Lions \cite{lions1996mathematical} and Feireisl \cite{feireisl1}.
 The local-well-posedness in the case of elastic plates is studied in \cite{MRT}\footnote{In \cite{MRT} a second order diffusion operator is considered in \eqref{eq:ShellEq}.} for the flat reference geometry (see also \cite{Mi} for the 2D case).
The existence of a local-in-time strong solution in the case of elastic shells has been established very recently in \cite{ngougoue2025local} by the second author. Eventually, he proved also the conditional regularity of solutions in \cite{MN}. The main assumptions are, roughly speaking:
\begin{itemize}
\item A compressible version of the classical Serrin condition;
\item Boundedness of the density;
\item A Baele--Kato--Majda type condition for velocity and density.
\end{itemize}
A weak-strong uniqueness results is proved in \cite{MN} as well (see also \cite{trifunovic2023compressible} for weak-strong uniqueness in the case of elastic plates).

\textbf{Heat-conducting compressible fluids:}
The motion of a general compressible and heat-conducting fluid is described by the Navier--Stokes--Fourier equations. In addition to the velocity field $\bfv$ and density $\varrho$, the absolute temperature $\vartheta:(0,T)\times\Omega_\eta\rightarrow [0,\infty)$ is an unknown. In the case of an ideal gas the pressure law is given by\footnote{The results concerning weak solutions to the Navier--Stokes--Fourier equations are based on
 the more complicated pressure law $p\sim \vartheta\rho+\rho^\gamma+\vartheta^4$, cf. \cite{FeNobook}.}
\begin{align*}
p=p(\rho,\vartheta)=\rho\vartheta.
\end{align*}
The balance of mass and momentum (the first two equations in \eqref{eq:ContMomentEq}) are complemented by the balance of energy (the third equations in \eqref{eq:ContMomentEq}).
The existence of weak solutions to \eqref{eq:ShellEq}--\eqref{interfaceCond2} in its full extend is shown in
\cite{BrSc23} (in the case of a general reference geometry as in Figure 2). These solutions satisfy an energy equality, where the energy of the fluid system \eqref{eq:ContMomentEq} is given by
\begin{align*}
\tfrac{1}{2}\int_{\Omega_\eta} \rho|\bfv|^2\dx+\int_{\Omega_\eta} \rho c_v\vartheta\dx.
\end{align*}
Further results can be found in \cite{MMNRT} and \cite{TrWa}, where the possibility of heat-transfer through the shell is included.
 The local well-posedness of \eqref{eq:ShellEq}--\eqref{interfaceCond2} in the case of elastic plates is studied in \cite{maity2021existence}, while the case of elastic shells remains open.
Results regarding conditional regularity and weak-strong uniqueness seem to be missing completely.
The main purpose of this work is to fill these gaps.

\subsection{Main results}
\textbf{Local well-posedness:}
Our first main result is the local-in-time existence of a unique strong solution to \eqref{eq:ShellEq}--\eqref{interfaceCond2} in classes of smooth function spaces, see Theorem \ref{theo:mainresult} for the complete statement.
Our strategy of proof is
based on transformation a to the fixed reference domain, linearisation and a fixed point argument. For given to be completed in Section \ref{sec:localstrongfixedpoint} by establishing the usual self-mapping and contraction property of the fixed point map.
As a preparation we consider several subproblems with fixed geometry in Section \ref{sec:continuitysubprob}--\ref{sec:InternalProb}:
\begin{itemize}
\item The continuity subproblem,
where the velocity field is given.
\item The momentum-structure subproblem, where all nonlinear terms are prescribed through the given data (geometry, density, velocity and temperature)
\item The internal energy subproblem, where the density and velocity field are given.
\end{itemize}
The first two bullets are already analysed in \cite{MN} such that our focus is on the third one.
Here, a careful analysis of the transformed internal energy equation in the reference geometry is required.  After transformation we obtain an equation with coefficients depending on the structure function being of limited regularity.

\textbf{Conditional regularity:}
Our second main result concerns the conditional regularity of solutions. Complementing the conditions from the isentropic case from \cite{MN} mentioned above by according conditions for the temperature (boundedness and a Baele--Kato--Majda type condition) we can prove a blow-up condition for strong solutions, see  \cref{theo:MainResult}
and \cref{theo:MainResult2}. In a first step we prove an acceleration estimate, which is of its own interest. But, unfortunately, it does not yield enough regularity to obtain a strong solution in the sense of our definition. Hence we derive higher order estimates in a second step implying sufficient smoothness. Our main focus here is again the internal energy equation. The analysis of the transformed internal energy equations is more delicate than that from \cref{sec:InternalProb}.
In fact, the regularity of the structure obtained from the acceleration estimate is just critical to obtain the required maximal regularity estimate for the internal energy equation, cf.~estimate \eqref{eq:theend}. 

Finally, we obtain a weak-strong uniqueness result to justify the computations leading to conditional higher order estimates for weak solutions, see  \cref{cor:WeakStrongUniqueMain}.


\section{Mathematical framework and main results}
In this section we present the geometric framework
describing moving domains and our main results concerning well-posedness, conditional regularity and weak-strong uniqueness.
\subsection{Geometric Framework} 
We adopt the framework of   \cite[Section 2.4]{breit2024regularity}, and  assume that the \phantom{refer} reference domain $\Omega \subset \mathbb{R}^3 $ is open, bounded, and has a smooth, orientable boundary $\partial\Omega$, so that a consistent outward unit normal vector $\bn$ is well-defined and smooth.  To formalise the structure of the boundary $\partial\Omega$, we assume that it can be  parametrised by an injective map  $\bm{\varphi} \in C^{k}(\omega; \mathbb{R}^3) $  for some sufficiently large $k \in \mathbb{N}$.  For each $ \by = (y_1, y_2) \in \omega $, the pair of tangent vectors \( \partial_1 \bm{\varphi}(\by) \) and \( \partial_2\bm{\varphi}(\by) \) is assumed to be linearly independent. This ensures that $\bm{\varphi}$ defines a regular surface embedded in $\mathbb{R}^3 $, and the corresponding unit normal vector can be explicitly determined from $\bm{\varphi} $ by 
\[\bn(\by) = \dfrac{\partial_1\bm{\varphi}(\by) \times \partial_2 \bm{\varphi}(\by)}{| \partial_1 \bm{\varphi}(\by) \times \partial_2 \bm{\varphi}(\by) |}.  \]
Importantly, as a compact connected $C^k$-hypersurface in $\mathbb{R}^3$, $\partial\Omega$ satisfies the uniform interior and exterior ball condition (cf. \cite[Section 4.1]{PruessSimonett2013}), that is, we find some $\, L > 0 $ such that for each point $\by \in \partial\Omega$, there are balls  $B(\bx, L) \subset \Omega $ and $B(\mathbf{z}, L) \subset \overline{\Omega}^c$ such that 
\[\partial\Omega \cap \overline{B}(\bx, L) \cap \overline{B}(\mathbf{z}, L) = \{\by\}. \]
Hence, as in \cite[Section 3.1]{PruessSimonett2016} we conclude that $\partial\Omega $ admits a tubular neighbourhood of radius $L >0 $, defined as 
\begin{equation}\label{eq:TubNeighbor}
\mathcal{N}_L := \left\{ \bx \in \mathbb{R}^3  \colon \mathrm{d}(\bx, \partial\Omega) < L \right\}.
\end{equation}
Within $\mathcal{N}_L $, every point $\bx $ can be uniquely projected onto the boundary $ \partial \Omega $  via the closest point projection. More precisely, for each $ \bx \in \mathcal{N}_L $, we define the map \[\by(\bx) = \arg\min\limits_{\by \in \omega} |\bx - \bm{\varphi}(\by)|,  \] so that the closest point on the boundary is given by $ \mathbf{p}(\bx) := \bm{\varphi}(\by(\bx)) $, and the signed distance to the boundary is defined as $s(\bx) := (\bx - \by(\bx))\cdot \bn(\by(\bx))$.

\noindent For $L > 0$ small enough, the projection $\mathbf{p}(\bx) $ is well-defined, and we have 
\[
|s(\bx)| = |\bx - \by(\bx)| = \min\limits_{\by \in \omega} |\bx - \bm{\varphi}(\by)| =  \mathrm{d}(\bx, \partial\Omega) \qquad \text{for all } \; \bx \in \mathcal{N}_L.
\]
Indeed, for large $L >0$, the normal vectors from different points on $\partial\Omega $ might intersect, leading to ambiguity in projection.  Consequently, $\mathcal{N}_L $ provides a natural coordinate system for points $\bx \in \mathcal{N}_L$. Specifically, each point $\bx \in \mathcal{N}_L $ can be uniquely written as 
\[\bx = \mathbf{p}(\bx) + s(\bx)\bn(\by(\bx)),  \] 
which implies -- together with \eqref{eq:TubNeighbor} -- the representation 
\begin{equation}\label{eq:TubNeighborFinal}
\mathcal{N}_L := \big\{ \by + s\bn(\by) \colon (s, \by) \in (-L, L)\times \omega \big\}.
\end{equation}

Furthermore, recall that the shell displacement $\eta$ defines a deformation of the boundary in the normal direction, leading to the time-dependent boundary 
\[\partial\Omega_{\eta (t)} = \big\{ \bm{\varphi}(\by) + \eta(t, \by)\bn(\by) \colon \by \in \omega\big\},  \]
with associated deformation map $ \bm{\varphi}_\eta (t, \by) = \bm{\varphi}(\by) + \eta(t, \by)\bn(\by). $ \\
However, to ensure that $\eta$ defines a valid normal parametrisation and that the deformed boundary $\partial\Omega_{\eta (t)}$ remains well-defined and regular over time, we required the following conditions:\\
\begin{itemize}

    \item[$\bullet$] The deformed boundary must remain within the tubular neighbourhood, that is, $$ \partial\Omega_{\eta (t)}  \subset \mathcal{N}_L \qquad \text{for all} \; t \in I. $$
    For this purpose, the displacement $\eta $ must satisfy $$\sup\limits_I \Vert \eta \Vert_{L^{\infty}(\omega)} < L .$$
    However, this condition is not enough as the deformed boundary $ \partial\Omega_{\eta } $ -- even for small Hausdorff distance -- can still fold or self intersect if the normals at different points intersect. To prevent this, we need to control the derivatives of $\eta$ to ensure the mapping $\by \longmapsto \bm{\varphi}_{\eta}(t, \by) $ remains \phantom{injec} injective. More precisely: \\
    
    \item[$\bullet$] The tangential vectors of the deformed boundary must remain non-degenerate, that is, 
    \begin{equation}\label{eq:Condl}
    \partial_1\bm{\varphi}_{\eta} \times \partial_2 \bm{\varphi}_{\eta} (t, \by) \neq 0, \quad  \text{for all} \; (t, \by) \in I \times \omega . 
    \end{equation}
    Moreover, to avoid boundary flipping, that is, to ensure the orientation of the deformed boundary $\partial\Omega_\eta$ matches the reference boundary $\partial\Omega$, we require that:
    \begin{equation}\label{eq:Condr} 
    \bn(\by)\cdot \bn_{\eta(t)}(\by) > 0 \quad \text{for all }\;  (t, \by) \in I \times \omega . 
    \end{equation}
    
\end{itemize}    
Hence, up to decreasing $L > 0$, one easily deduces that \eqref{eq:Condl} and \eqref{eq:Condr} hold provided that 
\[
\sup_{t \in I} \| \eta(t, \cdot) \|_{W^{1,\infty}(\omega)} < L.     
\] 

Within the preceding geometric setup, the boundary displacement $\eta$  admits a smooth normal extension into the interior of $\Omega$  via the Hanzawa transform 
\begin{equation}\label{eq:HanzawaT}
\bfPsi_\eta(t, \bx) :=
\begin{cases}
\mathbf{p}(\bx) + \Big( s(\bx) + \eta\big(\by(\bx)\big)\phi\big(s(\bx) \big)  \Big)\bn\big(\by(\bx)\big) & \text{if } \bx \in \mathcal{N}_L, \\
\bx & \text{elsewhere},
\end{cases}
\end{equation}
where \( \phi \in C^\infty(\mathbb{R}) \) is a cut-off function satisfying \( \phi \equiv 1 \) near  the boundary, that is,   \( s \approx 0 \), and far from the boundary, that is, $s \leq -L$,  $\phi \equiv 0$, so the transformation reduces to the identity map. In other words, the deformation $\eta$ is smoothly `turned off' as we move deeper into the interior of $\Omega$.

\noindent The following proposition gives some estimates for $\bfPsi_\eta$ proved in  \cite[Section 2.4]{breit2024regularity}. 

\begin{proposition}\label{prop:estimatePsiEta}
Let \( k \in \mathbb{N} \), \( p \in [1, \infty] \), and assume that \( \eta,\, \zeta \in W^{k,p}(\omega) \). Then  \( \bfPsi_\eta \) satisfies 
\begin{align}
\Vert  \bfPsi_\eta \Vert_{W^{k,p}(\Omega)} &\lesssim 1 + \Vert \eta \Vert_{W^{k,p}(\omega)},  \label{eq:HanzEstim1}
\\
\Vert \bfPsi_\eta - \bfPsi_\zeta \Vert_{W^{k,p}(\Omega)} &\lesssim \Vert \eta - \zeta \Vert_{W^{k,p}(\omega)}, \label{eq:HanzEstim2}
\end{align}
where the hidden constants depend on the curvature of \( \partial \Omega \), and the radius \( L > 0 \).\\

\noindent Furthermore, provided that   \;$ \sup\limits_{t \in I} \Vert \eta(t, \cdot) \Vert_{W^{1,\infty}(\omega)} < L $,  the inverse map \( \bfPsi_\eta^{-1} \) exists and satisfies
\begin{align}
\Vert  \bfPsi_{\eta}^{-1} \Vert_{W^{k,p}(\Omega)} &\lesssim 1 + \Vert \eta \Vert_{W^{k,p}(\omega)}, \label{eq:HanzEstim3}
\\
\Vert \bfPsi_{\eta}^{-1}\circ\bfPsi_{\eta} - \bfPsi_{\zeta}^{-1}\circ\bfPsi_{\zeta} \Vert_{W^{k,p}(\Omega)} &\lesssim \Vert \eta - \zeta \Vert_{W^{k,p}(\omega)}. \label{eq:HanzEstim4}
\end{align}
\end{proposition}

\subsection{Local strong solutions}
 We present now
our main result concerning the well-posedness of the fluid-structure system \eqref{eq:ShellEq}--\eqref{interfaceCond2}. It concerns the local-in-time existence of a unique strong solution.
We speak about a strong solution provided its regularity is high enough to interpret all objects in the equations
as square integrable functions in space-time. Also, we require continuity in time (with values in $L^2$ in space) for the attainment of the initial data. Finally, the boundary conditions make sense in $L^2$ of space-time in the sense of traces.
Our main result is the following:

\begin{theorem}\label{theo:mainresult}
Assume the initial data $(\rho_0,\bv_0,\vartheta_0, \eta_0,\eta_*)$ satisfy 
\begin{align}
&\rho_0 \in W^{3,2}(\Omega_{\eta_0}),\qquad \bv_0,  \in W^{3,2}(\Omega_{\eta_0}), \qquad \vartheta_0  \in W^{3,2}(\Omega_{\eta_0}),   \qquad\eta_0\in W^{5,2}(\omega), \label{eq:InitialCondSpace}
\\&
\eta_*\in W^{3,2}(\omega),\qquad  \quad  \Vert\eta_0\Vert_{L^\infty(\omega)}<L, \qquad \quad     \bv_0\circ \bm{\varphi}_{\eta_0}=\eta_*\bn \;\; \text{ on } \omega. \label{eq:InitialCondInterface}
\end{align}
Then there  exists   $T_* \in I $ such that  \eqref{eq:ShellEq}--\eqref{interfaceCond2}  admits a unique strong solution  $(\rho,\bv, \vartheta,\eta)$ on $I_* := (0, T_*]$ satisfying 
\begin{align*}
&\rho\in L^{\infty}\big(I_*;W^{3,2}(\Omega_\eta)\big)\cap W^{1,\infty}\big(I_*;W^{2,2}(\Omega_\eta)\big),
\\
&\bv\in L^{2}\big(I_*;W^{4,2}(\Omega_\eta)\big)
\cap W^{2,2}\big(I_*;L^{2}(\Omega_\eta)\big),
\\
&\vartheta\in L^{2}\big(I_*;W^{4,2}(\Omega_\eta)\big)
\cap W^{2,2}\big(I_*;L^{2}(\Omega_\eta)\big),
\\
&\eta \in L^2\big(I_*;W^{6,2}(\omega)\big)  
\cap W^{3,2}\big(I_*;L^{2}(\omega)\big)
  \cap W^{2,\infty}\big(I_*;W^{1,2}(\omega)\big),
\\
&\partial_t\eta \in L^\infty\big(I_*;W^{3,2}(\omega)\big) \cap L^{2}\big(I_*;W^{4,2}(\omega)\big).
\end{align*}
\end{theorem} 
\begin{remark}
\begin{enumerate}\item[(a)]As in other papers on the local well-posedness for similar compressible fluid-structure interaction systems, the existence is only shown in a rather high topology, see \cite{MRT,maity2021existence,ngougoue2025local}.
It is still an open problem (even for isentropic fluids interacting with elastic plates) if it is possible
to obtain existence for less regular data. For fixed domains, such results can be found in \cite{ChKi} and \cite{ZhFa}, where the possibility of a vacuum is included. Classical existence results for the compressible Navier--Stokes--Fourier system work instead in ``our'' topology for $(\rho,\bfv,\vartheta)$, see \cite{Ta,Va}.
\item[(b)] In \cite{maity2021existence}
the existence of local strong solutions is studied in the framework of Besov spaces. The regularity of the underlying Besov spaces must by high enough to guarantee Lipschitz continuity of the velocity field and thus obtain a maximum principle for the density.
We expect that with some technical effort such a result
can be obtained in our set-up for elastic shells as well.
\end{enumerate}
\end{remark}

\subsection{Blow-up conditions}
In this section we present our results concerning the blow-up of strong solution. In the following subsection corresponding conditional regularity results are stated which are a consequence of the results presented here and the weak-strong result from  \cref{theo:WeakStrongUniqueMain}.

We recall that the total energy balance implies 
\begin{align}\label{eq:TotalEnergy}
&\sup\limits_{I_*} \left(  \int_{\Omega_\eta} \left(  \frac{1}{2} \rho|\bv|^2  + c_v \rho \vartheta  \right)\dx    +  \frac{1}{2} \int_{\omega} \left( |\partial_t \eta|^2 + |\Dely\eta|^2 \right) \dy  \right)   +  \int_{I_*} \int_{\omega} |\partial_t\naby\eta|^2 \dy\dt 
\leq  C_0,
\end{align} 
where \[C_0  = \dfrac{1}{2} \left(   \Vert \rho_0^{1/2}\bv_0\Vert_{L^2(\Omega_{\eta_0})}^2 +   \Vert \eta_*\Vert_{L^2(\omega)}^2    + \Vert \Dely\eta_0\Vert_{L^2(\omega)}^2          \right)    + c_v  \int_{\Omega_{\eta_0}}  \rho_0 \vartheta_0 \dx.          \]
Moreover the entropy balance\begin{align}\label{eq:entropy}
\int_{I_*}\int_{\Omega_\eta}\frac{1}{\vartheta}\Big[\mathbb S(\nabla\bfv):\nabla\bfv+\frac{\kappa|\nabla\vartheta|^2}{\vartheta}\Big]\dx\dt =  \int_{\Omega_{\eta(t)}} \rho(t)(c_v \log\vartheta(t)  - \log\rho(t))\dx\bigg|_{t=0}^{t=T_*}
\end{align}
holds for weak solutions, see \cite[equ. (2.18) with $\psi=1$]{BrSc23}.

We formulate now the blow-up conditions.

\begin{enumerate}[label={(A\arabic*)}]
    \item \label{A1}
    (\textbf{Serrin-type control of the momentum.})
    The velocity satisfies the integrability condition
    \[
        \rho^{1/2}\bv \in L^\mathtt{s}\big(I_*; L^\mathtt{r}(\Omega_\eta)\big), 
        \quad  \text{with }\;\;  \frac{2}{\mathtt{s}} + \frac{3}{\mathtt{r}} \le 1,\quad \mathtt{r} \in(3,\infty],\; \mathtt{s}\in[2,\infty).
    \]

    \item \label{A2}
    (\textbf{Control of compressibility and thermal effects.})
    The following conditions holds: \\[-0.25cm]
    \begin{enumerate}
        \item \label{A2a}
        $ \|\Div \bv\|_{L^1\big(I_*; L^\infty(\Omega_\eta)\big)} < \infty$; 
        \item[(b)] \label{A2c} 
        $ \Vert \vartheta\Vert_{L^\infty\big(I_*; L^\infty(\Omega_\eta)\big)} < \infty$.  \\[-0.15cm]
    \end{enumerate}
    \item \label{A3}
    (\textbf{Geometric regularity of the structure.})
    The shell displacement satisfies
    \[
        \eta \in L^\infty\left(I_*; C^1(\omega)\right),
    \]
    and the fluid--structure interface stays nondegenerate.        
    \item \label{A4}
     (\textbf{Beale--Kato--Majda type control}.)
 \[     \int_{I_*}\Big( \Vert \nabla\rho \Vert_{L^{\infty}(\Omega_\eta)}^2+\Vert \nabla\bv \Vert_{L^{\infty}(\Omega_\eta)}^2+\Vert \nabla\vartheta \Vert_{L^{\infty}(\Omega_\eta)}^2\Big)  \dt < \infty . \]
\end{enumerate}

Note that condition \ref{A1} is a compressible version of the classical Serrin condition for incompressible fluids. Assumptions \ref{A2} (a) directly implies boundedness (form above and below) of the density which is required for most
currently known blow-up condition for compressible fluids. The same applies for condition \ref{A2} (b) in the context of heat-dependent fluids.
Condition \ref{A3} appears also for incompressible fluids, see \cite{breit2023ladyzhenskaya}. Note that weak solutions satisfy already $\eta\in L^\infty(I_*;W^{2,2}(\omega))$ such that \ref{A3} requires only an instant of additional regularity. Under the assumptions  \ref{A1}--\ref{A3} we obtain the following result.

\begin{theorem}\label{theo:MainResult}
 Let $(\rho, \bv, \vartheta,\eta) $ be a strong solution of \eqref{eq:ShellEq}--\eqref{interfaceCond2} in the sense of \cref{theo:mainresult}.
 Suppose that Assumptions \ref{A1}--\ref{A3} hold, 
 then the following  acceleration estimate holds:
 \begin{equation}\label{eq:AccelEstimate}
\begin{aligned}
{\mathlarger{\mathtt{E}}}_{\mathrm{acc}} :=&  \sup\limits_{I_*} \int_\omega \left( |\partial_t\naby\eta|^2  +   |\naby\Dely\eta|^2\right) \dy   + \frac{\mu}{2}\sup\limits_{I_*} \int_{\Omega_\eta} |\nabla \bv|^2 \dx + \frac{\lambda +\mu}{2}\sup\limits_{I_*} \int_{\Omega_\eta} |\Div \bv|^2 \dx  
\\
&\quad +  \int_{I_*}\int_\omega \left( |\partial^2_{t}\eta|^2 + |\partial_t\Dely\eta|^2  \right) \dy\dt    +  \int_{I_*}\int_{\Omega_\eta} \left( |\nabla^2 \bv|^2 +  \rho|\partial_t \bv|^2   + |\nabla p|^2  \right) \dx\dt
\\
&\lesssim  \int_{I_*}\int_{\Omega_\eta}|\partial_t\vartheta|^2\dx\dt  +\int_\omega \left( |\naby\eta_*|^2 +  |\naby\Dely\eta_0|^2  \right) \dy\\&   + \int_{\Omega_{\eta_0}} \left( |\nabla \bv_0|^2 + |\Div \bv_0|^2     + \vartheta_0 \rho_0 \Div\bv_0  \right) \dx,
\end{aligned}
\end{equation}
where the implicit constant depends only on the norms  specified in Assumptions \ref{A1}--\ref{A3}, on $T_*$ and on the constant $C_0 > 0$ coming  from the total energy of the fluid--structure system, cf.~\eqref{eq:TotalEnergy}.

\end{theorem}

Importantly,
in the present heat-conducting model condition \ref{A4} is also required at the level of the acceleration estimate to control the temperature effects, see \cref{lem:MainResult} below.   Indeed, for a barotropic pressure law $p = a \rho^\gamma \; (a > 0, \; \gamma > 1),$ the renormalised continuity equation yields
\begin{equation}\label{eq:RenormCE}
\dot{p}  + \gamma p \Div\bv = 0 \quad (\text{with }  \dot{p} := \partial_t p  + \bv\cdot \nabla p ), 
\end{equation}
so that the pressure evolution  is completely determined by the density.  For the ideal gas law $p = \rho\theta$, however, 
\[
\dot{p}  +  p \Div\bv = \rho \dot{\theta}. 
\] 
Hence, the pressure evolution is no longer determined solely by the continuity equation, and  the \phantom{ac} acceleration estimate necessarily involves the thermal contribution $\rho \dot{\theta}$. The latter is  controlled through the  internal energy equation 
$\eqref{eq:ContMomentEq}_3$, whose viscous dissipation term $\mathbb{S}(\nabla\bv)\colon \nabla\bv$  gives rise to Assumption \ref{A4}. 

 Assumptions \ref{A1}--\ref{A4} are eventually sufficient to propagate higher-order regularity of the strong solution.
 We summarise the resulting higher-order a priori estimate in the following theorem. 

%

\begin{theorem}\label{theo:MainResult2}
 Let $(\rho, \bv, \vartheta, \eta) $ be a strong solution of \eqref{eq:ShellEq}--\eqref{interfaceCond2} in the sense of \cref{theo:mainresult}. Under Assumptions \ref{A1}--\ref{A4}, and the compatibility condition 
 \begin{equation}\label{eq:CC}
 \left(\partial_t\bv\right)\circ\bm{\varphi}_\eta  \raisebox{-1.6ex}{$\Big|_{t=0}$}  = \left( \partial^2_t \eta \right)\bn\raisebox{-1.6ex}{$\Big|_{t=0}$} \qquad \text{ on } I \times \omega  \tag{CC}
 \end{equation}
 the following a priori estimate holds:
 \begin{equation}\label{eq:HigherEstimate}
\begin{aligned}
{\mathlarger{\mathtt{E}}}_{\mathrm{high}} :=& \sup\limits_{I_*} \int_\omega \left( |\partial^2_t\naby\eta|^2  +   |\partial_t\naby\Dely\eta|^2\right) \dy   + \frac{\mu}{2}\sup\limits_{I_*} \int_{\Omega_\eta} |\partial_t\nabla \bv|^2 \dx 
\\
&\quad + \frac{\lambda +\mu}{2}\sup\limits_{I_*} \int_{\Omega_\eta} |\partial_t\Div \bv|^2 \dx  +  \int_{I_*}\int_{\Omega_\eta} \left( |\partial_t \nabla^2 \bv|^2 +  \rho|\partial^2_t \bv|^2 \right) \dx\dt 
\\
& \quad +      \int_{I_*}\int_\omega \left( |\partial^3_{t}\eta|^2 + |\partial^2_t\Dely\eta|^2 + |\partial_t\Dely^2\eta|^2 + |\Dely^3\eta|^2  \right) \dy\dt    
\\
& \quad + \int_{I_*} \int_{\Omega_\eta} \left( |\nabla^4\bv|^2+|\nabla^4\vartheta|^2  +  |\nabla^3 p|^2         \right) \dx\dt   
\\
& \quad + \sup\limits_{I_*} \int_{\Omega_\eta} \left( |\nabla^3 \bv|^2  + |\partial_t\nabla^2\vartheta|^2              + |\partial_t\nabla^2\rho|^2            \right) \dx
\\
&\lesssim  \Vert \eta_0 \Vert_{W^{5,2}(\omega)}^2 + \Vert \eta_* \Vert_{W^{3,2}(\omega)}^2   + \Vert \bv_0  \Vert_{W^{3,2}(\Omega_{\eta_0})}^2  + \Vert \rho_0  \Vert_{W^{3,2}(\Omega_{\eta_0})}^2 + \Vert \vartheta_0  \Vert_{W^{3,2}(\Omega_{\eta_0})}^2,
\end{aligned}
\end{equation}
where the hidden constant depends solely on the quantities specified in \ref{A1}--\ref{A4}, on  $T_*$, and on  the acceleration energy ${\mathlarger{\mathtt{E}}}_{\mathrm{acc}}$ defined in \eqref{eq:AccelEstimate}.
\end{theorem}

\begin{remark}
In the case of fixed domains it is proved in \cite{FeWaZh} that weak solutions to the Navier--Stokes Fourier system are regular provided density, velocity and temperature stay bounded. This is the celebrated Nash's conjecture. It remains unclear if a corresponding result can be achieved in the context of fluid-structure interaction.
\end{remark}

\subsection{Weak-strong uniqueness and conditional regularity}
We conclude with a weak-strong uniqueness statement. Importantly, we do not address here the construction of finite-energy weak solutions for the coupled  heat-conducting fluid--shell system  \eqref{eq:ShellEq}--\eqref{interfaceCond2}. Nevertheless, the relevant notion of weak-solution is naturally modelled on the thermodynamic framework developed for the full compressible Navier--Stokes--Fourier system on fixed domains by \cite{feireisl2012weak}. The result below should therefore be understood as a conditional uniqueness principle asserting that whenever a finite-energy weak solution   $(\rho_1, \bv_1, \vartheta_1, \eta_1 )$ and a strong solution $(\rho_2, \bv_2, \vartheta_2, \eta_2)$ emanating from the same initial data coexist, they necessarily coincide.  

The corresponding weak-strong uniqueness property for compressible  FSI systems with barotropic fluid was established in \cite{MN} by means of a relative entropy method. In that setting, the fluid contribution to the modulated energy is based on the pressure potential associated with the isentropic constitutive law $p(\rho) = a\rho^\gamma$. The presence of temperature in the present model requires a different stability  functional. Indeed, since the pressure now depends on both the density and the temperature through  $p(\rho, \vartheta)$, the barotropic pressure potential is no longer the appropriate quantity for comparing two fluid states.  For this purpose, we introduce the following ballistic free energy 
\begin{equation}\label{eq:BallisticEner}
{\mathlarger{ \mathtt{H}_{\vartheta_2} }}  (\rho_1, \vartheta_1) :=  \rho_1 \mathbf{e}(\rho_1, \vartheta_1) - \vartheta_2 \rho_1 \mathbf{s}(\rho_1, \vartheta_1) ,
\end{equation}
with the specific internal energy 
\[
\mathbf{e}(\rho_1, \vartheta_1) :=  c_v \vartheta_1, 
\]
and specific entropy 
\[
\mathbf{s}(\rho_1, \vartheta_1) := c_v \log\vartheta_1   - \log\rho_1 + \mathbf{s}_0, 
\]
so that the Gibbs' relation 
\[
\vartheta_1 \dd\mathbf{s}(\rho_1, \vartheta_1) =   \dd\mathbf{e}(\rho_1, \vartheta_1)  + p(\rho_1, \vartheta_1) \dd\!\left( \dfrac{1}{\rho_1}\right)
\]
holds.
The thermodynamic potential \eqref{eq:BallisticEner} plays the role of the pressure potential in the isentropic theory and constitutes the natural fluid contribution to the  relative entropy functional (see e.g., \cite{feireisl2012weak}). 

Since both solutions are defined on different moving domains, we  map the strong solution  $(\rho_2, \bv_2, \vartheta_2, \eta_2)$ onto the  domain $\Omega_{\eta_1}$ of the weak solution. To this end, we introduce the Hanzawa map   
\[
\bfPsi_{\eta_2 \to \eta_1 } :=   \bfPsi_{\eta_2}\circ \left( \bfPsi_{\eta_1 }\right)^{-1} \colon \Omega_{\eta_1} \to \Omega_{\eta_2} ,
\]
and define 
\begin{equation}
\bv^{\sharp}_2 = \bv_2 \circ \bfPsi_{\eta_2 \to \eta_1 }, \quad  \rho^{\sharp}_2 = \rho_2 \circ \bfPsi_{\eta_2 \to \eta_1 }, \quad   \vartheta^{\sharp}_2 = \vartheta_2 \circ \bfPsi_{\eta_2 \to \eta_1 }. 
\end{equation}   
We then compare $(\rho_1, \bv_1, \vartheta_1,  \eta_1)$ with $(\rho^{\sharp}_2, \bv^{\sharp}_2, \vartheta^\sharp_2,  \eta_2)$. 
Accordingly,  for  $t \in I_*$, we define the relative entropy with respect to $(\rho^{\sharp}_2, \bv^{\sharp}_2, \vartheta^{\sharp}_2,  \eta_2)$  as 
\begin{equation}\label{eq:EntropyRelativeInitial}
\begin{aligned}
{\mathlarger{\EuScript{E}} }_{\mathrm{rel}}\! \left( (\rho_1, \bv_1, \vartheta_1, \eta_1) \bm{\Big|} (\rho^{\sharp}_2, \bv^{\sharp}_2, \vartheta^{\sharp}_2, \eta_2) \right) (t) & :=  \dfrac{1}{2}  \int_{\Omega_{\eta_1(t)}} \!\!\!\! \rho_1 |\bv_1 - \bv^{\sharp}_2|^2 \dx  +  \int_{\Omega_{\eta_1(t)}} \!\!\!\! {\mathlarger{\mathcal{H}}}\!\left((\rho_1, \vartheta_1)\bm{\big|}( \rho^\sharp_2, \vartheta^\sharp_2) \right) \dx  
\\[0.4em]
& \quad + \dfrac{1}{2} \int_{\omega} \bigl(  |\partial_t \eta_1 - \partial_t \eta_2|^2  + |\Dely \eta_1 - \Dely \eta_2|^2 \bigr) \dy ,  
\end{aligned}
\end{equation}
where 
\[
{\mathlarger{\mathcal{H}}}\!\left((\rho_1, \vartheta_1)\bm{\big|}( \rho^\sharp_2, \vartheta^\sharp_2) \right) :=  {\mathlarger{ \mathtt{H}_{\vartheta^\sharp_2} }}  (\rho_1, \vartheta_1) - {\mathlarger{ \mathtt{H}_{\vartheta^\sharp_2} }} (\rho^\sharp_2, \vartheta^\sharp_2) - \partial_{\rho}{\mathlarger{ \mathtt{H}_{\vartheta^\sharp_2} }} (\rho^\sharp_2, \vartheta^\sharp_2) (\rho_1 - \rho^\sharp_2).
\]
The subsequent analysis follows the strategy developed in the barotropic setting by  \cite[Section 5]{MN}. The only substantial modification is  the replacement of the barotropic pressure potential by the thermodynamic potential \eqref{eq:BallisticEner} together with the additional contribution generated by the temperature equation $\eqref{eq:ContMomentEq}_3$.  Since the argument is identical to that of \cite{feireisl2012weak} (see also \cite{MN} for the isentropic problem in fluid-structure interaction), we omit the proof and state only the resulting estimate.

\begin{proposition}\label{prop:ws}
Let  $(\rho_1, \bv_1,\vartheta_1, \eta_1)$ be a finite-energy weak solution of \eqref{eq:ShellEq}--\eqref{interfaceCond2} with initial data $(\rho^0, \bv^0,\vartheta^0, \eta^0, \eta_{*})$  in the sense of  \cite[Section 2.5, Definition 2.15]{BrSc23}. Moreover, let $(\rho_2, \bv_2, \vartheta_2,\eta_2)$ be a strong solution of \eqref{eq:ShellEq}--\eqref{interfaceCond2} with the same initial data  $(\rho_2, \bv^0,\vartheta^0,\eta^0, \eta_{*})$.  Suppose, in addition,  that 
\begin{equation*}
\eta_1 \in L^\infty\left(I, C^1(\omega) \right) .
\end{equation*}
Then the following relative energy inequality holds:
\begin{equation}\label{eq:RelativEnergyWeakL}
\begin{aligned}
&{\mathlarger{\EuScript{E}} }_{\mathrm{rel}}\! \left( (\rho_1, \bv_1,\vartheta_1, \eta_1) \bm{\Big|} (\rho^{\sharp}_2, \bv^{\sharp}_2,\vartheta_2^{\sharp}, \eta_2) \right)\!(t) + \int_0^t  \int_{\Omega_{\eta_1(t)}}  \mathbb{S}\left(\nabla\bv_1 - \nabla\bv^\sharp_2\right)\colon \left(\nabla\bv_1 - \nabla\bv^\sharp_2\right)  \dx \ds 
\\
& + \int_0^t  \int_{\Omega_{\eta_1(t)}} \kappa\bigg(\frac{\vartheta_2^{\sharp}|\nabla\vartheta_1|^2}{\vartheta_1^2}+\frac{\nabla\vartheta_1
\cdot\nabla\vartheta_2^{\sharp}}{\vartheta_1}\bigg)  \dx \ds +   \int_0^t \int_\omega \bigl|\partial_t\naby\eta_1 - \partial_t\naby\eta_2\bigr|^2 \dy\ds  
\\
&\leq  {\mathlarger{\EuScript{E}} }_{\mathrm{rel}}\! \left( (\rho_1, \bv_1,\vartheta_1, \eta_1) \bm{\Big|} (\rho^{\sharp}_2, \bv^{\sharp}_2,\vartheta_2^{\sharp}, \eta_2) \right)\!(0)   + \int_0^t  {\mathlarger{\EuScript{R}}}\!\left(\rho_1, \bv_1, \vartheta_1,\eta_1 \bm{\big|} \rho^{\sharp}_2, \bv^{\sharp}_2, \vartheta_2^{\sharp},\eta_2 \right)   \ds ,
\end{aligned}
\end{equation}
where 
\begin{align*}
{\mathlarger{\EuScript{R}}}&\!\left(\rho_1, \bv_1, \vartheta_1,\eta_1 \bm{\big|} \rho^{\sharp}_2, \bv^{\sharp}_2, \vartheta_2^\sharp,\eta_2 \right)\\ & :=  \int_{\Omega_{\eta_1(t)}} \!\!\!\!\! \mathbb{S}\left(\nabla\bv^\sharp_2\right)\!\colon \! \left(\nabla\bv^\sharp_2 - \nabla\bv_1\!\right)  \dx  +  \int_{\Omega_{\eta_1(t)}} \!\!\!\!\! \rho_1 \left(\partial_s \bv^\sharp_2 + \bv_1\cdot \nabla\bv^\sharp_2\right) \cdot \left(\!\bv^\sharp_2 - \bv_1\!\right) \dx 
\\[0.2em]
&\quad +   \int_{\Omega_{\eta_1(t)}} \left(p(\rho^\sharp_2,\vartheta_2^{\sharp})  - p(\rho_1,\vartheta_1)\right) \Div\!\left(\bv^\sharp_2\right) \dx  
\\[0.2em]
&\quad - \int_{\omega} \!(\partial_s\eta_2 - \partial_s\eta_1)\bn\cdot \!\left(p(\rho^\sharp_2,\vartheta^\sharp_2)\bn_{\eta_1}\!\right)\!\circ \bm{\varphi}_{\eta_1}  \mathrm{det}(\naby\bm{\varphi}_{\eta_1}) \dy   + \int_{\omega}\! (\partial_s\eta_2 - \partial_s\eta_1)\Dely^2\eta_2 \dy
\\[0.2em]
&\quad  + \int_{\omega} (\partial_s\eta_2 - \partial_s\eta_1)\partial_s^2\eta_2 \dy    -  \int_{\omega} (\partial_s\eta_2 - \partial_s\eta_1)\partial_t \Dely\eta_2 \dy \\[0.2em]
&\quad+\int_{\Omega_{\eta_1(t)}}\bigg(\varrho \mathbf s (\rho^\sharp_2,\vartheta^\sharp_2) \partial_{t}\vartheta^\sharp_2+ \rho_1\bu_1 \mathbf s(\rho^\sharp_2,\vartheta^\sharp_2)\cdot\nabla_x\vartheta^\sharp_2\bigg)\, \dd x\dd t \nonumber\\
&\quad+\int_{\Omega_{\eta_1(t)}} \bigg(\Big(1-\frac{\rho_1}{\rho^\sharp_2}\Big)\partial_{t}p(\rho^\sharp_2,\vartheta^\sharp_2)-\frac{\rho_1\bu_1}{\rho^\sharp_2}\cdot\nabla_x p(\rho^\sharp_2,\vartheta^\sharp_2)\bigg)\,\dd x
\end{align*}
\end{proposition}
In order to estimate the remainder we follow \cite{feireisl2012weak} for the the terms related to the temperature, while estimating the rest is in accordance with \cite{MN}. 
Setting 
\begin{align*}
{\mathlarger{\mathcal{D}}}(t)&:=\int_0^t  \int_{\Omega_{\eta_1(t)}}  \mathbb{S}\left(\nabla\bv_1 - \nabla\bv^\sharp_2\right)\colon \left(\nabla\bv_1 - \nabla\bv^\sharp_2\right)  \dx \ds\\&+ \int_0^t  \int_{\Omega_{\eta_1(t)}} \kappa\vartheta_2^{\sharp}|\nabla\log(\vartheta_1)
-\nabla\log(\vartheta_2^{\sharp})|^2  \dx \ds+   \int_0^t \int_\omega \bigl|\partial_t\naby\eta_1 - \partial_t\naby\eta_2\bigr|^2 \dy\ds  
\end{align*}
we obtain for a.e. $t \in I$, 
\begin{equation}\label{eq:RemainderEstimFinal}
 \int_0^t  {\mathlarger{\EuScript{R}}}(s)  \ds  \leq C(\varepsilon) \int_0^t \mathtt{G}(s)  {\mathlarger{\EuScript{E}} }_{\mathrm{rel}}(s)  \ds + \varepsilon {\mathlarger{\mathcal{D}}}(t) , 
\end{equation}
with $\mathtt{G} \in L^1(I)$ and arbitrary $\varepsilon > 0$.  This implies uniqueness by Gr\"onwall's inequality, provided the initial data coincide and thus implies the following: 

\begin{theorem}\label{theo:WeakStrongUniqueMain}
Let  $(\rho_1, \bv_1,\vartheta_1, \eta_1)$ be a finite-energy weak solution of \eqref{eq:ShellEq}--\eqref{interfaceCond2} with initial data\\ $(\rho_1^0, \bv_1^0,\vartheta_1^0, \eta_1^0, \eta_{*, 1})$  in the sense of  \cite[Section 2.5, Definition 2.15]{BrSc23}. Moreover, let $(\rho_2, \bv_2, \vartheta_2,\eta_2)$ be a strong solution of \eqref{eq:ShellEq}--\eqref{interfaceCond2} with initial data  $(\rho_2^0, \bv_2^0,\vartheta_2^0,\eta_2^0, \eta_{*, 2})$.  Suppose, in addition,  that 
\begin{equation*}
\eta_1 \in L^\infty\left(I, C^1(\omega) \right) .
\end{equation*}
Then we have $(\rho_1, \bv_1,\vartheta_1, \eta_1)=(\rho_2, \bv_2, \vartheta_2,\eta_2)$.
\end{theorem}
In combination with \cref{theo:MainResult2} we obtain the following:
\begin{corollary}\label{cor:WeakStrongUniqueMain}
Let the initial data $(\rho_0, \bv_0, \vartheta_0,  \eta_0, \eta_*)$ satisfy \eqref{eq:InitialCondSpace}--\eqref{eq:InitialCondInterface}, and  let $(\rho, \bv, \vartheta,  \eta)$ be a finite-energy weak solution of \eqref{eq:ShellEq}--\eqref{interfaceCond2} in the sense of  \cite[Section 2.5, Definition 2.15]{BrSc23}.  Moreover, assume  that Assumptions  \ref{A1}--\ref{A4}, and  \eqref{eq:CC} hold on every finite interval $I = (0, T)$. \\
Then  $(\rho, \bv, \vartheta, \eta)$ is a strong solution on $I$ in the sense of \cref{theo:mainresult}.  In particular, $(\rho, \bv, \vartheta,  \eta)$ is unique in the class of weak solutions with deformation in $L^{\infty}\left( I; C^{1}(\omega) \right)$.
\end{corollary}

\section{Local well-poesedness}
\label{sec:well}

The aim of this section is the proof of \cref{theo:mainresult}. The proof will be completed in \cref{sec:localstrongfixedpoint}. Our strategy is
based on transformation to the fixed reference domain, linearisation and a fixed point argument.
As a preparation we consider several subproblems with fixed geometry:
\begin{itemize}
\item The continuity subproblem,
where the velocity field is given.
\item The momentum-structure subproblem, where all nonlinear terms are prescribed through the given data (geometry, density, velocity and temperature)
\item The internal energy subproblem, where the density and velocity field are given.
\end{itemize}
Note that this scheme differs from the approach for plates in \cite{maity2021existence}.

\subsection{The continuity subproblem} \label{sec:continuitysubprob}


For a known flexible domain $\Omega_\zeta $ and a known velocity field $\bu, $ we aim in this section to construct a strong  solution to the subproblem
\begin{equation}\label{rhoEquAlone}
\partial_t \rho +  (\bu\cdot\nabla )\rho = -\rho\, \divx\bu
\end{equation} 
in $I\times\Omega_\zeta\subset \mathbb R^{1+3}$ subject to the initial condition
\begin{align} 
&\rho(0,\cdot)=\rho_0(\cdot) &\text{in }\Omega_{\zeta(0)}.
\label{initialCondSolvSubPro}  
\end{align} 
Although a characteristic-based existence theory on the evolving domain $\Omega_{\zeta}$ is, in principle, feasible, it necessitates nontrivial control of trajectories to ensure they remain within $\Omega_{\zeta(t)}$ for all suitable time $t \in I$. To avoid this geometric subtlety, we reformulate the problem on the reference configuration $\Omega$ via the Hanzawa transform, see \eqref{eq:HanzawaT}. For this purpose, we introduce the pull-back variables $\overline\rho := \rho\circ\bfPsi_{\zeta}, \; \overline\bu := \bu\circ\bfPsi_{\zeta}$ and define $\mathbf{B}_{\zeta} = J_\zeta \left(\nabx \bfPsi_\zeta^{-1}\circ \bfPsi_\zeta\right)^\intercal$ where  $ J_\zeta=\mathrm{det}(\nabla\bfPsi_\zeta)$. 

\noindent Thus, the pull-back density $\overline\rho$ solves 
\begin{equation}\label{rhoEquAloneTransform}
\partial_t \overline\rho +  (\overline{\bu}_{eff} \cdot\nabla )\overline\rho  + \delta_{\Div}\,\overline\rho = 0,  
\end{equation} 
in $I\times\Omega \subset \mathbb R^{1+3}$, subject to the initial condition
\begin{align} 
&\overline{\rho}(0,\cdot)=\overline{\rho}_0(\cdot) &\text{in }\Omega,
\label{initialCondSolvSubProTransform}  
\end{align} 
where the effective transport field is given by
\begin{equation}
\overline{\bu}_{eff} := \partial_t \bfPsi_{\zeta}^{-1}\circ \bfPsi_{\zeta} + \dfrac{1}{J_{\zeta} } \mathbf{B}_{\zeta}^\intercal \overline\bu,
\end{equation}
and the local rate of compression -- that is, the transformed divergence of the Eulerian velocity field $\bu$ -- is defined as
\begin{equation}
\delta_{\Div} := \dfrac{1}{J_{\zeta} } \mathbf{B}_{\zeta} \colon \nabla\overline\bu.
\end{equation}
To proceed rigorously, we first introduce the notion of a strong solution.

\begin{definition} 
\label{def:strsolmartFP}
Let the data triple  $(\overline{\rho}_{0}, \overline\bu, \zeta)$ satisfy:
\begin{equation}
\begin{aligned}
\label{fokkerPlanckDataAlone}
&\overline{\rho}_{0} \in  W^{3,2}( \Omega ), 
\\&
\overline\bu\in
W^{2,2}\big( I;L^2(\Omega)  \big) \cap L^2\big(I;W^{4,2}(\Omega )  \big),
\\
&\zeta\in W^{2,\infty}\big(I;W^{1,2}(\omega)  \big) \cap W^{3,2}\big(I;L^{2}(\omega)  \big)\cap L^{2}\big(I;W^{6,2}(\omega)  \big)
,
\\& \overline\bu  \circ \bm{\varphi} =(\partial_t\zeta)\bn
\quad \text{on }I \times \omega,  \quad\|\zeta\|_{L^\infty(I\times\omega)} < L, .
\end{aligned}
\end{equation}
We call
$\overline\rho$
a \textit{strong solution} of   \eqref{rhoEquAloneTransform}--\eqref{initialCondSolvSubProTransform}  with dataset $(\overline{\rho}_0, \overline\bu, \zeta)$ if 
\begin{itemize}
\item[(a)] $\overline\rho \in   W^{1,\infty} \big(I; W^{2,2}(\Omega ) \big)\cap L^\infty\big(I;W^{3,2}(\Omega )  \big)$; 
\item[(b)] Equation \eqref{rhoEquAloneTransform} holds a.e. in $I\times\Omega$.
\end{itemize}
\end{definition}


\noindent We now formulate the result on the existence of a unique strong solution to the pull-back continuity equation \eqref{rhoEquAloneTransform}--\eqref{initialCondSolvSubProTransform} from \cite[Theorem 2.3]{ngougoue2025local}.

\begin{theorem}\label{thm:mainFP}
Let  $(\overline{\rho}_0,  \overline\bu, \zeta)$ satisfy  \eqref{fokkerPlanckDataAlone}.
Then there exists a unique strong solution $\overline\rho $ to the pull-back continuity equation  \eqref{rhoEquAloneTransform}--\eqref{initialCondSolvSubProTransform}, in the sense of \cref{def:strsolmartFP}, 
 such that
 \begin{equation}
\begin{aligned} \label{eq:ContSubProbEstimate}
\sup_{t\in I} &\Big( \Vert \overline{\rho}\Vert_{W^{3,2}(\Omega )}^2 
+
\Vert \partial_t\overline{\rho}\Vert_{W^{2,2}(\Omega )}^2 
\Big)
\\&\lesssim
 \Vert  \overline{\rho}_0\Vert_{W^{3,2}(\Omega)}^2  
\Bigg(1 + \sup\limits_I \Vert \partial_t \zeta \Vert_{W^{3,2}(\omega)}^2  + \int_I\Vert  \overline\bu
\Vert_{W^{4,2}(\Omega )}^2 \dt
\\
&\quad + 
\int_I\Vert  \partial_{t}^2 \overline\bu
\Vert_{L^{2}(\Omega)}^2 \dt
\Bigg)
  \exp{\bigg( c\int_I \big(\Vert \partial_t \zeta \Vert_{W^{4,2}(\omega)} + 
\Vert  \overline\bu\Vert_{W^{4,2}(\Omega)}\big)  \dt \bigg)}
\end{aligned}
\end{equation} 
holds for some constant $c = c(L) > 0$. 
\end{theorem}

%
%
%
%
%
%


\subsection{The momentum-structure subproblem}\label{sec:MomStrucSubProb}
\noindent Given a density and  temperature
\begin{equation}\label{eq:FixedDensitySpace}
 \overline{\varrho} \in L^\infty\big(I;W^{3,2}(\Omega)\big)\cap W^{1,\infty}\big( I;W^{2,2}(\Omega)\big),  \qquad \overline{\theta}   \in L^2\big(I;W^{4,2}(\Omega)\big)\cap W^{2,2}\big( I;L^{2}(\Omega)\big),
 \end{equation}
we study the well-posedness of the coupled momentum-structure subproblem with unknowns $( \bv, \eta) $. To prove existence and uniqueness of solutions, we rely on the classical \textit{linearisation--fixed-point}  framework developed in  \cite{ngougoue2025local} for barotropic compressible FSI problem.  Compared to  the latter  setting, the only essential difference lies in the pressure law, which now depends on both the density and the temperature. Since these quantities are prescribed throughout this section, the overall structure of the argument remains unchanged.  We therefore,  restrict ourselves  to outlining the main steps of the proof strategy, highlighting the changes induced by the pressure $p\big(\overline{\varrho}, \overline{\theta}\;\!\big)$, and refer to  \cite[Section 3]{ngougoue2025local}  for the complete analysis.


\begin{definition}\label{def:strongSolution}
Let the dataset $\left(\overline{\varrho}, \overline{\theta} ,  \eta_0, \eta_*, \bv_0\right)$ satisfies \eqref{eq:FixedDensitySpace},  \eqref{eq:InitialCondSpace} and \eqref{eq:InitialCondInterface}. 
We call a pair $( \bv,  \eta ) $ a strong solution to the resulting momentum--structure subproblem  \eqref{eq:ShellEq}--\eqref{interfaceCond}  with data $\left(\overline{\varrho}, \overline{\theta} ,  \eta_0, \eta_*, \bv_0\right)$ provided that the following conditions hold:
\begin{itemize}

\item[(a)] The structure displacement $\eta $ satisfies 
\begin{align*}
&\eta \in L^2\big(I; W^{6,2}(\omega)\big)  
\cap W^{3,2}\big(I ; L^{2}(\omega)\big)
  \cap W^{2,\infty}\big(I; W^{1,2}(\omega)\big),
\\
&\partial_t\eta \in L^\infty\big(I; W^{3,2}(\omega)\big) \cap L^{2}\big(I; W^{4,2}(\omega)\big). 
\end{align*}

\item[(b)]  The velocity field $\bv$ satisfies 
\begin{align*}
&\bv\in L^{2}\big(I; W^{4,2}(\Omega_\eta)\big)
\cap W^{2,2}\big(I; L^{2}(\Omega_\eta)\big). 
\end{align*}

\item[(c)]  The momentum--structure subsystem \eqref{eq:ShellEq}--\eqref{interfaceCond} -- with prescribed $\left(\overline{\varrho}, \overline{\theta} \right)$ -- holds a.e. in $I\times \Omega_\eta$. 

\end{itemize}

\end{definition}

For a solution $\left(\varrho, \theta, \bv,  \eta\right)$  of \eqref{eq:ShellEq}--\eqref{interfaceCond2}, we set the new function
$
\overline{\bv}=\bv\circ \bfPsi_\eta$, while $\overline{\varrho}=\varrho\circ \bfPsi_\eta$ and $\overline{\theta}  = \theta \circ \bfPsi_\eta$.
Moreover, we define
\begin{equation*}\label{matrices}
\begin{aligned}
\mathbf{A}_\eta=J_\eta \big(\nabx \bfPsi_\eta^{-1}\circ \bfPsi_\eta\big)\big( \nabx \bfPsi_\eta^{-1}\circ \bfPsi_\eta \big)^\intercal,&\\
\mathbf{B}_\eta=J_\eta \left(\nabx \bfPsi_\eta^{-1}\circ \bfPsi_\eta\right)^\intercal,
&\\
\mathbf{h}_\eta(\overline{\bv}) = \left(J_{\eta_0}\overline{\varrho}_0 - J_{\eta}\overline{\varrho} \right)\partial_t\overline{\bv}-J_\eta\overline{\varrho}\,\nabla\overline{\bv}\cdot\partial_t\bfPsi_\eta^{-1}\circ\bfPsi_\eta - \overline{\varrho}\, \overline{\bv}\big( \nabx\overline{\bv} \colon \mathbf{B}_\eta \big),&
\\
\mathbf{H}_\eta(\overline{\bv})
=\mu\left(\mathbf{A}_{\eta_0}-\mathbf{A}_{\eta}\right)\nabla\overline{\bv} + (\lambda+\mu)\left[ \dfrac{1}{J_{\eta_0}}\left(\mathbf{B}_{\eta_0}\colon \nabla\overline\bv \right)\mathbf{B}_{\eta_0} -  \dfrac{1}{J_\eta}\left(\mathbf{B}_\eta\colon\nabla\overline\bv \right)\mathbf{B}_\eta  \right]  + p\big(\overline{\varrho}, \overline{\theta}\;\!\big)\big(\mathbf{B}_{\eta} - \mathbf{B}_{\eta_0}  \big) ,
\end{aligned}
\end{equation*}
where $J_\eta=\mathrm{det}(\nabla\bfPsi_\eta)$.

\noindent To implement the linearisation-maximal-regularity fixed-point scheme, we rewrite the system \eqref{eq:ShellEq}--$\eqref{eq:ContMomentEq}_2$  on the fixed reference domain. The following lemma recasts the momentum-structure subsystem on the fixed reference domain through \(\bfPsi_\eta\).  

\begin{lemma}\label{lem:EquivProblem}
	Suppose that the dataset
	$\left(\overline{\varrho}, \overline{\theta},  \eta_0, \eta_*, \bv_0\right)$
	satisfies  \eqref{eq:FixedDensitySpace}, \eqref{eq:InitialCondSpace} and \eqref{eq:InitialCondInterface}.
	Then $( \bv,  \eta )$ is a strong solution to the momentum--structure subproblem \eqref{eq:ShellEq}--\eqref{interfaceCond} in the sense of Definition \ref{def:strongSolution}, if and only if $( \overline{\bv}, \eta  )$ is a strong solution of
	\begin{align}
	J_{\eta_0}\overline{\varrho}_0 \,\partial_t\overline{\bv} - \Div\left[\mu \mathbf{A}_{\eta_0}\nabla\overline{\bv} + \dfrac{ \lambda + \mu }{J_{\eta_0}} \left(\mathbf{B}_{\eta_0}\colon \nabla\overline\bv\right)\mathbf{B}_{\eta_0} -  p\big(\overline{\varrho}, \overline{\theta}\;\!\big)\mathbf{B}_{\eta_0}  \right] 
	= \mathbf{h}_\eta(\overline{\bv})-\Div{\mathbf{H}_\eta(\overline{\bv})} ,
	\label{momEqAloneBar}
	\\
	\partial_t^2\eta - \partial_t\Dely \eta + \Dely^2\eta
	=
   \bn^\intercal \left[\mathbf{H}_\eta( \overline{\bv}) - \mu \mathbf{A}_{\eta_0}\nabla\overline{\bv} - \dfrac{\lambda + \mu}{J_{\eta_0}}\left(\mathbf{B}_{\eta_0}\colon \nabla\overline\bv \right)\mathbf{B}_{\eta_0}  +  p\big(\overline{\varrho}, \overline{\theta}\;\!\big) \mathbf{B}_{\eta_0} \right]\circ\bm{\varphi} \bn ,
	\label{shellEqAloneBar}
	\end{align}
	with  $\overline{\bv}  \circ \bm{\varphi}  =(\partial_t\eta)\bn$ on $I\times \omega$.
\end{lemma}

\noindent We now state the results from \cite[Theorem 3.3]{ngougoue2025local} and \cite[Proposition 3.1]{ngougoue2025local}.

\begin{theorem}\label{thm:transformedSystem1}
Let  the dataset
	$\left(\overline{\varrho}, \overline{\theta},  \eta_0, \eta_*, \bv_0\right)$
	satisfy  \eqref{eq:FixedDensitySpace}, \eqref{eq:InitialCondSpace} and \eqref{eq:InitialCondInterface}. Then there exists  $T_{**}  \in I,$ such that \eqref{momEqAloneBar}--\eqref{shellEqAloneBar}  admits a unique strong solution $(\overline{\bv}, \eta)$ on  $I_{**} := (0, T_{**}].$
\end{theorem}

Suppose that the data $\left(\overline{\varrho}, \overline{\theta},  \eta_0, \eta_*, \bv_0\right)$ satisfies
\begin{align}
	&\overline{\varrho} \in L^\infty\big(I;W^{3,2}(\Omega)\big)\cap W^{1,\infty}\big( I;W^{2,2}(\Omega)\big), \qquad \overline{\theta}   \in L^2\big(I;W^{4,2}(\Omega)\big)\cap W^{2,2}\big( I;L^{2}(\Omega)\big),
	\; \tag{\ref{eq:FixedDensitySpace}} 
	\\
	& \qquad \qquad \qquad  \eta_0\in W^{5,2}(\omega), \qquad \eta_*\in W^{3,2}(\omega),\qquad \bv_0 \in W^{3,2}(\Omega_{\eta_0}). \label{initialcond}
\end{align}
Let the functions $(\mathbf{h}, \mathbf{H})$ satisfy
\begin{align}
	&\mathbf{h} \in L^2(I;W^{2,2}(\Omega))\cap W^{1,2}(I;L^{2}(\Omega)), \qquad
	\mathbf{h}(0)\in {W^{1,2}(\Omega)}, \nonumber
	\\&
	\mathbf{H} \in L^2(I;W^{3,2}(\Omega))\cap W^{1,2}(I;W^{1,2}(\Omega)), 
	\qquad \mathbf{H}(0) = 0,   \label{sourcecond}
\end{align}
where $\mathbf{h}(0)$ and $\mathbf{H}(0)$  are  the initial values of $\mathbf{h}$ and $\mathbf{H}$ respectively. We linearise the system \eqref{momEqAloneBar}--\eqref{shellEqAloneBar} as follows:
\begin{align}
	J_{\eta_0}\overline{\varrho}_0\,\partial_t\overline{\bv} - \Div\left[\mu \mathbf{A}_{\eta_0}\nabla\overline{\bv} + \dfrac{ \lambda + \mu }{J_{\eta_0}} \left(\mathbf{B}_{\eta_0}\colon \nabla\overline\bv\right)\mathbf{B}_{\eta_0} - p\big(\overline{\varrho}, \overline{\theta}\;\!\big)\mathbf{B}_{\eta_0} \right]  
	= \mathbf{h}-\Div{\mathbf{H}} , \label{linMomAlone}
	\\
	\partial_t^2\eta - \partial_t\Dely \eta + \Dely^2\eta
	=
   \bn^\intercal\left[\mathbf{H} - \mu \mathbf{A}_{\eta_0}\nabla\overline{\bv} - \dfrac{\lambda + \mu}{J_{\eta_0}}\left(\mathbf{B}_{\eta_0}\colon \nabla\overline\bv \right)\mathbf{B}_{\eta_0}   + p\big(\overline{\varrho}, \overline{\theta}\;\!\big)\mathbf{B}_{\eta_0}  \right]\circ\bm{\varphi} \bn ,\label{linShelAlone}
	\end{align}
with  $\overline{\bv}  \circ \bm{\varphi}  =(\partial_t\eta)\bn$ on $I\times \omega$. \\

As a preliminary step for the fixed-point scheme, we require well-posedness for the system \eqref{linMomAlone}–\eqref{linShelAlone}. 
The following proposition provides existence, uniqueness, and the required maximal-regularity estimate under compatible data assumptions specified therein.

\begin{proposition}
	\label{thm:transformedSystem}
Suppose that the dataset $\left(\overline{\varrho},  \overline{\theta}, \mathbf{\bv}_0, \eta_0, \eta_*, \mathbf{h}, \mathbf{H}\right)$ satisfy \eqref{eq:FixedDensitySpace}--\eqref{sourcecond} supplemented by  the compatibility condition
 \begin{equation}\label{eq:CC}
 \begin{aligned}
  \Bigg( \dfrac{1}{\overline{\varrho}_{0}J_{\eta_0}} \Big(  \Div\big( \overline{\bm{\tau} }_0 -  \mathbf{H}(0) \big)   + \mathbf{h}(0)  \Big)  \Bigg)\circ \bm{\varphi} 
 & =  \Bigg( \Dely\eta_{*} - \Dely^{2}\eta_{0}  + \bn^\intercal\Big(  \mathbf{H}(0)  - \overline{\bm{\tau} }_0 \Big)\circ\bm{\varphi}\bn   \Bigg)\bn  \quad \text{on} \;\; \omega,
  \end{aligned}  \tag{$\mathtt{CC}$}
 \end{equation}\\
 where 
 \[  \overline{\bm{\tau} }_0 :=  \mu \mathbf{A}_{\eta_0}\nabla\overline{\bv}_0 + \dfrac{ \lambda + \mu }{J_{\eta_0}} \big(\mathbf{B}_{\eta_0}\colon \nabla\overline{\bv}_0\big)\mathbf{B}_{\eta_0}  - p\big(\overline{\varrho}_0, \overline{\theta}_0\big)\mathbf{B}_{\eta_0} . \]  
	Then there exists a unique  strong solution $( \overline{\bv}, \eta  )$ of \eqref{linMomAlone}--\eqref{linShelAlone} satisfying
\begin{equation}
\begin{aligned}
& \left(\lambda + 2\mu\right)\sup_I\int_\Omega|\partial_t\nabla\overline\bv|^2\dx +\sup_I \int_\omega\left(|\partial_t^2\naby\eta|^2+|\partial_t\naby\Dely\eta|^2\right)\dy + \int_I\int_\Omega |\partial_t^2\overline\bv|^2 \dx\dt
\\
&\qquad   + \left(\lambda + 2\mu\right)\int_I\int_\Omega \left( |\partial_t\nabla^{2}\overline\bv|^{2}  + |\Delta^{2}\overline\bv|^{2} \right)\dx\dt
\\
&\qquad  +\int_I\int_\omega\left(|\partial_t^2\Dely\eta|^2+|\partial_t^3\eta|^2+|\partial_t\Dely^2\eta|^2+|\Dely^3\eta|^2\right)\dy\dt 
\\
& \lesssim 
\Vert \eta_0\Vert_{W^{5,2}(\omega)}^2
+
\Vert \eta_*\Vert_{W^{3,2}(\omega)}^2
+
\Vert \overline{\varrho}_0\Vert_{W^{3,2}(\Omega)}^2
+
\Vert\overline\bv_0\Vert_{W^{3,2}(\Omega)}^2  +
\Vert \mathbf{h}(0)\Vert_{W^{1,2}(\Omega)}^2
+
\Vert \mathbf{H}(0)\Vert_{W^{2,2}(\Omega)}^2
\\
&\qquad +\int_I \left(
\Vert \mathbf{h}\Vert_{W^{2,2}(\Omega)}^2
+
\Vert \mathbf{H}\Vert_{W^{3,2}(\Omega)}^2
+
\Vert\partial_t \mathbf{h}\Vert_{L^{2}(\Omega)}^2
+
\Vert\partial_t\mathbf{H}\Vert_{W^{1,2}(\Omega)}^2  \right)  \dt
\\
& \qquad +  \Vert p\big(\overline{\varrho}, \overline{\theta}\;\!\big)\Vert_{L^{2}\big( I; W^{3,2}(\Omega ) \big) }^2
+
\Vert \partial_t p\big(\overline{\varrho}, \overline{\theta}\;\!\big)\Vert_{L^{2}\big(I; W^{1,2}(\Omega ) \big)}^2   .
 \label{linearestimate}
\end{aligned}
\end{equation}
	\end{proposition}


\subsection{The internal energy subproblem}\label{sec:InternalProb}
In the moving domain,  equation $\eqref{eq:ContMomentEq}_3$ reads as
\begin{equation*}
\left\{\begin{aligned}
c_v\rho\big( \partial_t\vartheta+(\bv\cdot\nabla)\vartheta)
&= 
\kappa\Delx \vartheta +\mathbb S(\nabla\bfv):\nabx\bv- p(\rho,\vartheta)\divx \bfv    &  \text{in } I \times \Omega_\eta,  
\\
\nabla\vartheta\cdot \bfn_\eta\circ\bfvarphi_\eta^{-1}&=0  &  \text{on }  I\times \partial\Omega_\eta, 
\\
\vartheta(0,\cdot)&=\vartheta_0(\cdot) & \text{in }\Omega_{\eta(0)}.
\end{aligned}\right.
\end{equation*} 
After a change of variables induced by the Hanzawa transform, we obtain
\begin{equation*}
\left\{\begin{aligned}
c_vJ_\eta\overline\rho\partial_t\overline\vartheta-
\kappa\Div(\bfA_\eta \nabla\overline\vartheta) 
&=-c_v\overline\rho(\overline\bv\cdot\bfB_\eta\nabla)\overline\vartheta -J_\eta\overline{\rho}\,\nabla\overline{\vartheta}\cdot\partial_t\bfPsi_\eta^{-1}\circ\bfPsi_\eta  
\\
&\quad +J_\eta^{-1}\mathbb S(\bfB_\eta\nabla\overline\bfv):\bfB_\eta\nabx\overline\bv- p(\overline\rho,\overline\vartheta)\bfB_\eta :\nabla\overline\bfv    & \text{in }   I\times \Omega,
\\[0.4em]
\bfA_\eta\nabla\overline\vartheta\cdot \bfn\circ\bfvarphi^{-1}&=0 & \text{on } I\times \partial\Omega,\\
\overline\vartheta(0,\cdot)&=\overline\vartheta_0(\cdot) & \text{in }\Omega.
\end{aligned}\right.
\end{equation*} 
This is equivalent to
\begin{equation}\label{eq:31.03}
\left\{\begin{aligned}
\partial_t\overline\vartheta-
\Div\bigg(\frac{\kappa }{c_vJ_{\eta_0}\overline\rho_0} \bfA_{\eta_0} \nabla\overline\vartheta\;\!\bigg) 
&=\mathfrak g_\eta(\overline\varrho,\overline\bv,\overline\vartheta)  & \text{in  } I\times \Omega,
\\[0.4em]
\frac{\kappa }{c_vJ_{\eta_0}\overline\rho_0}\bfA_{\eta_0}\nabla\overline\vartheta\cdot \bfn\circ\bfvarphi^{-1}&=\mathfrak h_\eta(\overline\varrho,\overline\vartheta) & \text{on } I\times \partial\Omega,
\\[0.4em]
\overline\vartheta(0,\cdot)&=\overline\vartheta_0(\cdot) & \text{in }\Omega ,
\end{aligned} \right.
\end{equation} 
where
\begin{align*}
\mathfrak g_\eta(\overline\rho,\overline\bv,\overline\vartheta)&=\frac{1}{c_vJ_{\eta_0}\overline\rho_0}\bigg[c_v(J_{\eta_0}\overline\rho_0-J_\eta \overline\rho)\partial_t\overline\vartheta-c_v\overline\rho(\overline\bv\cdot\bfB_\eta\nabla)\overline\vartheta -J_\eta\overline{\rho}\,\nabla\overline{\vartheta}\cdot\partial_t\bfPsi_\eta^{-1}\circ\bfPsi_\eta \\& \quad \quad  -\kappa\Div\big((\bfA_{\eta_0}-\bfA_\eta) \nabla\overline\vartheta\;\!\big) +J_\eta^{-1}\mathbb S(\bfB_\eta\nabla\overline\bfv):\bfB_\eta\nabx\overline\bv- p(\overline\rho,\overline\vartheta)\bfB_\eta :\nabla\overline\bfv\bigg]\\
&\quad +\kappa c_v\nabla(J_{\eta_0}\overline\rho_0)^{-1}\cdot\bfA_{\eta_0} \nabla\overline\vartheta,
\\[0.4em]
\mathfrak h_\eta(\overline\rho,\overline\vartheta)&=\frac{\kappa }{c_vJ_{\eta_0}\overline\rho_0}(\bfA_{\eta_0}-\bfA_\eta)\nabla\overline\vartheta\cdot \bfn\circ\bfvarphi^{-1}.
\end{align*}
Similar to the analysis carried out in \cref{sec:MomStrucSubProb}, the well-posedness of \eqref{eq:31.03} is obtained through a linearisation and fixed-point procedure.  Suppose that the data $\left(\overline{\rho}, \overline{\bv},  \eta, \overline\vartheta_0\right)$ satisfies
\begin{align}
	&\overline{\rho} \in L^\infty\big(I;W^{3,2}(\Omega)\big)\cap W^{1,\infty}\big( I;W^{2,2}(\Omega)\big), \qquad \overline{\bv}   \in L^2\big(I;W^{4,2}(\Omega)\big)\cap W^{2,2}\big( I;L^{2}(\Omega)\big),
	\label{eq:FixDenVel} 
	\\[0.4em]
&\eta \in L^2\big(I; W^{6,2}(\omega)\big)  
\cap W^{3,2}\big(I ; L^{2}(\omega)\big)
  \cap W^{2,\infty}\big(I; W^{1,2}(\omega)\big),	 \label{eq:FixEta}
  \\[0.4em]
&\partial_t\eta \in L^\infty\big(I; W^{3,2}(\omega)\big) \cap L^{2}\big(I; W^{4,2}(\omega)\big), \quad \overline\vartheta_0 \in W^{3,2}(\Omega).  \label{eq:FixEtaT-Theta0}
\end{align}
Let the functions $(\mathfrak g, \mathfrak h)$ satisfy
\begin{align}
	&\mathfrak g \in L^2(I;W^{2,2}(\Omega))\cap W^{1,2}(I;L^{2}(\Omega)), \nonumber
	\\&
	\mathfrak h \in L^2(I;W^{5/2,2}(\partial\Omega))\cap W^{5/4, 2}(I;L^{2}(\partial\Omega)),  
	\qquad \mathfrak h (0) = 0,   \label{sourcecond-InternalE}
\end{align}
where $\mathfrak h  (0)$ is  the initial value of $\mathfrak h$.  We first linearise  \eqref{eq:31.03}, obtaining 
\begin{equation}\label{eq:31.03A}
\left\{\begin{aligned}
\partial_t\overline\vartheta-
\Div\bigg(\frac{\kappa }{c_vJ_{\eta_0}\overline\rho_0}\bfA_{\eta_0} \nabla\overline\vartheta\bigg) 
&=\mathfrak g     & \text{in } I \times \Omega,
\\[0.4em]
\frac{\kappa }{c_vJ_{\eta_0}\overline\rho_0}\bfA_{\eta_0}\nabla\overline\vartheta\cdot \bfn\circ\bfvarphi^{-1}&=\mathfrak h & \text{ on }  I\times \Omega,
\\[0.4em]
\overline\vartheta(0,\cdot)&=\overline\vartheta_0(\cdot) & \text{in }\Omega ,
\end{aligned}\right.
\end{equation} 
whose solvability and regularity properties are summarised in the following theorem.  
\begin{theorem}\label{prop:InternalEnerLin} 
Let the data $\left(\overline{\rho}, \overline{\bv},  \eta, \overline\vartheta_0\right)$  satisfy \eqref{eq:FixDenVel}--\eqref{eq:FixEtaT-Theta0} and let $(\mathfrak g, \mathfrak h)$ satisfy \eqref{sourcecond-InternalE}. Assume furthermore that the following compatibility conditions hold 
\begin{equation}\label{eq:NewCC} \tag{$\mathtt{C}$}
\left\{\begin{aligned} 
     \frac{\kappa }{c_vJ_{\eta_0}\overline\rho_0}\bfA_{\eta_0}\nabla\overline\vartheta_0\cdot \bfn\circ\bfvarphi^{-1}&=\mathfrak h (0)  ,
          \\[0.4em]
      \frac{\kappa }{c_vJ_{\eta_0}\overline\rho_0}\bfA_{\eta_0} \nabla \Biggl[   \Div\bigg(\frac{\kappa }{c_vJ_{\eta_0}\overline\rho_0}\bfA_{\eta_0} \nabla\overline\vartheta_0\bigg)        +   \mathfrak g (0)  \Biggr] \cdot \bfn\circ\bfvarphi^{-1}  & = \partial_t   \mathfrak h (0)  .
\end{aligned}  \right.   \qquad \qquad   \text{on }\omega,
\end{equation}
Then there exists a unique strong solution $\overline{\vartheta}$  of \eqref{eq:31.03A} satisfying 
\begin{align}\label{eq:31.03B}
\begin{aligned}
\int_I\Vert \partial_t^2\overline\vartheta\Vert_{L^2(\Omega)}^2 \dt+\int_I\|\overline\vartheta\|_{W^{4,2}(\Omega)}^2\dt&\lesssim\|\overline\vartheta_0\|_{W^{3,2}_x}^2 +  \|\mathfrak g\|_{W^{1,2}\left( I; L^2(\Omega)\right)}^2  +\|\mathfrak g\|_{L^{2}\left( I; W^{2,2}(\Omega)\right)}^2
\\[0.4em]
& \quad + \|\mathfrak h\|_{L^{2}\left( I; W^{5/2,2}(\partial\Omega) \right)}^2 + \|\mathfrak h\|_{W^{5/4, 2}\left( I; L^{2}(\partial\Omega) \right)}^2,
\end{aligned}
\end{align}
where the constant
depends on the $L^\infty\left( I;  W^{3,2}(\Omega) \right)-$norm of the coefficients
$\frac{\kappa}{c_vJ_{\eta_0}\overline\rho_0}\bfA_{\eta_0} $.

\end{theorem}

\begin{proof}
The result is a direct consequence   of Solonnikov's maximal regularity theory  for linear parabolic boundary value problems (see \cite[Chapter 5, Section 21]{solonnikov1965}).  Indeed,  the principal symbol of the spatial differential operator is given by 
\[
\ \mathfrak{a}_{\eta_0}\big(\bfA_{\eta_0}\bm{\xi} \big) \cdot \bm{\xi} ,   \qquad \mathfrak{a}_{\eta_0}:=\frac{\kappa}{c_vJ_{\eta_0}\overline\rho_0}, \qquad   \forall\, \bm{\xi} \in \R^3, 
\]
which is uniformly elliptic since $ \mathfrak{a}_{\eta_0} > 0$ and $\bfA_{\eta_0}$ is positive definite.  Moreover, the Lopatinskii--Shapiro condition is satisfied.
\end{proof}
We  now proceed to the fixed-point construction for the nonlinear problem \eqref{eq:31.03}. To this end, we  first introduce the functional framework and notation used in the contraction argument.  We set 
\[\mathbb{W}^{s, m}_{q,p} (I\times \Omega) :=  W^{s, q}\big( I; W^{m, p}(\Omega) \big) \quad \forall\, s, m \geq 0, \;\; \forall\,  q, p \in  (0, \infty],   \]
and introduce the following spaces: 
 \begin{equation*}
\begin{aligned}
  \mathcal{X}_{\overline\bv} & :=    \mathbb{W}^{0, 4}_{2,2} (I\times \Omega)  \cap \mathbb{W}^{1, 2}_{2,2} (I\times \Omega)  \cap \mathbb{W}^{2, 0}_{2,2} (I\times \Omega) \cap \mathbb{W}^{1, 1}_{\infty,2} (I\times \Omega) , &
\\[0.4em]
  \mathcal{X}_{\eta} &:=  \mathbb{W}^{1, 4}_{2,2} (I\times \omega)  \cap \mathbb{W}^{1, 3}_{\infty, 2} (I\times \omega) \cap \mathbb{W}^{2, 1}_{\infty,2} (I\times \omega) \cap \mathbb{W}^{2, 2}_{2,2} (I\times \omega) \cap \mathbb{W}^{3, 0}_{2,2} (I\times \omega)  \cap \mathbb{W}^{0, 6}_{2,2} (I\times \omega), &
  \\[0.4em]
   \mathcal{Z} & :=   \left(  \mathbb{W}^{0, 2}_{2,2} (I\times \Omega)  \cap \mathbb{W}^{1, 0}_{2, 2} (I\times \Omega) \right) \times   \left( \mathbb{W}^{0, 5/2}_{2,2} (I\times \partial\Omega)  \cap \mathbb{W}^{5/4, 0 }_{2, 2} (I\times \partial\Omega)  \right)    , &
\end{aligned}
\end{equation*}
endowed with the natural norms.
Furthermore, for arbitrary  $R > 0,$ we define the  closed subset $ \mathcal{Z}_{R} \subset \mathcal{Z} $ by
\begin{align*}
\mathcal{Z}_{R}  :=  \Bigg\{ & (\mathfrak g, \mathfrak h)  \in \mathcal{Z} \,\Big| \,   \mathfrak g (0) = \mathfrak{g}_{\eta_0}(\overline{\rho}_0, \overline{\bv}_0, \overline{\vartheta}_0)  , \;  \mathfrak h (0) =  0,  \; \eqref{eq:NewCC}  \text{ holds},
 \; \text{and} \;\;  \Vert (\mathfrak g, \mathfrak h) \Vert_{\raisebox{-1ex}{$\mathsmaller{ \mathcal{Z} }$}}   \leq R  \Bigg\}.  
\end{align*}
In particular, $\mathcal{Z}_{R} \neq \emptyset.$ Indeed, after possibly increasing the radius $R >0$, and choosing $T >0$ sufficiently small, 
it holds that  $\big( \mathfrak{g}_{\eta_0}(\overline{\rho}_0, \overline{\bv}_0, \overline{\vartheta}_0) ,   t \mathfrak{h}_{0}\big) \in  \mathcal{Z}_{R}, $ where 
\[
\mathfrak{h}_{0} := \frac{\kappa }{c_vJ_{\eta_0}\overline\rho_0}\bfA_{\eta_0} \nabla \Biggl[   \Div\bigg(\frac{\kappa }{c_vJ_{\eta_0}\overline\rho_0}\bfA_{\eta_0} \nabla\overline\vartheta_0\bigg)        +   \mathfrak g (0)  \Biggr] \cdot \bfn\circ\bfvarphi^{-1} .
\]
We shall now show that, for $T > 0$ sufficiently small, the map 
\[ \mathcal{T} \colon \mathcal{Z}_{R} \to  \mathcal{Z}_{ R},  \quad(\mathfrak g, \mathfrak h)  \mapsto \Big( \mathfrak{g}_\eta (\overline\rho, \overline{\bv}, \overline\vartheta), \, \mathfrak{h}_\eta (\overline\rho, \overline{\bv}, \overline\vartheta) \Big),   \] 
 is a strict contraction, where $\overline{\vartheta} $  is the solution of the linearised problem  \eqref{eq:31.03A}.    For clarity of the argument, we proceed in two steps.

 \medskip
 
 \noindent\textbf{Step 1: }  $\mathcal{T}-$invariance of the set $\mathcal{Z}_{R}. $ \\  
Let $c>0$ be arbitrary but fixed, and $R >0$ be sufficiently large  such that 
\begin{equation}\label{eq:Step1-0}
c\left( \Vert \overline{\vartheta}_0 \Vert_{W^{3,2}(\Omega)}  +  \Big(1+\Vert \eta\Vert_{ \raisebox{-1ex}{$\mathsmaller{ \mathcal{X} }$}_{\eta} }^3 \Big) \Vert \overline\bv\Vert_{ \raisebox{-1ex}{$\mathsmaller{ \mathcal{X} }$}_{\overline\bv} }^2  \right) \leq \dfrac{R}{2} .
\end{equation}
It is straightforward that  
\begin{align}
\Vert \mathfrak{h}_\eta \Vert_{ \mathbb{W}^{0, 5/2}_{2,2} (I\times \Omega)} & \lesssim  \Vert \bfA_{\eta_0} - \bfA_{\eta} \Vert_{\mathbb{W}^{0, 5/2}_{\infty,2} (I\times \partial\Omega)}  \Vert \overline\vartheta \Vert_{ \mathbb{W}^{0, 7/2}_{2,2} (I\times \partial\Omega)} \nonumber
\\[0.4em]
& \lesssim T^{1/2} \Vert \eta \Vert_{ \mathbb{W}^{1, 4}_{2,2} (I\times \omega)} \Vert \overline\vartheta \Vert_{ \mathbb{W}^{0, 4}_{2,2} (I\times \Omega)}  \lesssim T^{1/2}  \Vert \overline\vartheta \Vert_{ \mathbb{W}^{0, 4}_{2,2} (I\times \Omega)}.  \label{eq:Step1-1}
\end{align}
Furthermore, by the fractional Leibniz rule and the multiplier property 
\begin{equation}\label{eq:Multipl.Prop}
W^{3/2,2}(\partial\Omega) \cdot L^{ 2}(\partial\Omega)  \hookrightarrow  L^{2}(\partial\Omega)  ,
\end{equation}
we deduce that 
\begin{align}
\Vert  \mathfrak{h}_\eta \Vert_{ \mathbb{W}^{5/4, 0}_{2,2} (I\times \partial\Omega)} & \lesssim  \Vert \bfA_{\eta_0} - \bfA_{\eta} \Vert_{\mathbb{W}^{0, 3/2}_{\infty, 2} (I\times \partial\Omega)}  \Vert \nabla\overline\vartheta \Vert_{\mathbb{W}^{5/4, 0}_{2,2} (I\times \partial\Omega)}    \nonumber
\\[0.4em]
& \quad  +  \Vert \bfA_{\eta_0} - \bfA_{\eta} \Vert_{\mathbb{W}^{5/4, 3/2}_{2, 2} (I\times \partial\Omega)} \Big( \Vert \nabla\overline\vartheta -  \nabla\overline\vartheta_0  \Vert_{\mathbb{W}^{0, 0}_{\infty,2} (I\times \partial\Omega)}  +  \Vert  \nabla\overline\vartheta_0  \Vert_{L^{2} ( \partial\Omega)}    \Big)    \nonumber
\\[0.4em]
& \lesssim  \max\{ T, T^{1/2}\}  \Big( \Vert \partial_t \nabla\overline\vartheta \Vert_{ \mathbb{W}^{0, 0}_{2,2} (I\times \partial\Omega)} +\Vert \nabla\overline\vartheta \Vert_{ \mathbb{W}^{5/4, 0}_{2,2} (I\times \partial\Omega)} \Big) + \Vert \overline{\vartheta}_0 \Vert_{W^{3,2}(\Omega)} .  \label{eq:Step1-2}
\end{align}
Moreover, we have  
\begin{align*}
\Vert \mathfrak{g}_\eta \Vert_{ \mathbb{W}^{0, 2}_{2,2} (I\times \omega)} &   \lesssim  \Big( \Vert \eta - \eta_0 \Vert_{ \mathbb{W}^{0, 3}_{\infty,2} (I\times \omega) }   + \Vert \overline{\rho}_0 - \overline\rho \Vert_{ \mathbb{W}^{0, 2}_{\infty,2} (I\times \Omega)} \Big)  \Vert \partial_t \overline\vartheta \Vert_{ \mathbb{W}^{0, 2}_{2,2} (I\times \Omega)}   + T^{1/2}\Vert \overline{\vartheta} \Vert_{ \mathbb{W}^{0, 3}_{\infty,2} (I\times \Omega)}
\\
&\quad +  \Vert  \eta - \eta_0 \Vert_{ \mathbb{W}^{0, 4}_{\infty,2} (I\times \omega) }  \Vert \overline{\vartheta} \Vert_{ \mathbb{W}^{0, 4}_{2,2} (I\times \Omega)}    +\Big(1+ \Vert \eta \Vert_{ \mathbb{W}^{\infty, 3}_{2,2} (I\times \Omega) }^2\Big) \Vert \overline{\bv} \Vert_{ \mathbb{W}^{0, 3}_{2,2} (I\times \Omega) }^2\\& + T^{1/2}\Vert \partial_t \overline{\vartheta} \Vert_{ \mathbb{W}^{0, 2}_{2, 2} (I\times \Omega)}.
\end{align*}
Hence,
\begin{align}
\Vert \mathfrak{g}_\eta \Vert_{ \mathbb{W}^{0, 2}_{2,2} (I\times \omega)} & \lesssim   T^{1/2} \Vert \overline\vartheta \Vert_{ \raisebox{-1ex}{$\mathsmaller{ \mathcal{X} }$}_{\overline\bv} }   +  \Big(1+\Vert \eta\Vert_{ \raisebox{-1ex}{$\mathsmaller{ \mathcal{X} }$}_{\eta} }^3\Big)\Vert \overline\bv\Vert_{ \raisebox{-1ex}{$\mathsmaller{ \mathcal{X} }$}_{\overline\bv} }^2  + \Vert \overline{\vartheta}_0 \Vert_{W^{3,2}(\Omega)}.    \label{eq:Step1-3}
\\[0.25cm]
\intertext{Similar, it holds that }
\Vert \partial_t \mathfrak{g}_\eta \Vert_{ \mathbb{W}^{0, 0}_{2,2} (I\times \omega)} & \lesssim  T^{1/2} \Vert \overline\vartheta \Vert_{ \raisebox{-1ex}{$\mathsmaller{ \mathcal{X} }$}_{\overline\bv} }   + \Vert \overline\bv\Vert_{ \raisebox{-1ex}{$\mathsmaller{  \mathcal{X} }$}_{\overline\bv} }^2  + \Vert \overline{\vartheta}_0 \Vert_{W^{3,2}(\Omega)}   \label{eq:Step1-4}. 
\end{align}  
Therefore, combining \eqref{eq:Step1-0}--\eqref{eq:Step1-4}, and by \eqref{eq:31.03B},  we conclude that  
\[
 \Vert \mathcal{T} (\mathfrak g, \mathfrak h) \Vert_{\raisebox{-1ex}{$\mathsmaller{\mathcal{Z} }$}} = \Vert (\mathfrak g_{\eta}, \mathfrak h_\eta) \Vert_{\raisebox{-1ex}{$\mathsmaller{ \mathcal{Z} }$}}  \leq R 
\]
for  sufficiently small $T >0$.\\

\medskip

\noindent\textbf{Step 2: }  $\mathcal{T} $ is a strict contraction. \\
Let  $\overline\vartheta_i,  \; i \in \{1, 2\}, $ denote the corresponding   solutions to the linearised system  \eqref{eq:31.03A},   with source terms $ \left( \mathfrak{g}_{i},  \mathfrak{h}_{i} \right)  \in \mathcal{Z}_{ R}, $ and data $\left(\overline{\rho}_i, \overline{\bv}_i,  \eta_i, \overline\vartheta_0\right).$

In order to obtain a contraction we need to estimate
$\mathfrak g_{\eta_1}(\overline\rho_1,\overline\bv_1,\overline\vartheta_1)-\mathfrak g_{\eta_2}(\overline\rho_2,\overline\bv_2,\overline\vartheta_2)$. The term 
$\frac{1}{c_vJ_{\eta_0}\overline\rho_0}$ can be ignored since it belongs to $W^{2,2}(\Omega)$, is time-independent and bounded from above and below. So, we must estimate
\begin{align*}
\int_I&\|J_{\eta_1}^{-1}\mathbb S(\bfB_{\eta_1}\nabla\overline\bfv_1):\bfB_{\eta_1}\nabx\overline\bv_1-J_{\eta_2}^{-1}\mathbb S(\bfB_{\eta_2}\nabla\overline\bfv_2):\bfB_{\eta_2}\nabx\overline\bv_2\|_{W^{2,2}(\Omega)}^2\dt\\[0.4em]
&\lesssim \int_I\|(J_{\eta_1}^{-1}-J_{\eta_2}^{-1})\mathbb S(\bfB_{\eta_1}\nabla\overline\bfv_1):\bfB_{\eta_1}\nabx\overline\bv_1\|_{W^{2,2}(\Omega)}^2\dt\\[0.4em]
& \quad + \int_I\|J_{\eta_2}^{-1}(\mathbb S(\bfB_{\eta_1}-\bfB_{\eta_2})\nabla\overline\bfv_1):\bfB_{\eta_2}\nabx\overline\bv_1\|_{W^{2,2}(\Omega)}^2\dt\\[0.4em]
& \quad + \int_I\|J_{\eta_2}^{-1}\mathbb S(\bfB_{\eta_2}\nabla(\overline\bfv_1-\overline\bfv_2)):\bfB_{\eta_1}\nabx\overline\bv_1\|_{W^{2,2}(\Omega)}^2\dt\\[0.4em]
& \quad + \int_I\|J_{\eta_2}^{-1}\mathbb S(\bfB_{\eta_2}\nabla\overline\bfv_2):(\bfB_{\eta_1}-\bfB_{\eta_2})\nabx\overline\bv_1\|_{W^{2,2}(\Omega)}^2\dt\\[0.4em]
& \quad + \int_I\|J_{\eta_2}^{-1}\mathbb S(\bfB_{\eta_2}\nabla\overline\bfv_2):\bfB_{\eta_2}\nabx(\overline\bv_1-\overline\bfv_2)\|_{W^{2,2}(\Omega)}^2\dt.
\end{align*}
Note that $\bfB_{\eta}$ and $J_\eta$ are multipliers on $W^{2,2}(\Omega)$  as  $\eta\in \mathbb{W}^{0, 3}_{\infty,2}(I\times\omega)$. Similarly $\nabla\bfv\in  \mathbb{W}^{0, 2}_{\infty,2}(I\times\Omega)$ is a multiplier. Hence we have
\begin{align*}
\int_I&\|J_{\eta_1}^{-1}\mathbb S(\bfB_{\eta_1}\nabla\overline\bfv_1):\bfB_{\eta_1}\nabx\overline\bv_1-J_{\eta_2}^{-1}\mathbb S(\bfB_{\eta_2}\nabla\overline\bfv_2):\bfB_{\eta_2}\nabx\overline\bv_2\|_{W^{2,2}(\Omega)}^2\dt\\[0.4em]
&\lesssim \int_I\|J_{\eta_1}^{-1}-J_{\eta_2}^{-1}\|_{W^{2,2}(\Omega)}^2\dt+ \int_I\|\bfB_{\eta_1}-\bfB_{\eta_2}\|_{W^{2,2}(\Omega)}^2\dt+ \int_I\|\nabx(\overline\bv_1-\overline\bfv_2)\|_{W^{2,2}(\Omega)}^2\dt\\[0.4em]
&\lesssim  \int_I\|\eta_1-\eta_2\|_{W^{3,2}(\omega)}^2\dt+ \int_I\|\overline\bv_1-\overline\bfv_2\|_{W^{3,2}(\Omega)}^2\dt\\
&\lesssim T \bigg( \sup_I\|\eta_1-\eta_2\|_{W^{3,2}(\omega)}^2+ \sup_I\|\overline\bv_1-\overline\bfv_2\|_{W^{3,2}(\Omega)}^2\bigg).
\end{align*}
Similarly,
\begin{align*} 
\int_I&\|p(\overline\rho_1,\overline\vartheta_1)\bfB_{\eta_1} :\nabla\overline\bfv_1-p(\overline\rho_2,\overline\vartheta_2)\bfB_{\eta_2} :\nabla\overline\bfv_2\|_{W^{2,2}(\Omega)}^2\dt\\
&\lesssim T \sup_I \Big( \|\eta_1-\eta_2\|_{W^{3,2}(\omega)}^2+  \|\overline\bv_1-\overline\bfv_2\|_{W^{3,2}(\Omega)}^2+ \|\overline\rho_1-\overline\rho_2\|_{W^{2,2}(\Omega)}^2+ \|\overline\vartheta_1-\overline\vartheta_2\|_{W^{2,2}(\Omega)}^2\Big),
\end{align*}
\begin{align*}
\int_I&\|\overline\rho(\overline\bv\cdot\bfB_\eta\nabla)\overline\vartheta-\overline\rho(\overline\bv\cdot\bfB_\eta\nabla)\overline\vartheta\|_{W^{2,2}(\Omega)}^2\dt
\\
&\lesssim T \sup_I \Big( \|\eta_1-\eta_2\|_{W^{3,2}(\omega)}^2+ \|\overline\rho_1-\overline\rho_2\|_{W^{2,2}(\Omega)}^2+ \|\overline\vartheta_1-\overline\vartheta_2\|_{W^{3,2}(\Omega)}^2\Big),
\end{align*}
and
\begin{align*}
\int_I&\|
\nabla(J_{\eta_0}\overline\rho_0)^{-1}\cdot\bfA_{\eta_0} \nabla\overline\vartheta_1-\nabla(J_{\eta_0}\overline\rho_0)^{-1}\cdot\bfA_{\eta_0} \nabla\overline\vartheta_2\|_{W^{2,2}(\Omega)}\dt\\
&\lesssim T \sup_I\|\overline\vartheta_1-\overline\vartheta_2\|_{W^{3,2}(\Omega)}^2
\end{align*}
since $\eta_0\in W^{4,2}(\omega)$ and $\overline\rho_0\in W^{3,2}(\Omega)$. Furthermore, we have 
\begin{align*}
\int_I&\|
J_{\eta_1}\overline{\rho}_1\,\nabla\overline{\vartheta}_1\cdot\partial_t\bfPsi_{\eta_1}^{-1}\circ\bfPsi_{\eta_1} -J_{\eta_2}\overline{\rho}_2\,\nabla\overline{\vartheta}_2\cdot\partial_t\bfPsi_{\eta_2}^{-1}\circ\bfPsi_{\eta_2}\|_{W^{2,2}(\Omega)}^2\dt
\\
&\lesssim T \sup_I \Big( \|\eta_1-\eta_2\|_{W^{2,2}(\omega)}^2 + \|\partial_t(\eta_1-\eta_2)\|_{W^{2,2}(\omega)}^2 + \|\overline\rho_1-\overline\rho_2\|_{W^{2,2}(\Omega)}^2 +  \|\overline\vartheta_1-\overline\vartheta_2\|_{W^{3,2}(\Omega)}^2\Big),
\end{align*}
\begin{align*}
\int_I&\|(J_{\eta_0}\overline\rho_0-J_{\eta_1} \overline\rho_1)\partial_t\overline\vartheta_1-(J_{\eta_0}\overline\rho_0-J_{\eta_2} \overline\rho_2)\partial_t\overline\vartheta_2\|_{W^{2,2}(\Omega)}^2\dt\\
&\lesssim\int_I\|(J_{\eta_0}\overline\rho_0-J_{\eta_1} \overline\rho_1)\partial_t(\overline\vartheta_1-\overline\vartheta_2)\|_{W^{2,2}(\Omega)}^2\dt+\int_I\|(J_{\eta_1}\overline\rho_1-J_{\eta_2} \overline\rho_2)\partial_t\overline\vartheta_2\|_{W^{2,2}(\Omega)}^2\dt\\
&\lesssim\sup_I \!\|J_{\eta_0}\overline\rho_0-J_{\eta_1}\overline\rho_1\|_{W^{2,2}(\Omega)}^2\! \!\int_I \! \|
\partial_t\overline\vartheta_1-\partial_t\overline\vartheta_2\|_{W^{2,2}(\Omega)}^2\dt +\sup_I\!\|J_{\eta_1}\overline\rho_1-J_{\eta_2}\overline\rho_2\|_{W^{2,2}(\Omega)}^2\! \!\int_I \!\|
 \partial_t\overline\vartheta_2\|_{W^{2,2}(\Omega)}^2\dt\\
&\lesssim T\|(\overline\vartheta_1,\eta_1)-(\overline\vartheta_2,\eta_2)\|_{ \raisebox{-1ex}{$\mathsmaller{ \mathcal{X} } $}_{\overline\bv}  \raisebox{-1ex}{$\mathsmaller{ \times \mathcal{X}  }$}_{\eta} }
\end{align*}
using that $\eta, \overline\rho\in \mathbb{W}^{1,2}_{2,2}(I\times \Omega))$. Finally, we obtain
\begin{align*}
\int_I&\|
\kappa\Div((\bfA_{\eta_0}-\bfA_{\eta_1}) \nabla\overline\vartheta_1)-\kappa\Div((\bfA_{\eta_0}-\bfA_{\eta_2}) \nabla\overline\vartheta_2)\|_{W^{2,2}(\Omega)}^2\dt\\
&\lesssim\int_I\|
(\bfA_{\eta_0}-\bfA_{\eta_1}) \nabla\overline\vartheta_1-(\bfA_{\eta_0}-\bfA_{\eta_2}) \nabla\overline\vartheta_2\|_{W^{3,2}(\Omega)}^2\dt\\
&\lesssim\int_I\|
(\bfA_{\eta_0}-\bfA_{\eta_1}) (\nabla\overline\vartheta_1- \nabla\overline\vartheta_2)\|_{W^{3,2}(\Omega)}^2\dt+\int_I\|
(\bfA_{\eta_1}-\bfA_{\eta_2}) \nabla\overline\vartheta_2\|_{W^{3,2}(\Omega)}^2\dt\\
&\lesssim\sup_I\|\eta_0-\eta_1\|_{W^{4,2}(\omega)}^2\int_I\|
\overline\vartheta_1-\overline\vartheta_2\|_{W^{4,2}(\Omega)}^2\dt+\sup_I\|\eta_1-\eta_2\|_{W^{4,2}(\omega)}^2\int_I\|
\overline\vartheta_2\|_{W^{4,2}(\Omega)}^2\dt\\
&\lesssim T \|(\overline\vartheta_1,\eta_1)-(\overline\vartheta_2,\eta_2)\|_{ \raisebox{-1ex}{$\mathsmaller{ \mathcal{X} } $}_{\overline\bv}  \raisebox{-1ex}{$\mathsmaller{ \times \mathcal{X}  }$}_{\eta} }
\end{align*}
using that $\eta\in \mathbb{W}^{1,4}_{2,2}(I\times \omega))$.

To complete the  contraction argument, we proceed by  estimating
$\partial_t\big(\mathfrak g_{\eta_1}(\overline\varrho,\overline\bv,\overline\vartheta)-\mathfrak g_{\eta_2}(\overline\varrho_2,\overline\bv_2,\overline\vartheta_2)\big)$. We first consider
\begin{align*}
\int_I&\|\partial_t(J_{\eta_1}^{-1}\mathbb S(\bfB_{\eta_1}\nabla\overline\bfv_1):\bfB_{\eta_1}\nabx\overline\bv_1-J_{\eta_2}^{-1}\mathbb S(\bfB_{\eta_2}\nabla\overline\bfv_2):\bfB_{\eta_2}\nabx\overline\bv_2)\|_{L^{2}(\Omega)}^2\dt\\
&\lesssim \int_I\|\partial_t(J_{\eta_1}^{-1}-J_{\eta_2}^{-1})\mathbb S(\bfB_{\eta_1}\nabla\overline\bfv_1):\bfB_{\eta_1}\nabx\overline\bv_1\|_{L^{2}(\Omega)}^2\dt\\
&\quad + \int_I\|(J_{\eta_1}^{-1}-J_{\eta_2}^{-1})\partial_t(\mathbb S(\bfB_{\eta_1}\nabla\overline\bfv_1):\bfB_{\eta_1}\nabx\overline\bv_1)\|_{L^{2}(\Omega)}^2\dt\\
&\quad + \int_I\|(\mathbb S(\partial_t(\bfB_{\eta_1}-\bfB_{\eta_2}))\nabla\overline\bfv_1):J_{\eta_2}^{-1}\bfB_{\eta_2}\nabx\overline\bv_1\|_{L^{2}(\Omega)}^2\dt\\
&\quad + \int_I\|(\mathbb S((\bfB_{\eta_1}-\bfB_{\eta_2}))\nabla\partial_t\overline\bfv_1):J_{\eta_2}^{-1}\bfB_{\eta_2}\nabx\overline\bv_1\|_{L^{2}(\Omega)}^2\dt\\
&\quad + \int_I\|(\mathbb S((\bfB_{\eta_1}-\bfB_{\eta_2}))\nabla\overline\bfv_1):\partial_t(J_{\eta_2}^{-1}\bfB_{\eta_2}\nabx\overline\bv_1)\|_{L^{2}(\Omega)}^2\dt\\
&\quad + \int_I\|\mathbb S(\bfB_{\eta_2}\nabla\partial_t(\overline\bfv_1-\overline\bfv_2)):J_{\eta_2}^{-1}\bfB_{\eta_1}\nabx\overline\bv_1\|_{L^{2}(\Omega)}^2\dt\\
&\quad + \int_I\|\mathbb S(\partial_t\bfB_{\eta_2}\nabla(\overline\bfv_1-\overline\bfv_2)):J_{\eta_2}^{-1}\bfB_{\eta_1}\nabx\overline\bv_1\|_{L^{2}(\Omega)}^2\dt\\
&\quad + \int_I\|\mathbb S(\bfB_{\eta_2}\nabla(\overline\bfv_1-\overline\bfv_2)):\partial_t(J_{\eta_2}^{-1}\bfB_{\eta_1}\nabx\overline\bv_1)\|_{L^{2}(\Omega)}^2\dt\\
&\quad + \int_I\|J_{\eta_2}^{-1}\mathbb S(\bfB_{\eta_2}\nabla\overline\bfv_2):\partial_t(\bfB_{\eta_1}-\bfB_{\eta_2})\nabx\overline\bv_1\|_{L^{2}(\Omega)}^2\dt\\
& \quad + \int_I\|\partial_t(J_{\eta_2}^{-1}\mathbb S(\bfB_{\eta_2}\nabla\overline\bfv_2)):(\bfB_{\eta_1}-\bfB_{\eta_2})\nabx\overline\bv_1\|_{L^{2}(\Omega)}^2\dt\\
& \quad + \int_I\|J_{\eta_2}^{-1}\mathbb S(\bfB_{\eta_2}\nabla\overline\bfv_2):(\bfB_{\eta_1}-\bfB_{\eta_2})\nabx\partial_t\overline\bv_1\|_{L^{2}(\Omega)}^2\dt\\
&\quad + \int_I\|J_{\eta_2}^{-1}\mathbb S(\bfB_{\eta_2}\nabla\overline\bfv_2):\bfB_{\eta_2}\nabx\partial_t(\overline\bv_1-\overline\bfv_2)\|_{L^{2}(\Omega)}^2\dt\\
& \quad + \int_I\|\partial_t(J_{\eta_2}^{-1}\mathbb S(\bfB_{\eta_2}\nabla\overline\bfv_2)):\bfB_{\eta_2}\nabx(\overline\bv_1-\overline\bfv_2)\|_{L^{2}(\Omega)}^2\dt\\
&\quad + \int_I\|J_{\eta_2}^{-1}\mathbb S(\bfB_{\eta_2}\nabla\overline\bfv_2):\partial_t\bfB_{\eta_2}\nabx(\overline\bv_1-\overline\bfv_2)\|_{L^{2}(\Omega)}^2\dt .
\end{align*}
Since $\partial_t\eta\in \mathbb{W}^{0,3}_{\infty, 2}(I\times \omega)$ and $\bfv\in \mathbb{W}^{1,1}_{\infty, 2}(I\times \Omega)\cap \mathbb{W}^{0,3}_{\infty, 2}(I\times \Omega)$ we have
\begin{align*}
\int_I&\|\partial_t(J_{\eta_1}^{-1}\mathbb S(\bfB_{\eta_1}\nabla\overline\bfv_1):\bfB_{\eta_1}\nabx\overline\bv_1-J_{\eta_2}^{-1}\mathbb S(\bfB_{\eta_2}\nabla\overline\bfv_2):\bfB_{\eta_2}\nabx\overline\bv_2)\|_{L^{2}(\Omega)}^2\dt\\
&\lesssim T \sup_I \Big( \|\partial_t(\eta_1-\eta_2)\|_{W^{1,2}(\omega)}^2+ \|\eta_1-\eta_2\|_{W^{3,2}(\omega)}^2  +  \|\partial_t(\overline\bv_1-\overline\bfv_2)\|_{W^{1,2}(\Omega)}^2+ \|\overline\bv_1-\overline\bfv_2\|_{W^{3,2}(\Omega)}^2\Big)
\end{align*}
Similarly,

\begin{align*}
\int_I&\|\partial_t(p(\overline\rho_1,\overline\vartheta_1)\bfB_{\eta_1} :\nabla\overline\bfv_1-p(\overline\rho_2,\overline\vartheta_2)\bfB_{\eta_2} :\nabla\overline\bfv_2)\|_{L^{2}(\Omega)}^2\dt\\
&\lesssim T \sup_I \Big( \|\partial_t(\eta_1-\eta_2)\|_{W^{1,2}(\omega)}^2+ \|\partial_t(\overline\bv_1-\overline\bfv_2)\|_{W^{1,2}(\Omega)}^2+ \|\partial_t(\overline\rho_1-\overline\rho_2)\|_{L^{2}(\Omega)}^2+ \|\partial_t(\overline\vartheta_1-\overline\vartheta_2)\|_{L^{2}(\Omega)}^2\Big)\\
&+ T \sup_I \Big( \|\eta_1-\eta_2\|_{W^{3,2}(\omega)}^2+ \|\overline\bv_1-\overline\bfv_2\|_{W^{3,2}(\Omega)}^2+ \|\overline\rho_1-\overline\rho_2\|_{W^{2,2}(\Omega)}^2+ \|\overline\vartheta_1-\overline\vartheta_2\|_{W^{2,2}(\Omega)}^2\Big),
\end{align*}

\begin{align*}
\int_I&\|\partial_t(\overline\rho(\overline\bv\cdot\bfB_\eta\nabla)\overline\vartheta-\overline\rho(\overline\bv\cdot\bfB_\eta\nabla)\overline\vartheta)\|_{L^{2}(\Omega)}^2\dt
\\
&\lesssim T \sup_I \Big( \|\partial_t(\eta_1-\eta_2)\|_{W^{1,2}(\omega)}^2+ \|\partial_t(\overline\rho_1-\overline\rho_2)\|_{L^{2}(\Omega)}^2+ \|\partial_t\overline\vartheta_1-\overline\vartheta_2\|_{W^{1,2}(\Omega)}^2\Big)\\
&+ T \sup_I \Big( \|\eta_1-\eta_2\|_{W^{3,2}(\omega)}^2+ \|\overline\bv_1-\overline\bfv_2\|_{W^{2,2}(\Omega)}^2+ \|\overline\rho_1-\overline\rho_2\|_{W^{2,2}(\Omega)}^2+ \|\overline\vartheta_1-\overline\vartheta_2\|_{W^{3,2}(\Omega)}^2\Big),
\end{align*}

\begin{align*}
\int_I&\|
\nabla(J_{\eta_0}\overline\rho_0)^{-1}\cdot\bfA_{\eta_0} \partial_t\nabla\overline\vartheta_1-\nabla(J_{\eta_0}\overline\rho_0)^{-1}\cdot\bfA_{\eta_0} \partial_t\nabla\overline\vartheta_2\|_{W^{2,2}(\Omega)}\dt\\
&\lesssim T\sup_I\|\partial_t(\overline\vartheta_1-\overline\vartheta_2)\|_{W^{1,2}(\Omega)}^2,
\end{align*}

\begin{align*}
\int_I&\|
\partial_t\left(J_{\eta_1}\overline{\rho}_1\,\nabla\overline{\vartheta}_1\cdot\partial_t\bfPsi_{\eta_1}^{-1}\circ\bfPsi_{\eta_1} -J_{\eta_2}\overline{\rho}_2\,\nabla\overline{\vartheta}_2\cdot\partial_t\bfPsi_{\eta_2}^{-1}\circ\bfPsi_{\eta_2} \right)\|_{L^{2}(\Omega)}^2\dt
\\
&\lesssim T \sup_I \Big( \|\partial_t(\eta_1-\eta_2)\|_{L^{2}(\omega)}^2+ \|\partial_t^2(\eta_1-\eta_2)\|_{L^{2}(\omega)}^2+ \|\partial_t(\overline\rho_1-\overline\rho_2)\|_{L^{2}(\Omega)}^2 +  \|\partial_t(\overline\vartheta_1-\overline\vartheta_2)\|_{W^{1,2}(\Omega)}^2\Big),
\end{align*}
\begin{align*}
\int_I&\left\|\partial_t\left((J_{\eta_0}\overline\rho_0-J_{\eta_1} \overline\rho_1)\partial_t\overline\vartheta_1-(J_{\eta_0}\overline\rho_0-J_{\eta_2} \overline\rho_2)\partial_t\overline\vartheta_2 \right) \right\|_{L^{2}(\Omega)}^2\dt\\
&\lesssim\int_I\|(J_{\eta_0}\overline\rho_0-J_{\eta_1} \overline\rho_1)\partial_t^2(\overline\vartheta_1-\overline\vartheta_2)\|_{L^{2}(\Omega)}^2\dt+\int_I\|(J_{\eta_1}\overline\rho_1-J_{\eta_2} \overline\rho_2)\partial_t^2\overline\vartheta_2\|_{L^{2}(\Omega)}^2\dt\\
&\quad +\int_I\|\partial_t(J_{\eta_0}\overline\rho_0-J_{\eta_1} \overline\rho_1)\partial_t(\overline\vartheta_1-\overline\vartheta_2)\|_{L^{2}(\Omega)}^2\dt+\int_I\|\partial_t(J_{\eta_1}\overline\rho_1-J_{\eta_2} \overline\rho_2)\partial_t\overline\vartheta_2\|_{L^{2}(\Omega)}^2\dt\\
&\lesssim\sup_I\|J_{\eta_0}\overline\rho_0-J_{\eta_1}\overline\rho_1\|_{W^{2,2}(\Omega)}^2\int_I\|
\partial_t^2\overline\vartheta_1-\partial_t^2\vartheta_2\|_{L^{2}(\Omega)}^2\dt+\sup_I\|J_{\eta_1}\overline\rho_1-J_{\eta_2}\overline\rho_2\|_{W^{2,2}(\Omega)}^2\int_I\|
 \partial_t^2\overline\vartheta_2\|_{L^{2}(\Omega)}^2\dt\\
&\quad +\sup_I\|\partial_tJ_{\eta_1}\overline\rho_1\|_{W^{2,2}(\Omega)}^2\int_I\|
\partial_t\overline\vartheta_1-\partial_t\vartheta_2\|_{L^{2}(\Omega)}^2\dt+\sup_I\|\partial_t(J_{\eta_1}\overline\rho_1-J_{\eta_2}\overline\rho_2)\|_{W^{2,2}(\Omega)}^2\int_I\|
 \partial_t\overline\vartheta_2\|_{L^{2}(\Omega)}^2\dt\\
&\lesssim T \|(\overline\vartheta_1,\eta_1)-(\overline\vartheta_2,\eta_2)\|_{ \raisebox{-1ex}{$\mathsmaller{ \mathcal{X} } $}_{\overline\bv}  \raisebox{-1ex}{$\mathsmaller{ \times \mathcal{X}  }$}_{\eta} },
\end{align*}

\begin{align*}
\int_I&\|
\partial_t\Big(\kappa\Div\left((\bfA_{\eta_0}-\bfA_{\eta_1}) \nabla\overline\vartheta_1\right)-\kappa\Div((\bfA_{\eta_0}-\bfA_{\eta_2}) \nabla\overline\vartheta_2) \Big)\|_{L^{2}(\Omega)}^2\dt\\
&\lesssim\int_I\|
\partial_t\big((\bfA_{\eta_0}-\bfA_{\eta_1}) \nabla\overline\vartheta_1-(\bfA_{\eta_0}-\bfA_{\eta_2}) \nabla\overline\vartheta_2\big)\|_{W^{1,2}(\Omega)}^2\dt\\
&\lesssim\int_I\|
(\bfA_{\eta_0}-\bfA_{\eta_1}) \partial_t(\nabla\overline\vartheta_1- \nabla\overline\vartheta_2)\|_{W^{1,2}(\Omega)}^2\dt+\int_I\|
(\bfA_{\eta_1}-\bfA_{\eta_2}) \partial_t\nabla\overline\vartheta_2\|_{W^{1,2}(\Omega)}^2\dt\\
&\quad +\int_I\|
\partial_t(\bfA_{\eta_0}-\bfA_{\eta_1}) (\nabla\overline\vartheta_1- \nabla\overline\vartheta_2)\|_{W^{1,2}(\Omega)}^2\dt+\int_I\|
\partial_t(\bfA_{\eta_1}-\bfA_{\eta_2}) \nabla\overline\vartheta_2\|_{W^{1,2}(\Omega)}^2\dt\\
&\lesssim\sup_I\|\eta_0-\eta_1\|_{W^{2,2}(\omega)}^2\int_I\|
\partial_t(\overline\vartheta_1-\overline\vartheta_2)\|_{W^{2,2}(\Omega)}^2\dt+\sup_I\|\eta_1-\eta_2\|_{W^{2,2}(\omega)}^2\int_I\|
\partial_t\overline\vartheta_2\|_{W^{2,2}(\Omega)}^2\dt\\
&\quad +\sup_I\|\partial_t\eta_1\|_{W^{2,2}(\omega)}^2\int_I\|
\overline\vartheta_1-\overline\vartheta_2\|_{W^{2,2}(\Omega)}^2\dt+\sup_I\|\partial_t(\eta_1-\eta_2)\|_{W^{2,2}(\omega)}^2\int_I\|
\overline\vartheta_2\|_{W^{2,2}(\Omega)}^2\dt\\
&\lesssim T \|(\overline\vartheta_1,\eta_1)-(\overline\vartheta_2,\eta_2)\|_{ \raisebox{-1ex}{$\mathsmaller{ \mathcal{X} } $}_{\overline\bv}  \raisebox{-1ex}{$\mathsmaller{ \times \mathcal{X}  }$}_{\eta} }
\end{align*}
The estimates for $\mathfrak h$ are  derived analogously  to  those for
$$-\kappa\Div\big((\bfA_{\eta_0}-\bfA_\eta) \nabla\overline\vartheta\big).$$
In fact, taking the trace is only half a derivative, whereas $\Div(\cdot)$ is a full one.  Consequently, the boundary estimates are less restrictive and follow by the same argument  together with the fractional Leibniz rule.

 
 \subsection{Local Strong Solutions}\label{sec:localstrongfixedpoint}
 We now establish the existence of a local-in-time strong solution to the full system \eqref{eq:ShellEq}--\eqref{eq:ContMomentEq}. Our approach relies on combining the subproblems into a single fixed-point framework. For notational consistency, we continue to work on the time interval $ I = (0, T) $  -- with $ T> 0 $ implicitly restricted to ensure that all subsequent estimates remain valid -- and introduce the space 
\begin{equation*}
\begin{aligned}
  \mathcal{X}_{\overline{\rho}} &:=   \mathbb{W}^{0, 3}_{\infty,2} \left(I\times \Omega\right)\cap \mathbb{W}^{1, 2}_{\infty,2} \left(I\times \Omega\right), &
\end{aligned}
\end{equation*}
endowed with the norm
\begin{equation*}
\begin{aligned}
 \Vert \overline{\rho}\Vert_{ \raisebox{-1.5ex}{$\mathcal{X}$}_{\overline{\rho}} } &:=  \sup_{t\in I}\Big( \Vert \overline{\rho}(t)\Vert_{W^{3,2}(\Omega)}
+
\Vert \partial_t\overline{\rho}(t)\Vert_{W^{2,2}(\Omega )} 
\Big). &
\end{aligned}
\end{equation*}
Furthermore, we consider the subspace of initial data
\[\mathcal{I} := \bigg\{ \left( \overline{\rho}_0, \overline{\bv}_0, \overline\theta_0, \eta_0, \eta_{*}  \right) \in \big(W^{3,2}(\Omega)\big)^3\times W^{5,2}(\omega)\times  W^{3,2}(\omega) \, \Big| \, \overline{\bv}_{0}\circ\bm{\varphi} = \eta_{*}\bn \; \text{ and } \; \nabla\overline\theta_0 \cdot \bn\circ \bm{\varphi}   = 0 \quad \text{on } \omega  \bigg\}, \]
endowed with the norm
\[
\Vert \left(\overline{\rho}_0, \overline{\bv}_0, \overline\theta_0,  \eta_0, \eta_{*}  \right) \Vert_{\raisebox{-1ex}{$\mathsmaller{ \mathcal{I} }$} } :=  \Vert  \overline{\rho}_0  \Vert_{W^{3,2}(\Omega)} + \Vert  \overline{\bv}_0  \Vert_{W^{3,2}(\Omega)}  + \Vert  \overline{\theta}_0  \Vert_{W^{3,2}(\Omega)} +  \Vert  \eta_0  \Vert_{W^{5,2}(\omega)} + \Vert  \eta_*  \Vert_{W^{3,2}(\omega)}.  
\]
 \noindent For $ \overline{\rho} \in  \mathcal{X}_{\overline\rho},  $ and $ \overline{\theta} \in  \mathcal{X}_{\overline\bv},  $ let  $(\overline\bv, \eta) $ be the unique strong solution of the momentum--structure subproblem corresponding to the system \eqref{momEqAloneBar}--\eqref{shellEqAloneBar}, with data $ \left( \overline{\rho}, \overline{\theta}, \overline{\bv}_0, \eta_0, \eta_*\right)$ obtain from \cref{thm:transformedSystem1}. Given such a pair $(\overline\bv, \eta), $ we define $\overline{\rho}^{\#}$ to be the unique strong solution of \eqref{rhoEquAloneTransform}--\eqref{initialCondSolvSubProTransform} associated with  the data $\left(\overline{\rho}_0, \overline\bv, \eta  \right)$. Likewise, we define  $\overline{\theta}^{\#}$ as the unique strong solution of  \eqref{eq:31.03}  corresponding to the  data $( \overline{\rho}, \overline\bv, \eta, \overline\theta_0 )$. \\
This induces the mapping $\mathbf{F} = \mathbf{F}_1 \circ \mathbf{F}_2 $ where 
\[
\mathbf{F}(\overline{\rho}, \overline\theta) = \big(\overline{\rho}^{\#}, \overline{\theta}^{\#} \big), \qquad 
\mathbf{F}_2(\overline{\rho}, \overline\theta) = (\overline\bv, \eta) , \quad \text{and} \;\;\;  \mathbf{F}_1(\overline\bv, \eta) = \big(\overline{\rho}^{\#}, \overline{\theta}^{\#} \big). 
\]
 We shall consider $\mathbf{F} $ on the closed ball 
 \[\mathcal{B}_R := \bigg\{ (\overline{\rho}, \overline\theta) \in \mathcal{X}_{\overline\rho} \times \mathcal{X}_{\overline\bv} \colon  \Vert (\overline{\rho}, \overline\theta)\Vert_{ \raisebox{-1ex}{$\mathsmaller{ \mathcal{X}  }$}_{\overline\rho} \raisebox{-1ex}{$\mathsmaller{ \times \mathcal{X} }$}_{\overline\bv}   } \leq R   \bigg\}.  \]
 The remainder of this section is devoted to proving that, for  sufficiently small time $T > 0 $, the map  
 \[
 \mathbf{F} \colon  \mathcal{X}_{\overline\rho} \times \mathcal{X}_{\overline\bv} \to  \mathcal{X}_{\overline\rho} \times \mathcal{X}_{\overline\bv} 
 \]
 maps $ \mathcal{B}_R $ into itself and is  a strict contraction. Consequently, $\mathbf{F} $ admits a unique fixed point,  which in turn yields existence of  a unique  strong solution to the fully coupled system \eqref{eq:ShellEq}--\eqref{interfaceCond2}.\\[-0.2cm]

\medskip

\noindent\textbf{Step 1: }  $\mathbf{F}\big( \mathcal{B}_R  \big) \subset \mathcal{B}_R  $. \\
Let $(\overline{\rho}, \overline\theta) \in \mathcal{B}_R $,  then by the a priori estimate \eqref{eq:ContSubProbEstimate},  
\begin{equation}
 \begin{aligned} \label{eq:lrhoHatEstimate}
\sup_{t\in I} & \Big( \Vert \overline{\rho}^{\#}(t)\Vert_{W^{3,2}(\Omega )}^2 
+
\Vert \partial_t\overline{\rho}^{\#}(t)\Vert_{W^{2,2}(\Omega )}^2 
\Big)
\\
& \lesssim
 \Vert  \overline{\rho}_0\Vert_{W^{3,2}(\Omega)}^2  
\Bigg(1 + \sup\limits_I \Vert \partial_t \eta \Vert_{W^{3,2}(\omega)}^2  + \int_I\Vert  \overline\bv
\Vert_{W^{4,2}(\Omega )}^2 \dt    
\\
&\qquad  \qquad\qquad\qquad  + 
\int_I\Vert  \partial_{t}^2 \overline\bv
\Vert_{L^{2}(\Omega)}^2 \dt
\Bigg)
  \exp{\bigg( c\int_I \big(\Vert \partial_t \eta \Vert_{W^{4,2}(\omega)} + 
\Vert  \overline\bv\Vert_{W^{4,2}(\Omega)}\big)  \dt \bigg)}.
\end{aligned}
\end{equation} 
\noindent Moreover, using \eqref{linearestimate}, it follows that 
\begin{align}
 \nonumber 
\sup_{t\in I} &\Big( \Vert \overline{\rho}^{\#}(t)\Vert_{W^{3,2}(\Omega )}^2 
+
\Vert \partial_t \overline{\rho}^{\#}(t)\Vert_{W^{2,2}(\Omega)}^2 
\Big)
\\\nonumber&\lesssim\Vert  \overline{\rho}_0\Vert_{W^{3,2}(\Omega)}^2 e^{CT} 
\bigg(1+CT+\Vert \left(\overline{\rho}_0, \overline{\bv}_0, \overline\theta_0,  \eta_0, \eta_{*}  \right) \Vert_{\raisebox{-1ex}{$\mathsmaller{ \mathcal{I} }$} }^{2}  +  T\sup_{ I}\left( \Vert \overline{\rho}\,\overline\theta \Vert_{W^{3,2}(\Omega )}^2 + \Vert \partial_t (\overline{\rho}\,\overline\theta )\Vert_{W^{2,2}(\Omega )}^2 \right) 
\bigg)
\\
& \lesssim\Vert  \overline{\rho}_0\Vert_{W^{3,2}(\Omega)}^2 e^{CT} 
\bigg(1+CT+\Vert \left(\overline{\rho}_0, \overline{\bv}_0, \overline\theta_0,  \eta_0, \eta_{*}  \right) \Vert_{\raisebox{-1ex}{$\mathsmaller{ \mathcal{I} }$} }^{2}  +  T\Vert \overline{\rho}\Vert_{ \raisebox{-1ex}{$ \mathsmaller{ \mathcal{X} }$}_{\overline\rho } }^2 \Vert \overline\theta \Vert_{\raisebox{-1ex}{$\mathsmaller{ \mathcal{X} }$}_{\overline\bv} }^2
\bigg),\label{eq:rrhoHatEstimate}
\end{align}  
where the constant $C > 0 $ depends on $R > 0$, on  the tubular neighbourhood radius $L > 0$, and linearly on $\mu $ and $\lambda$.  However, the constant hidden in ``$\lesssim$'' is independent of $R$.
On the other hand, we deduce from \eqref{eq:31.03B}, \eqref{eq:Step1-1}, \eqref{eq:Step1-2} and \eqref{eq:Step1-3}--\eqref{eq:Step1-4}   that
\begin{equation} \label{eq:ThetaHatEstimate}
\begin{aligned} 
\Vert \overline{\theta}^{\#} \Vert_{ \raisebox{-1ex}{$\mathsmaller{ \mathcal{X} }$}_{\overline\bv} }^2  & \lesssim  \Vert \left(\overline{\rho}_0, \overline{\bv}_0, \overline\theta_0,  \eta_0, \eta_{*}  \right) \Vert_{\raisebox{-1ex}{$\mathsmaller{ \mathcal{I} }$} }^{2}    +  T \Vert \left(\overline{\bv}, \eta   \right)\Vert_{ \raisebox{-1ex}{$\mathsmaller{ \mathcal{X} } $}_{\overline\bv}  \raisebox{-1ex}{$\mathsmaller{ \times \mathcal{X}  }$}_{\eta} }^2
\\[0.4em]
& \lesssim  \Vert \left(\overline{\rho}_0, \overline{\bv}_0, \overline\theta_0,  \eta_0, \eta_{*}  \right) \Vert_{\raisebox{-1ex}{$\mathsmaller{ \mathcal{I} }$} }^{2}    
+  C Te^{CT}  \left( \Vert \left(\overline{\rho}_0, \overline{\bv}_0, \overline\theta_0,  \eta_0, \eta_{*}  \right) \Vert_{\raisebox{-1ex}{$\mathsmaller{ \mathcal{I} }$} }^{2}  +  T\Vert \overline{\rho}\Vert_{ \raisebox{-1ex}{$\mathsmaller{ \mathcal{X} }$}_{\overline\rho } }^2 \Vert \overline\theta \Vert_{\raisebox{-1ex}{$\mathsmaller{ \mathcal{X} }$}_{\overline\bv} }^2  \right)
\end{aligned}
\end{equation}
Hence,  up to increasing the radius  $R > 0 $,  we deduce from \eqref{eq:rrhoHatEstimate}--\eqref{eq:ThetaHatEstimate} that, for  $T> 0 $ small enough,  
\[
\Vert ( \overline{\rho}^{\#} , \overline{\theta}^{\#} ) \Vert_{ \raisebox{-1ex}{$\mathsmaller{ \mathcal{X} }$}_{\overline\rho}  \raisebox{-1ex}{$\mathsmaller{\times \mathcal{X} }$}_{\overline\bv} }  \leq R .
\]

\medskip

\noindent\textbf{Step 2: }  $\mathbf{F} $ is a strict contraction. \\
For each $\mathtt{j} \in \{ 1, 2\}$, let $(\overline{\bv}_\mathtt{j}, \eta_\mathtt{j}) = \mathbf{F}_2(\overline{\rho}_\mathtt{j}, \overline\theta_\mathtt{j}) $ be the unique strong solution of the subproblem  \eqref{momEqAloneBar}--\eqref{shellEqAloneBar} obtained from  \cref{thm:transformedSystem1}, and let 
$(\overline{\rho}^{\#}_\mathtt{j},  \overline{\theta}^{\#}_\mathtt{j}) =  \mathbf{F}_1(\overline{\bv}_\mathtt{j}, \eta_\mathtt{j}) $ denote the corresponding  solutions of \eqref{rhoEquAloneTransform}--\eqref{initialCondSolvSubProTransform} and   \eqref{eq:31.03}  respectively (their existence follows from \cref{thm:mainFP} and \cref{prop:InternalEnerLin}). 
 In view of the contraction argument, we aim to bound the difference $(\overline{\rho}^{\#}, \overline{\theta}^{\#})_{1,2} :=  (\overline{\rho}^{\#}_1, \overline{\theta}^{\#}_1) - (\overline{\rho}^{\#}_2, \overline{\theta}^{\#}_2) $   in the norm $ \mathcal{X}_{\overline\rho} \times \mathcal{X}_{\overline\bv} $ in terms of $ \overline{\bv}_{1,2} := \overline{\bv}_1 - \overline{\bv}_2 $ and $\eta_{1,2} := \eta_1 - \eta_2 $ in suitable norms. \\ 
To simplify the notation in the sequel,  we define the spaces 
  \[
\begin{aligned}
\mathcal{Y}_{\overline{\rho}}
& :=
\mathbb{W}^{0, 2}_{\infty,2}(I\times \Omega)
\cap
\mathbb{W}^{1, 1}_{\infty,2}(I\times \Omega),
&
\mathcal{Y}_{\overline\bv}
&:=
\mathbb{W}^{0, 3}_{2,2}(I\times \Omega)
\cap
\mathbb{W}^{1, 2}_{2,2}(I\times \Omega),
\\[0.4em]
\mathcal{Y}_{\eta}
& :=
\mathbb{W}^{1, 3}_{2,2}(I\times \omega)
\cap
\mathbb{W}^{2, 2}_{2,2}(I\times \omega),
&
\mathcal{Y}_{\overline\theta}
&:=
\mathcal{Y}_{\overline\bv}
\cap \mathcal{Y}_{\overline\rho},
\end{aligned}
\]
endowed respectively with their canonical norms.
Since the velocity and the displacement are regarded as prescribed,  the analysis of the continuity equation \eqref{rhoEquAloneTransform}--\eqref{initialCondSolvSubProTransform} is independent of the temperature variable. Hence, the arguments and  estimates established  in \cite[Section 4]{ngougoue2025local}   apply without modification.  More precisely, it holds that 
\begin{equation}\label{eq:LinTransportTypEstimFinal}
  \Vert  \overline{\rho}^{\#}_{1,2}  \Vert_{\raisebox{-1ex}{$\mathsmaller{ \mathcal{Y} }$}_{\overline{\rho}}  } \leq cT^{1/2} e^{cT}\Vert \left( \overline{\bv}_{1,2}, \eta_{1,2}   \right)\Vert_{  \raisebox{-1ex}{$\mathsmaller{ \mathcal{Y} }$}_{\overline\bv}  \raisebox{-1ex}{$ \mathsmaller{ \times \mathcal{Y} }$}_{\eta}}, 
 \end{equation} 
 where  $c = c\Big( \Omega, \eta_1, \eta_2,   \overline{\bv}_1,  \overline{\bv}_2, \overline{\rho}^{\#}_1, \overline{\rho}^{\#}_2  \Big) > 0 .$
 
Moreover, in order to control the pair $\left( \overline{\bv}_{1,2}, \eta_{1,2}   \right) $  in terms of $\overline{\rho}_{1,2} := \overline{\rho}_1 - \overline{\rho}_2 $ in the  $ \mathcal{Y}_{\overline{\rho}}\, -$norm, we note 
that  $\left( \overline{\bv}_{1,2}, \eta_{1,2}   \right) $  satisfies the momentum--structure system 
\begin{align}
 \mathbf{\mathcal{L}}_{\overline{\bv}} (\overline{\bv}_{1,2}) &= \mathbf{h}_{1,2} - \Div{\mathbf{H}_{1,2}}  - \Div\Big( \left( \overline\rho_1\,\overline\theta_1 - \overline\rho_2\,\overline\theta_2 \right) \mathbf{B}_{\eta_0} \Big), 
 \label{momEqAloneBarDiff}
 \\
 \mathbf{\mathcal{L}}_{\eta} (\eta_{1,2}) &=  \bn^\intercal \left[\mathbf{H}_{1,2} - \mu \mathbf{A}_{\eta_0}\nabla\overline{\bv}_{1,2} - \dfrac{\lambda + \mu}{J_{\eta_0}}\left(\mathbf{B}_{\eta_0}\colon \nabla\overline{\bv}_{1,2} \right)\mathbf{B}_{\eta_0}  + \left( \overline\rho_1\,\overline\theta_1 - \overline\rho_2\,\overline\theta_2 \right) \mathbf{B}_{\eta_0}  \right]\circ\bm{\varphi} \bn ,
 \label{shellEqAloneBarDiff}
\end{align} 
 with  $\overline{\bv}_{1,2}  \circ \bm{\varphi}  = (\partial_t\eta_{1,2})\bn$ on $I\times \omega$, and source terms  
 \[ \mathbf{h}_{1,2} := \mathbf{h}_{\eta_1}(\overline{\bv}_1)\!\left[\,\overline{\rho}_1\right] - \mathbf{h}_{\eta_2}(\overline{\bv}_2)\!\left[\,\overline{\rho}_2\right],  \qquad  \mathbf{H}_{1,2} := \mathbf{H}_{\eta_1}(\overline{\bv}_1)\!\left[\,\overline{\rho}_1\right] - \mathbf{H}_{\eta_2}(\overline{\bv}_2)\!\left[\,\overline{\rho}_2 \right].  \]
 Of note, this notation highlights the dependence of $\mathbf{h}_{\eta}(\overline{\bv}) $ and $\mathbf{H}_{\eta}(\overline{\bv})$ on the prescribed density $\overline{\rho}$, which we now indicate explicitly via square brackets. Importantly, 
 the linear differential operators $ \mathbf{\mathcal{L}}_{\overline{\bv}} $  and  $ \mathbf{\mathcal{L}}_{\eta} $ are defined respectively on the fixed reference domains $I\times\Omega $ and $I\times \omega$ by
 	\begin{align*}
	\mathbf{\mathcal{L}}_{\overline{\bv}}(\bm{\phi}) &:= J_{\eta_0}\overline{\rho}_0 \,\partial_t\bm{\phi} - \Div\left[\mu \mathbf{A}_{\eta_0}\nabla\bm{\phi} + \dfrac{ \lambda + \mu }{J_{\eta_0}} \left(\mathbf{B}_{\eta_0}\colon \nabla\bm{\phi} \right)\mathbf{B}_{\eta_0} \right]  ,
	\\
	\mathbf{\mathcal{L}}_{\eta}(\zeta) &:=  \partial_t^2\zeta - \partial_t\Dely \zeta + \Dely^2\zeta.
	\end{align*}
Owing to the continuous dependence of  solutions to \eqref{momEqAloneBar}--\eqref{shellEqAloneBar} on the input data (cf.~\eqref{linearestimate}), the following  energy estimate holds for the pair $\left( \overline{\bv}_{1,2}, \eta_{1,2}   \right) $ (see \cite[Section 4]{ngougoue2025local} for further details) : 
\begin{align}\label{eq:EstimCombVeloDispl1}
\Vert \left( \overline{\bv}_{1,2}, \eta_{1,2}   \right)\Vert_{  \raisebox{-1ex}{$\mathsmaller{ \mathcal{Y} }$}_{\overline\bv}  \raisebox{-1ex}{$\mathsmaller{ \times \mathcal{Y} }$}_{\eta}} \leq  CT^{1/2}e^{CT} \left(  \Vert  \overline{\rho}_{1}  -  \overline{\rho}_{2} \Vert_{ \raisebox{-1ex}{$\mathsmaller{ \mathcal{Y} }$}_{\overline\rho} }  
+     \Vert \overline\rho_1\,\overline\theta_1 - \overline\rho_2\,\overline\theta_2\Vert_{ \raisebox{-1ex}{$\mathsmaller{ \mathcal{Y} }$}_{\overline\rho} }    \right),
\end{align}  
with a constant $C = C(\mu, \lambda, R,  \eta_1, \eta_2,   \overline{\bv}_1,  \overline{\bv}_2) > 0$ controlled in terms of $R$. \\
Whence, 
\begin{align}\label{eq:EstimCombVeloDispl}
\Vert \left( \overline{\bv}_{1,2}, \eta_{1,2}   \right)\Vert_{  \raisebox{-1ex}{$\mathsmaller{ \mathcal{Y} }$}_{\overline\bv}  \raisebox{-1ex}{$\mathsmaller{ \times \mathcal{Y} }$}_{\eta}} \leq  C T^{1/2} e^{CT} \left(  \Vert  \overline{\rho}_{1,2}  \Vert_{ \raisebox{-1ex}{$\mathsmaller{ \mathcal{Y} }$}_{\overline\rho} }  
+     \Vert \overline\theta_{1,2} \Vert_{ \raisebox{-1ex}{$\mathsmaller{ \mathcal{Y} }$}_{\overline\theta} }    \right),
\end{align}  
Substituting \eqref{eq:EstimCombVeloDispl} in \eqref{eq:LinTransportTypEstimFinal} yields 
 \begin{equation}\label{eq:FixPointFinalRho}
  \Vert  \overline{\rho}^{\#}_{1,2}  \Vert_{\raisebox{-1ex}{$ \mathsmaller{ \mathcal{Y} }$}_{\overline{\rho}} }   \leq c  T e^{cT}  \Vert \left( \overline{\rho}_{1,2},  \overline{\theta}_{1,2} \right)\Vert_{ \raisebox{-1ex}{$ \mathsmaller{ \mathcal{Y} }$}_{\overline{\rho} }    \raisebox{-1ex}{$\mathsmaller{ \times \mathcal{Y} }$}_{\overline\theta}    }  , 
 \end{equation} 
 with $c = c(\gamma, \mu, \lambda, R,  \eta_1, \eta_2,   \overline{\bv}_1,  \overline{\bv}_2)$ controlled in terms of $R$.  \\ 
 
To close the contraction argument, it remains to control $ \Vert  \overline{\theta}^{\#}_{1,2}  \Vert_{\raisebox{-1ex}{$ \mathsmaller{ \mathcal{Y} }$}_{\overline{\theta}} } $. This is achieved by a two-step application of Solonnikov's maximal regularity theory, combined with   Sobolev embedding arguments. We first recall that $\overline{\theta}^{\#}_{1,2}$ is governed by the following evolution equation:
\begin{equation}\label{eq:31.03.Diff}
\left\{\begin{aligned}
\partial_t\overline{\theta}^{\#}_{1,2}    -
\Div\bigg(\frac{\kappa }{c_vJ_{\eta_0}\overline\rho_0} \bfA_{\eta_0} \nabla\overline{\theta}^{\#}_{1,2}\;\!\bigg) 
&= \mathfrak g_{\eta_1} - \mathfrak g_{\eta_2}   & \text{in  } I\times \Omega,
\\[0.4em]
\frac{\kappa }{c_vJ_{\eta_0}\overline\rho_0}\bfA_{\eta_0}\nabla\overline{\theta}^{\#}_{1,2}\cdot \bfn\circ\bfvarphi^{-1}&=\mathfrak h_{\eta_1} - h_{\eta_2}      & \text{on } I\times \partial\Omega,
\\[0.4em]
\overline{\theta}^{\#}_{1,2}(0,\cdot)&= 0 & \text{in }\Omega .
\end{aligned} \right.
\end{equation} 
By \cite[Chapter 5, Section 21]{solonnikov1965}, we deduce  that 
\begin{align}\label{eq:31.03B.Diff}
\begin{aligned}
 \Vert \overline{\theta}^{\#}_{1,2} \Vert_{\mathbb{W}^{3/2, 0}_{2, 2}(I\times\Omega)}^2  +  \Vert \overline{\theta}^{\#}_{1,2} \Vert_{\mathbb{W}^{0, 3}_{2, 2}(I\times\Omega)}^2  &\lesssim   \|\mathfrak g_{\eta_1} - \mathfrak g_{\eta_2}  \|_{ \mathbb{W}^{0, 1}_{2, 2}(I\times\Omega)}^2  +\|\mathfrak g_{\eta_1} - \mathfrak g_{\eta_2} \|_{\mathbb{W}^{1/2, 0}_{2, 2}(I\times\Omega) }^2
\\[0.4em]
& \quad + \|\mathfrak h_{\eta_1} - \mathfrak h_{\eta_2}  \|_{ \mathbb{W}^{0, 3/2}_{2, 2}(I\times\partial\Omega)}^2  +\|\mathfrak h_{\eta_1} - \mathfrak h_{\eta_2} \|_{\mathbb{W}^{3/4, 0}_{2, 2}(I\times\partial\Omega) }^2 .
\end{aligned}
\end{align}
It is immediate from the definition of $\mathfrak h_{\eta}$ (omitting the domain from the norm notation) that 
\begin{align*}
\|\mathfrak h_{\eta_1} - \mathfrak h_{\eta_2}  \|_{ \mathbb{W}^{0, 3/2}_{2, 2}} \leq  \Vert \eta_2 - \eta_1 \Vert_{\mathbb{W}^{0, 3}_{\infty, 2}}  \Vert \nabla\overline{\theta}_{2} \Vert_{\mathbb{W}^{0, 3/2}_{2, 2}}  +  \Vert \bfA_{\eta_1} - \bfA_{\eta_0} \Vert_{\mathbb{W}^{0, 3/2}_{\infty, 2}} \Vert \nabla\overline{\theta}_{1, 2} \Vert_{\mathbb{W}^{0, 3/2}_{2, 2}}.
\end{align*}
Thus,
\begin{align}\label{eq:Est.h.1}
\|\mathfrak h_{\eta_1} - \mathfrak h_{\eta_2}  \|_{ \mathbb{W}^{0, 3/2}_{2, 2}} \lesssim  T^{1/2}  \left( \Vert \partial_t\eta_{1,2} \Vert_{\mathbb{W}^{0, 3}_{2, 2}}   +  \Vert  \overline{\theta}_{1, 2}  \Vert_{\mathbb{W}^{0, 5/2}_{2, 2}} \right) .  
\end{align}
Moreover, by \eqref{eq:Multipl.Prop} and  fractional Leibniz rule, it follows that 
\begin{equation*}
\begin{aligned}
\|\mathfrak h_{\eta_1} - \mathfrak h_{\eta_2} \|_{\mathbb{W}^{3/4, 0}_{2, 2} }  & \leq  \Vert \bfA_{\eta_1} -  \bfA_{\eta_2} \Vert_{\mathbb{W}^{0, 3/2}_{\infty, 2}} \Vert \nabla\overline{\theta}_{2} \Vert_{\mathbb{W}^{3/4, 0 }_{2, 2}}  +  \Vert \bfA_{\eta_1} -  \bfA_{\eta_2} \Vert_{\mathbb{W}^{3/4, 3/2}_{2, 2}} \Vert \nabla\overline{\theta}_{2} \Vert_{\mathbb{W}^{0, 0}_{\infty, 2}} 
\\[0.4em]
& \quad + \Vert \bfA_{\eta_1} - \bfA_{\eta_0} \Vert_{\mathbb{W}^{0, 3/2}_{\infty, 2}} \Vert \nabla\overline{\theta}_{1, 2} \Vert_{\mathbb{W}^{3/4, 0}_{2, 2}}    +  \Vert \bfA_{\eta_1} - \bfA_{\eta_0} \Vert_{\mathbb{W}^{3/4, 3/2}_{2, 2}} \Vert \nabla\overline{\theta}_{1, 2} \Vert_{\mathbb{W}^{0, 0}_{\infty, 2}} .
\end{aligned}
\end{equation*}
Whence, 
\begin{align}\label{eq:Est.h.2}
\|\mathfrak h_{\eta_1} - \mathfrak h_{\eta_2} \|_{\mathbb{W}^{3/4, 0}_{2, 2} }  \lesssim \max\big\{T,  T^{1/2} \big\} \left( \Vert \eta_{1,2} \Vert_{  \raisebox{-1ex}{$\mathsmaller{ \mathcal{Y} }$}_{\eta}}    +  \Vert \nabla \overline{\theta}_{1, 2}  \Vert_{\mathbb{W}^{3/4, 0}_{2, 2}}    +   \Vert  \partial_t \nabla \overline{\theta}_{1, 2}  \Vert_{\mathbb{W}^{0, 0}_{2, 2}}      \right) .  
\end{align}
Arguing as in \cref{sec:InternalProb}, we deduce that
\begin{align}
\|\mathfrak g_{\eta_1} - \mathfrak g_{\eta_2}  \|_{ \mathbb{W}^{0, 1}_{2, 2}}  & \lesssim T \left( \Vert \eta_{1,2} \Vert_{  \raisebox{-1ex}{$\mathsmaller{ \mathcal{Y} }$}_{\eta}} + \Vert \overline\bv_{1,2} \Vert_{  \raisebox{-1ex}{$\mathsmaller{ \mathcal{Y} }$}_{\overline\bv}} +  \Vert \overline\rho_{1,2} \Vert_{  \raisebox{-1ex}{$\mathsmaller{ \mathcal{Y} }$}_{\overline\rho}} +  \Vert  \overline{\theta}_{1, 2}  \Vert_{  \raisebox{-1ex}{$\mathsmaller{ \mathcal{Y} }$}_{\overline\bv}  }   \right) , \label{eq:Est.g.1}
\\[0.25cm]
\|\mathfrak g_{\eta_1} - \mathfrak g_{\eta_2} \|_{\mathbb{W}^{1/2, 0}_{2, 2} }  & \lesssim T  \left( \Vert \eta_{1,2} \Vert_{  \raisebox{-1ex}{$\mathsmaller{ \mathcal{Y} }$}_{\eta}} + \Vert \overline\bv_{1,2} \Vert_{  \raisebox{-1ex}{$\mathsmaller{ \mathcal{Y} }$}_{\overline\bv}} +  \Vert \overline\rho_{1,2} \Vert_{  \raisebox{-1ex}{$\mathsmaller{ \mathcal{Y} }$}_{\overline\rho}} +  \Vert  \overline{\theta}_{1, 2}  \Vert_{  \raisebox{-1ex}{$\mathsmaller{ \mathcal{Y} }$}_{\overline\bv}  }   \right) .  \label{eq:Est.g.2}
\end{align} 
Hence, substituting \eqref{eq:Est.h.1}--\eqref{eq:Est.g.2} into \eqref{eq:31.03B.Diff} yields 
 \begin{align}\label{eq:31.03B.Diff.Final1}
\begin{aligned}
 \Vert \overline{\theta}^{\#}_{1,2} \Vert_{\mathbb{W}^{3/2, 0}_{2, 2}}^2  +  \Vert \overline{\theta}^{\#}_{1,2} \Vert_{\mathbb{W}^{0, 3}_{2, 2}}^2  &\lesssim   T^{1/2}    \left( \Vert \eta_{1,2} \Vert_{  \raisebox{-1ex}{$\mathsmaller{ \mathcal{Y} }$}_{\eta}} + \Vert \overline\bv_{1,2} \Vert_{  \raisebox{-1ex}{$\mathsmaller{ \mathcal{Y} }$}_{\overline\bv}} +  \Vert \overline\rho_{1,2} \Vert_{  \raisebox{-1ex}{$\mathsmaller{ \mathcal{Y} }$}_{\overline\rho}} +  \Vert  \overline{\theta}_{1, 2}  \Vert_{  \raisebox{-1ex}{$\mathsmaller{ \mathcal{Y} }$}_{\overline\bv}  }   \right).
\end{aligned}
\end{align}
However, \eqref{eq:31.03B.Diff.Final1} does not provide control of $\partial_t \overline{\theta}^{\#}_{1, 2}$  in $\mathbb{W}^{0, 1}_{\infty, 2}$. To recover this additional time regularity control, we differentiate   \eqref{eq:31.03.Diff} with respect to time and apply Solonnikov's maximal regularity once again.  Indeed, $\partial_t \overline{\theta}^{\#}_{1, 2}$ solves 
\begin{equation}\label{eq:31.03.Diff.Time}
\left\{\begin{aligned}
\partial_t   \Theta   -
\Div\bigg(\frac{\kappa }{c_vJ_{\eta_0}\overline\rho_0} \bfA_{\eta_0} \nabla\Theta\;\!\bigg) 
&= \partial_t \left( \mathfrak g_{\eta_1} - \mathfrak g_{\eta_2}   \right)  & \text{in  } I\times \Omega,
\\[0.4em]
\frac{\kappa }{c_vJ_{\eta_0}\overline\rho_0}\bfA_{\eta_0}\nabla\Theta\cdot \bfn\circ\bfvarphi^{-1}&= \partial_t \left( \mathfrak h_{\eta_1} - h_{\eta_2}    \right)  & \text{on } I\times \partial\Omega,
\\[0.4em]
\Theta(0,\cdot)&= 0 & \text{in }\Omega .
\end{aligned} \right.
\end{equation} 
Hence,
\begin{align}\label{eq:31.03B.Diff.Time}
\begin{aligned}
& \Vert \overline{\theta}^{\#}_{1,2} \Vert_{\mathbb{W}^{2, 0}_{2, 2}(I\times\Omega)}^2  +  \Vert \overline{\theta}^{\#}_{1,2} \Vert_{\mathbb{W}^{1, 2}_{2, 2}(I\times\Omega)}^2  
\\[0.4em]
& \quad \lesssim   \| \partial_t \left( \mathfrak g_{\eta_1} - \mathfrak g_{\eta_2} \right)  \|_{ \mathbb{W}^{0, 0}_{2, 2}(I\times\Omega)}^2   + \| \partial_t \left( \mathfrak h_{\eta_1} - \mathfrak h_{\eta_2} \right)  \|_{ \mathbb{W}^{0, 1/2}_{2, 2}(I\times\partial\Omega)}^2  + \| \partial_t \left( \mathfrak h_{\eta_1} - \mathfrak h_{\eta_2}  \right) \|_{\mathbb{W}^{1/4, 0}_{2, 2}(I\times\partial\Omega) }^2 .
\end{aligned}
\end{align}
We now estimate the terms on the right-hand side of \eqref{eq:31.03B.Diff.Time}. Indeed, 
by the nonlinear estimates derived in  \cref{sec:InternalProb}, we infer  that 
\begin{equation} \label{eq:Est.g.1.1}
\| \partial_t \left( \mathfrak g_{\eta_1} - \mathfrak g_{\eta_2} \right)  \|_{ \mathbb{W}^{0, 0}_{2, 2}}   \lesssim T \left( \Vert \eta_{1,2} \Vert_{  \raisebox{-1ex}{$\mathsmaller{ \mathcal{Y} }$}_{\eta}} + \Vert \overline\bv_{1,2} \Vert_{  \raisebox{-1ex}{$\mathsmaller{ \mathcal{Y} }$}_{\overline\bv}} +  \Vert \overline\rho_{1,2} \Vert_{  \raisebox{-1ex}{$\mathsmaller{ \mathcal{Y} }$}_{\overline\rho}} +  \Vert  \overline{\theta}_{1, 2}  \Vert_{  \raisebox{-1ex}{$\mathsmaller{ \mathcal{Y} }$}_{\overline\bv}  }  + \Vert \overline{\theta}_{1,2} \Vert_{\mathbb{W}^{2, 0}_{2, 2}} \right) .
\end{equation}
Moreover, using that $W^{3/2,2}(\partial\Omega)$ is a multiplier in $W^{1/2,2}(\partial\Omega)$, we deduce that 
\begin{align*}
 \| \partial_t \left( \mathfrak h_{\eta_1} - \mathfrak h_{\eta_2} \right)  \|_{ \mathbb{W}^{0, 1/2}_{2, 2}} & \lesssim \Vert \bfA_{\eta_1} -  \bfA_{\eta_2} \Vert_{\mathbb{W}^{0, 3/2}_{\infty, 2}} \Vert \partial_t\nabla\overline{\theta}_{2} \Vert_{\mathbb{W}^{0, 1/2}_{2, 2}}  +  \Vert \partial_t \left( \bfA_{\eta_1} -  \bfA_{\eta_2} \right)  \Vert_{\mathbb{W}^{0, 1/2}_{\infty, 2}} \Vert \nabla\overline{\theta}_{2} \Vert_{\mathbb{W}^{0, 3/2}_{2, 2}} 
\\[0.4em]
& \quad + \Vert \bfA_{\eta_1} - \bfA_{\eta_0} \Vert_{\mathbb{W}^{0, 3/2}_{\infty, 2}} \Vert \partial_t \nabla\overline{\theta}_{1, 2} \Vert_{\mathbb{W}^{0, 1/2}_{2, 2}}    +  \Vert \partial_t \bfA_{\eta_1} \Vert_{\mathbb{W}^{0, 3/2}_{2, 2}} \Vert \nabla\overline{\theta}_{1, 2} \Vert_{\mathbb{W}^{0, 1/2}_{\infty, 2}} . 
\end{align*}
Whence, 
\begin{align}\label{eq:Est.h.1.1}
\| \partial_t \left( \mathfrak h_{\eta_1} - \mathfrak h_{\eta_2} \right)  \|_{ \mathbb{W}^{0, 1/2}_{2, 2}} \lesssim  T^{1/2}  \left( \Vert \eta_{1,2} \Vert_{  \raisebox{-1ex}{$\mathsmaller{ \mathcal{Y} }$}_{\eta}}   +  \Vert \overline{\theta}_{1, 2}  \Vert_{ \raisebox{-1ex}{$\mathsmaller{ \mathcal{Y} }$}_{\overline\bv}} \right) .  
\end{align}
Finally, to estimate the last term, we rely on  \eqref{eq:Multipl.Prop}  together with the fractional Leibniz rule. This yields 
\begin{equation*}
\begin{aligned}
& \| \partial_t \left( \mathfrak h_{\eta_1} - \mathfrak h_{\eta_2}  \right) \|_{\mathbb{W}^{1/4, 0}_{2, 2} }  
\\[0.4em]
& \lesssim  \Vert \bfA_{\eta_1} -  \bfA_{\eta_2} \Vert_{\mathbb{W}^{0, 3/2}_{\infty, 2}} \Vert \partial_t\nabla\overline{\theta}_{2} \Vert_{\mathbb{W}^{1/4, 0 }_{2, 2}}  +  \Vert \bfA_{\eta_1} -  \bfA_{\eta_2}  \Vert_{\mathbb{W}^{1/4, 3/2}_{\infty, 2}} \Vert \partial_t \nabla\overline{\theta}_{2} \Vert_{\mathbb{W}^{0, 0 }_{2, 2}} 
\\[0.4em]
& \quad  + \Vert ( \bfA_{\eta_1} - \bfA_{\eta_0} ) \partial_t \nabla \overline{\theta}_{1, 2}   \Vert_{ \mathbb{W}^{1/4, 0 }_{2, 2}  } +  \Vert  \partial_t \left( \bfA_{\eta_1} - \bfA_{\eta_2} \right) \nabla \overline{\theta}_{2}  \Vert_{ \mathbb{W}^{1/4, 0}_{2, 2}  }    
\\[0.4em]
&\quad   +  \Vert \partial_t \bfA_{\eta_1}   \Vert_{ \mathbb{W}^{0, 3/2}_{2, 2}  }   \Vert \nabla \overline{\theta}_{1, 2}   \Vert_{ \mathbb{W}^{1/4, 0 }_{\infty, 2}  }  +  \Vert \partial_t \bfA_{\eta_1}   \Vert_{ \mathbb{W}^{1/4, 3/2 }_{\infty, 2}  }   \Vert \nabla \overline{\theta}_{1, 2}   \Vert_{ \mathbb{W}^{0, 0}_{2, 2}  } .
\end{aligned}
\end{equation*} 
Thus, using the embedding 
\[
C^{0, 1/2}(I) \hookrightarrow W^{1/4, \infty}(I),
\]
the above estimate further reduces to 
\begin{align}\label{eq:Est.h.2.2}
 \| \partial_t \left( \mathfrak h_{\eta_1} - \mathfrak h_{\eta_2}  \right) \|_{\mathbb{W}^{1/4, 0}_{2, 2} }   \lesssim  \max\big\{ T, T^{1/2}, T^{1/4 } \big\} \left( \Vert \eta_{1,2} \Vert_{  \raisebox{-1ex}{$\mathsmaller{ \mathcal{Y} }$}_{\eta}}    +   \Vert \partial_t \nabla \overline{\theta}_{1, 2}  \Vert_{\mathbb{W}^{1/4, 0}_{2, 2}}    \right) .  
\end{align}
Hence, for $T > 0$ sufficiently small, $\mathbf{F} $ defines a strict contraction.  This thereby completes  the proof of  \cref{theo:mainresult}.


\section{Conditional regularity}

\subsection{The acceleration estimate}

\begin{proof}[Proof of  \cref{theo:MainResult}]
The proof follows exactly the arguments from the isentropic case in \cite{MN} using the test-functions
\[\bm{\psi} :=  \partial_t \bv + \mathcal{E}_\eta (\partial_t\eta\bn)\cdot \nabla\bv ,\]  
\[\bm{\psi}\circ\bm{\varphi}_{\eta} = \left( \partial^{2}_t \eta\right) \!\bn \quad \text{on} \;\;  I_{*}\times \omega ,   \] 
for the momentum-structure subproblem \eqref{eq:ShellEq}--$\eqref{eq:ContMomentEq}_2$. Here $\mathcal E_\eta$ is the extension operator constructed \cite[Section 2.5]{breit2024regularity}
 satisfying for $p\in(1,\infty)$, $s\in(1/p,1]$ and $\alpha\in(0,L)$
\begin{align}\label{eq:ext}
\|\mathcal{E}_\eta (b\mathbf n)\|_{W^{s,p}(\mathcal N_\alpha)}\leq c(\|\nabla\eta\|_{L^\infty(\omega)})\|b\|_{W^{s-1/p,p}(\omega)},\quad b\in W^{s-1/p,p}(\omega).
\end{align}
The only difference is the different pressure now given by $p=\rho\vartheta$ instead of $p=\rho^\gamma$. However, it is still bounded as a consequence of assumption \ref{A2}.  Indeed, after integration by parts,  we derive that 
\begin{align}
&\sup\limits_{I_*} \int_{\Omega_\eta} |\nabla\bv|^2\dx +  \sup\limits_{I_*} \int_{\Omega_\eta} |\Div\bv|^2\dx  + \int_{I_*}\int_{\Omega_\eta} \rho |\dot{\bv}|^2 \dx\dt + \int_{I_*}\int_\omega |\partial^{2}_t \eta|^2 \dy\dt + \sup\limits_{I_*} \int_\omega |\partial_t \naby\eta|^2 \dy
\nonumber
\\
& \lesssim  \int_{I_*}\int_{\Omega_\eta}  \rho |\bv\cdot\nabla\bv|^2 \dx\dt +  \int_{I_*}\int_{\partial\Omega_\eta} (\partial_t\eta \bn)\circ\bm{\varphi}_{\eta}^{-1}\cdot \bn_\eta |\nabla\bv|^2 \dH\dt + \int_{I_*}\Vert \rho^{1/2}\mathcal{E}_\eta (\partial_t\eta\bn) \cdot\nabla\bv \Vert_{L^{2}(\Omega_\eta)}^2 \dt
\nonumber
\\
&\quad  - \int_{I_*} \int_{\Omega_\eta} \nabla\bv \colon \nabla\big( \mathcal{E}_\eta (\partial_t\eta\bn)\cdot \nabla\bv  \big)\dx\dt  - \int_{I_*} \int_{\partial\Omega_\eta}\big(\dot{\bv} \cdot \nabla\bv \big) \cdot \bn_\eta \dH\dt  + \int_{I_*} \int_{\partial\Omega_\eta} (\Div\bv) \dot{\bv}\cdot \bn_\eta \dH\dt 
\nonumber
\\
&\quad + \int_{I_*} \int_{\partial\Omega_\eta} (\partial_t\eta \bn)\circ\bm{\varphi}_{\eta}^{-1}\cdot \bn_\eta |\Div\bv|^2 \dH\dt - \int_{I_*}\int_{\Omega_\eta} (\Div\bv) \Div\big(\mathcal{E}_\eta (\partial_t\eta\bn)\cdot \nabla\bv \big) \dx\dt   \label{eq:AccFirstEstimate}
\\
&\quad +  \int_{I_*} \int_{\Omega_\eta} p(\rho, \vartheta)\Div\big( \mathcal{E}_\eta (\partial_t\eta\bn)\cdot \nabla\bv \big) \dx\dt + \int_{I_*} \int_{\Omega_\eta} p(\rho, \vartheta)\partial_t\Div(\bv) \dx\dt
\nonumber
\\
&\quad + \int_{\Omega_{\eta_0}} |\nabla\bv_0|^2\dx + \int_{\Omega_{\eta_0}} |\Div\bv_0|^2\dx  + \int_\omega |\naby\eta_*|^2 \dy  + \sup\limits_{I_*} \int_\omega |\naby\Dely\eta|^2 \dy   + \int_{I_*}\int_\omega |\partial_t\Dely\eta|^2 \dy \dt 
\nonumber
\\
& \quad =: \sum\limits_{\mathsf{k}=1}^{15} {\mathlarger{\mathfrak{R}}}_\mathsf{k}. \nonumber
\end{align}
The estimates for ${\mathlarger{\mathfrak{R}}}_\mathsf{k} , \;  \mathsf{k} \neq 10 $ are identical to those in \cite[Section 3]{MN}.  The pressure-divergence term ${\mathlarger{\mathfrak{R}}}_{10}$, however, requires a different argument, as the estimate in  \cite[Section 3]{MN}  relies on the renormalised continuity equation \eqref{eq:RenormCE}, which can no longer be used for the ideal gas law $p = \rho \theta$ depending also on the temperature. By the continuity equation $\eqref{eq:ContMomentEq}_1$, we obtain
\[
{\mathlarger{\mathfrak{R}}}_{10} = \int_{I_*} \int_{\Omega_\eta}  \big( \partial_t (\rho \vartheta \Div\bv) - \partial_t\vartheta \rho \Div\bv + \vartheta \Div(\rho\bv) \Div(\bv) \big)  \dx\dt.
\]
Using Reynolds transport theorem and integrating by parts, we further arrive at
\begin{equation}
\begin{aligned}
{\mathlarger{\mathfrak{R}}}_{10} & =  \int_{I_*} \dfrac{\dd}{\dt} \int_{\Omega_\eta} \vartheta \rho \Div(\bv) \dx\dt - \int_{I_*} \int_{\Omega_\eta} \rho \bv \cdot \nabla  \vartheta \Div(\bv)  \dx\dt  - \int_{I_*} \int_{\Omega_\eta}  \rho \vartheta \bv \cdot \nabla\Div(\bv) \dx \dt 
\\
& \quad - \int_{I_*} \int_{\Omega_\eta} \rho\partial_t \vartheta  \Div(\bv) \dx \dt
\\
& =: {\mathlarger{\mathfrak{R}}}_{10}^{\mathtt{a}}  + {\mathlarger{\mathfrak{R}}}_{10}^{\mathtt{b}}  + {\mathlarger{\mathfrak{R}}}_{10}^{\mathtt{c}}   + {\mathlarger{\mathfrak{R}}}_{10}^{\mathtt{d}} . 
\end{aligned}
\end{equation}
It follows immediately from Assumption \ref{A2} that 
\[
 {\mathlarger{\mathfrak{R}}}_{10}^{\mathtt{a}} \lesssim \sup\limits_{I_*}   \Vert \Div(\bv) \Vert_{L^2(\Omega_\eta)}    + \int_{\Omega_{\eta_0}} \vartheta_0 \rho_0 \Div(\bv_0) \dx .
\]
Consequently,  Young's inequality yields
\begin{equation}\label{eq:R10a}
{\mathlarger{\mathfrak{R}}}_{10}^{\mathtt{a}}  \lesssim   \varepsilon \sup\limits_{I_*}  \Vert \nabla\bv \Vert_{L^2(\Omega_\eta)}^2 + c(\varepsilon)      + \int_{\Omega_{\eta_0}} \vartheta_0 \rho_0 \Div(\bv_0) \dx .
\end{equation}
For  ${\mathlarger{\mathfrak{R}}}_{10}^{\mathtt{b}}$, Assumption \ref{A2} implies that 
\[
{\mathlarger{\mathfrak{R}}}_{10}^{\mathtt{b}}  \lesssim \sup\limits_{I_*}  \Vert \nabla\vartheta \Vert_{L^2(\Omega_\eta)} \Vert \rho^{1/2}\bv \Vert_{L^2(\Omega_\eta)}.
\]
Whence, by the energy inequality \eqref{eq:TotalEnergy} and assumption \ref{A1} (yielding boundedness of the temperature),
\begin{equation}\label{eq:R10b}
{\mathlarger{\mathfrak{R}}}_{10}^{\mathtt{b}}  \lesssim   \varepsilon \sup\limits_{I_*}  \Vert \nabla\vartheta \Vert_{L^2(\Omega_\eta)}^2 + c(\varepsilon).
\end{equation}
Likewise,
\begin{equation} \label{eq:R10c}
 {\mathlarger{\mathfrak{R}}}_{10}^{\mathtt{c}} \lesssim   \varepsilon \int_{I_*} \Vert \nabla^2 \bv \Vert_{L^2(\Omega_\eta)}^2   \dt  + c(\varepsilon). 
\end{equation}
Finally, we derive that 
\begin{equation} \label{eq:R10d}
 {\mathlarger{\mathfrak{R}}}_{10}^{\mathtt{d}} \lesssim   \int_{I_*} \Vert \partial_t \vartheta \Vert_{L^2(\Omega_\eta)}^2   \dt  +  \int_{I_*}  \Vert \nabla\bv \Vert_{L^2(\Omega_\eta)}^2  \dt .
\end{equation}
The entropy balance together 
\eqref{eq:entropy} together with the boundedness of the temperature assumed in \ref{A2} allows to control the second integral.
Combining \eqref{eq:R10a}--\eqref{eq:R10d} immediately yields the desired estimate for ${\mathlarger{\mathfrak{R}}}_{10}$. 

Moreover, the estimates for $\nabla p$ follow again from 
the analysis of the steady boundary value problem 
\begin{equation}\label{eq:BVP}
\left\{\begin{aligned}
&\mathcal{A} (\bv, p)^\intercal
= 
\big( \rho\dot{\bv} , (2\mu + \lambda)^{-1}{\mathlarger{\mathfrak{F}}} \big)^\intercal &&\text{ in }  \Omega_\eta,\\[0.4em]
&\mathcal{B}(\bv, p)^\intercal
= 
\left(\partial_t\eta\bn\right)\circ\bfvarphi_{\eta}^{-1} &&\text{ on }  \partial\Omega_\eta,
\end{aligned}\right.
\end{equation}
where the interior operator is
\[\mathcal{A}  := \begin{pmatrix}
\mu\Delta \mathbb{I}_{3\times3} + (\mu+ \lambda)\nabla\Div &   & -\nabla \\\\
\Div &  & -(2\mu + \lambda)^{-1}
\end{pmatrix},
\]
the boundary operator is
\[\mathcal{B}  := \begin{pmatrix}
 \mathbb{I}_{3\times3}  &   & \mathbf{0}_3 \end{pmatrix},
\]
and 
\[ {\mathlarger{\mathfrak{F}}} := (2\mu + \lambda)\Div\bv - p(\rho,\vartheta) \]
denotes the effective viscous flux. 
\medskip
As shown in \cite[Lemma]{MN} for every fixed time $t \in I_*$,  there exists a constant 
\[
C = C\left(\mu, \lambda, \Omega_\eta, \Vert\rho\Vert_{L^{\infty}\left(I_* ; L^\infty( \Omega_\eta)\right) } \right) >0 ,
\]
such that the elliptic estimate
\begin{equation}\label{eq:EllipticRegularity}
\Vert \bv \Vert_{W^{2, 2}(\Omega_\eta)} + \Vert  p\Vert_{W^{1,2}(\Omega_\eta)} \leq C \left( \Vert \rho^{1/2}\dot{\bv} \Vert_{L^2(\Omega_\eta)} + \Vert \partial_t \eta\Vert_{W^{3/2,2}(\omega)} \right)
\end{equation}
holds. Combing all estimates yields the claim.
\end{proof}

\subsection{Higher order estimate}
As discussed above, and as already apparent from the estimate of ${\mathlarger{\mathfrak{R}}}_{10}$, the temperature dependence of the pressure necessitates additional control of thermal quantities. Under the additional assumption \ref{A4} we obtain the following
\begin{proposition}\label{lem:MainResult}
Under the assumptions of
Theorem \ref{theo:MainResult2} we have
\begin{align}\label{eq:22.06}
\begin{aligned}
\sup\limits_{I_*} \int_{\Omega_\eta}& |\nabla \vartheta|^2 \dx  +   \int_{I_*}\int_{\Omega_\eta}  \left( |\nabla^2 \vartheta|^2 +  \rho|\partial_t \vartheta|^2 \right)   \dx\dt\\&\lesssim  \int_\omega \left( |\naby\eta_*|^2 +  |\naby\Dely\eta_0|^2  \right) \dy   + \int_{\Omega_{\eta_0}} \left( |\nabla \bv_0|^2 + |\Div \bv_0|^2    + |\nabla \vartheta_0|^2  + \vartheta_0 \rho_0 \Div\bv_0  \right) \dx
\end{aligned}
\end{align}
\end{proposition}
\begin{proof}
We test the internal energy equation $\eqref{eq:ContMomentEq}_3$ using  the material-derivative-type test function 
\[
\Lambda :=  \partial_t \vartheta + \mathcal{E}_\eta (\partial_t\eta\bn)\cdot \nabla\vartheta 
\]
with teh extension operator from \eqref{eq:ext}.
This yields, upon integrating by parts and using Reynolds transport theorem 
\begin{equation*}
\begin{aligned}
& c_v \int_{I_*} \int_{\Omega_\eta}  \rho |\dot{\vartheta}|^2 \dx\dt   +  \dfrac{\kappa}{2} \int_{I_*} \dfrac{\dd}{\dt} \int_{\Omega_\eta} |\nabla \vartheta|^2 \dx\dt  
\\
& =  \dfrac{\kappa}{2}  \int_{I_*} \int_{\partial\Omega_\eta}  |\nabla\vartheta|^2 \bv \cdot \bn_\eta \dH\dt  + \kappa \int_{I_*} \int_{\Omega_\eta} \Delta \vartheta \left(  \mathcal{E}_\eta (\partial_t\eta\bn)\cdot \nabla\vartheta  \right) \dx\dt  + \int_{I_*} \int_{\Omega_\eta}  \big( \mathbb{S}(\nabla\bv) \colon \nabla \bv \big) \Lambda \dx\dt
\\
& \quad + \int_{I_*} \int_{\Omega_\eta}  \big(\rho\vartheta \Div(\bv) \big) \Lambda \dx\dt  - c_v \int_{I_*} \int_{\Omega_\eta}  \rho \dot{\vartheta} \big(  \mathcal{E}_\eta (\partial_t\eta\bn) - \bv  \big) \cdot \nabla\vartheta \dx\dt
\\
& =:  \sum\limits_{\mathsf{k}=1}^{5} {\mathlarger{\Xi}}_\mathsf{k}  . 
\end{aligned}
\end{equation*}
We now estimate each term ${\mathlarger{\Xi}}_\mathsf{k}$.  By H\"older's inequality and the continuous embeddings 
\begin{equation*}
\begin{aligned}
W^{1/2, 2}(\partial\Omega_\eta) & \hookrightarrow L^4(\partial\Omega_\eta)
\\
W^{1/4, 2}(\partial\Omega_\eta) & \hookrightarrow L^{8/3}(\partial\Omega_\eta),
\end{aligned}
\end{equation*}
it holds that 
\begin{equation*}
{\mathlarger{\Xi}}_1 \lesssim  \int_{I_*}   \Vert  \nabla\vartheta \Vert_{W^{1,2}(\Omega_\eta)}  \Vert  \nabla\vartheta \Vert_{W^{3/4, 2}(\Omega_\eta)}    \Vert \partial_t \eta \Vert_{W^{1/4, 2}(\omega)}   \dt .
\end{equation*}
Using the interpolation identities 
\begin{equation}\label{eq:InterpolationW3414}
\begin{aligned}
W^{3/4, 2}(\Omega_\eta) &= \big[  L^2(\Omega_\eta) , W^{1,2}(\Omega_\eta)  \big]_{3/4}
\\
W^{1/4, 2}(\omega) &= \big[  L^2(\omega) , W^{1,2}(\omega)  \big]_{1/4} ,
\end{aligned}
\end{equation}
together with Young's inequality, we deduce that 
\begin{equation}\label{eq:Xi1}
{\mathlarger{\Xi}}_1 \lesssim  \varepsilon \int_{I_*} \Vert  \nabla\vartheta \Vert_{W^{1,2}(\Omega_\eta)}^2  \dt  + c(\varepsilon) \int_{I_*}  \Vert \nabla\vartheta \Vert_{L^{2}(\Omega_\eta)}^2  \Vert \partial_t \eta \Vert_{W^{1, 2}(\omega)}^2  \dt ,
\end{equation}
where $\varepsilon>0$ is arbitrary.
For the second remainder term  ${\mathlarger{\Xi}}_2$,  an application of H\"older's and  Young's inequalities,  together with   \eqref{eq:InterpolationW3414}  and estimate \eqref{eq:ext}
yields
\begin{equation*}
\begin{aligned}
{\mathlarger{\Xi}}_2 & \lesssim \varepsilon \int_{I_*}  \Vert \Delta \vartheta \Vert_{L^2(\Omega_\eta)}^2 \dt   +   c(\varepsilon) \int_{I_*} \Vert  \nabla\vartheta \Vert_{L^{2}(\Omega_\eta)}^{1/2}   \Vert  \nabla\vartheta \Vert_{W^{1,2}(\Omega_\eta)}^{3/2}  \Vert   \partial_t \eta \Vert_{W^{1/2, 2}(\omega)}^{1/2} \Vert   \partial_t \eta \Vert_{L^ 2(\omega)}^{3/2}  \dt .
\end{aligned}
\end{equation*}
Whence,
\begin{equation}\label{eq:Xi2}
{\mathlarger{\Xi}}_2 \lesssim  \varepsilon \int_{I_*} \Vert  \nabla\vartheta \Vert_{W^{1,2}(\Omega_\eta)}^2  \dt    + c(\varepsilon) \int_{I_*}  \Vert \nabla\vartheta \Vert_{L^{2}(\Omega_\eta)}^2  \Vert \partial_t \eta \Vert_{W^{1, 2}(\omega)}^2  \dt .
\end{equation}
Considering 
\[
{\mathlarger{\Xi}}_3 =  \int_{I_*} \int_{\Omega_\eta}  \big( \mathbb{S}(\nabla\bv) \colon \nabla \bv \big) \big(  \partial_t \vartheta + \mathcal{E}_\eta (\partial_t\eta\bn)\cdot \nabla\vartheta \big) \dx\dt, 
\]
it follows immediately from H\"older's and  Young's inequalities that 
\begin{equation}\label{eq:Xi3}
\begin{aligned}
{\mathlarger{\Xi}}_3 & \lesssim  \varepsilon \int_{I_*}\Big( \Vert \partial_t \vartheta \Vert_{L^2(\Omega_\eta)}^2  +  \Vert  \nabla\vartheta \Vert_{W^{1,2}(\Omega_\eta)}^2  \Big)  \dt  
\\
 & \quad  + c(\varepsilon) \int_{I_*} \Big(   \Vert \nabla\vartheta \Vert_{L^{2}(\Omega_\eta)}^2  \Vert \partial_t \eta \Vert_{W^{1, 2}(\omega)}^2    +    \Vert \nabla\bv \Vert_{L^{2}(\Omega_\eta)}^2 \Vert \nabla\bv \Vert_{L^{\infty}(\Omega_\eta)}^2    \Big) \dt .
 \end{aligned}
\end{equation}
Likewise, we obtain 
\begin{equation}\label{eq:Xi4}
\begin{aligned}
{\mathlarger{\Xi}}_4 & \lesssim  \varepsilon \int_{I_*}\Big( \Vert \partial_t \vartheta \Vert_{L^2(\Omega_\eta)}^2  +  \Vert  \nabla\vartheta \Vert_{W^{1,2}(\Omega_\eta)}^2  \Big)  \dt  
 \\
 & \quad + c(\varepsilon) \int_{I_*} \Big(   \Vert \nabla\vartheta \Vert_{L^{2}(\Omega_\eta)}^2  \Vert \partial_t \eta \Vert_{W^{1, 2}(\omega)}^2    +    \Vert \nabla\bv \Vert_{L^{2}(\Omega_\eta)}^2  \Big) \dt .
 \end{aligned}
\end{equation}
Finally, for ${\mathlarger{\Xi}}_5$, a straightforward application of H\"older's and Young's inequalities gives
\begin{align}\label{eq:Xi5prior}
{\mathlarger{\Xi}}_5 & \lesssim \varepsilon \int_{I_*} \Vert  \rho^{1/2} \dot{\vartheta} \Vert_{L^2(\Omega_\eta)}^2 \dt   + c(\varepsilon) \int_{I_*} \Big(\Vert \rho^{1/2} \mathcal{E}_\eta (\partial_t\eta\bn)\cdot \nabla\vartheta  \Vert_{L^2(\Omega_\eta)}^2   +  \Vert \rho^{1/2} \bv \cdot \nabla\vartheta  \Vert_{L^2(\Omega_\eta)}^2   \Big) \dt . 
\end{align} 
Using H\"older's inequality with $\mathtt{q} := \dfrac{2\mathtt{r}}{\mathtt{r} -2} $ and the Gargliardo--Nirenberg interpolation 
\begin{equation}\label{eq:GargliadoNirenberg}
\Vert \nabla\vartheta\Vert_{L^{\mathtt{q}}(\Omega_\eta)}^2 \lesssim \Vert \vartheta\Vert_{W^{2,2}(\Omega_\eta)}^{\frac{6}{\mathtt{r}}}  \Vert \vartheta\Vert_{W^{1,2}(\Omega_\eta)}^{\frac{2\mathtt{r} - 6}{\mathtt{r}}} ,
\end{equation}
we obtain 
\begin{align*}
{\mathlarger{\Xi}}_5 & \lesssim \varepsilon \int_{I_*} \Vert  \rho^{1/2} \dot{\vartheta} \Vert_{L^2(\Omega_\eta)}^2 \dt   + c(\varepsilon) \int_{I_*} \Vert \rho^{1/2} \mathcal{E}_\eta (\partial_t\eta\bn)\cdot \nabla\vartheta  \Vert_{L^2(\Omega_\eta)}^2 
\\
& \quad   +  c(\varepsilon) \int_{I_*}   \Vert \rho^{1/2}\bv\Vert_{L^{\mathtt{r}}(\Omega_\eta)}^2  \Vert \vartheta\Vert_{W^{2,2}(\Omega_\eta)}^{\frac{6}{\mathtt{r}}}  \Vert \vartheta\Vert_{W^{1,2}(\Omega_\eta)}^{\frac{2\mathtt{r} - 6}{\mathtt{r}}} \dt. 
\end{align*} 
Applying Young's inequality with $\mathtt{s} := \dfrac{2\mathtt{r}}{\mathtt{r}-3} \in [2, \infty) $ and using \eqref{eq:ext} yields 
\begin{equation}\label{eq:Xi5}
\begin{aligned}
{\mathlarger{\Xi}}_5 & \lesssim \varepsilon \int_{I_*} \Big(  \Vert  \rho^{1/2} \dot{\vartheta} \Vert_{L^2(\Omega_\eta)}^2  +  \Vert  \nabla\vartheta \Vert_{W^{1,2}(\Omega_\eta)}^2  \Big)\dt     +  c(\varepsilon)  \int_{I_*}   \Vert \nabla\vartheta \Vert_{L^{2}(\Omega_\eta)}^2  \Vert \partial_t \eta \Vert_{W^{1, 2}(\omega)}^2  \dt
\\
& \quad + c(\varepsilon) \int_{I_*} \Vert \rho^{1/2}\bv\Vert_{L^{\mathtt{r}}(\Omega_\eta)}^{\mathtt{s}}  \Vert \vartheta\Vert_{W^{1, 2}(\Omega_\eta)}^{2} \dt . 
\end{aligned}
\end{equation} 
To close the acceleration estimate \eqref{eq:AccelEstimate}, it only remains to estimate $\Vert  \vartheta \Vert_{W^{2,2}(\Omega_\eta)}$. This is a consequence of elliptic regularity.  More precisely, for each fixed  $t \in I_*$, the internal energy equation $\eqref{eq:ContMomentEq}_3$ yields
\begin{equation}\label{eq:EllipticTheta}
\begin{aligned}
\Vert  \vartheta \Vert_{W^{2,2}(\Omega_\eta)}^2 & \lesssim    \Vert  \rho^{1/2} \partial_t \vartheta  \Vert_{L^{2}(\Omega_\eta)}^2 +  \Vert  \rho^{1/2} \bv \cdot \nabla \vartheta  \Vert_{L^{2}(\Omega_\eta)}^2  +  \Vert  \nabla\bv \Vert_{L^{2}(\Omega_\eta)}^2\Vert  \nabla\bv \Vert_{L^{\infty}(\Omega_\eta)}^2
\\
& \quad + \Vert  \Div(\bv) \Vert_{L^{2}(\Omega_\eta)}^2   + \Vert  \vartheta \Vert_{L^{2}(\Omega_\eta)}^2 , 
\end{aligned}
\end{equation}
where the hidden constant depends on $\Vert  \rho \Vert_{L^{\infty}(\Omega_\eta)}$ and $\Vert  \vartheta \Vert_{L^{\infty}(\Omega_\eta)}$. Importantly, \eqref{eq:EllipticTheta} holds  due to assumption \ref{A3} and the regularity of $\eta$, notably  $\eta\in L^\infty(I_\ast,W^{2,2}(\omega))$, as a consequence of the elliptic regularity theory in  \cite[Chapter 14]{MaSh} (see also \cite[Remark 2.10]{breit2023ladyzhenskaya}). 


\cref{lem:MainResult} then follows by combining the above estimates, choosing $\varepsilon > 0$ sufficiently small, using \cref{theo:MainResult} and applying Gr\"onwall's lemma together with assumption \ref{A4} for the velocity field.  
\end{proof}

\begin{proof}[Proof of \cref{theo:MainResult2}]
We first note that maximal regularity (see, e.g.,  \cite[Section 2, Theorem 2.2]{denk2007optimal}) together with the continuous embedding result of \cite[Section 3, Theorem 3.1]{LionsMagenes1} yields, in view of \eqref{eq:ShellEq}, that for every $\sigma \in [0, 2]$,
\begin{equation}\label{eq:MRforLemma41}
\begin{aligned}
& \int_{I_*} \left( \Vert \partial_t^2 \eta \Vert_{W^{\sigma, 2}(\omega)}^2 + \Vert \partial_t \eta \Vert_{W^{\sigma+2, 2}(\omega)}^2 + \Vert  \eta \Vert_{W^{\sigma + 4, 2}(\omega)}^2  \right) \dt  +  \sup\limits_{I_*} \left(  \Vert \partial_t \eta \Vert_{W^{\sigma + 1, 2}(\omega)}^2  + \Vert \eta \Vert_{W^{\sigma + 3, 2}(\omega)}^2 \right)
\\[0.4em]
& \quad \lesssim \int_{I_*} \left( \Vert \bv \Vert_{W^{\sigma +3/2,2}(\Omega_\eta)}^2 + \Vert p \Vert_{W^{\sigma +1/2,2}(\Omega_\eta)}^2 \right) \dt + \Vert \eta_0 \Vert_{W^{5,2}(\omega)}^2 + \Vert \eta_* \Vert_{W^{3,2}(\omega)}^2.
\end{aligned}
\end{equation}
In view of \cref{theo:MainResult} and \cref{lem:MainResult} the right-hand side stays bounded for $\sigma\leq 1/2$. 
Following the approach of  \cite[Section 4, Lemma 4.1]{MN}, we aim to prove that
\begin{equation}\label{eq:VelocityCondEstimate}
\begin{aligned}
& \int_{I_*} \left( \Vert \partial_t \bv \Vert_{W^{2,2}(\Omega_\eta)}^2   + \Vert \bv \Vert_{W^{4,2}(\Omega_\eta)}^2  + \Vert p \Vert_{W^{3,2}(\Omega_\eta)}^2 \right) \dt  + \sup\limits_{I_*} \Vert \rho \Vert_{W^{3,2}(\Omega_\eta)}^2 
\\[0.4em]
& \quad \lesssim {\mathlarger{\mathtt{E}}}_{\mathrm{acc}}  +   \Vert \bv_0 \Vert_{W^{3,2}(\Omega_{\eta_0})}^2 +  \Vert \rho_0 \Vert_{W^{3,2}(\Omega_{\eta_0})}^2 + \Vert \eta_0 \Vert_{W^{5,2}(\omega)}^2 +  \Vert \eta_* \Vert_{W^{3,2}(\omega)}^2,
\end{aligned}
\end{equation}
with the hidden constant depending on \ref{A1}--\ref{A4},  $T_*$,   and  $C_0$.

For this purpose, we introduce the time-dependent  diffeomorphism 
\[
\bfPsi_{\eta \to \eta_0 }  := \bfPsi_\eta \circ \left(\bfPsi_{\eta_0} \right)^{-1} \colon \Omega_{\eta_0} \to   \Omega_{\eta},  
\]
and define the pull-back density, velocity and temperature by 
\[
\underline{\rho} = \rho \circ \bfPsi_{\eta \to \eta_0 }, \qquad \underline{\bv} := \bv\circ \bfPsi_{\eta \to \eta_0 }, \qquad \underline{\vartheta} := \vartheta\circ \bfPsi_{\eta \to \eta_0 }. 
\]
Moreover, we define
\begin{equation*}\label{matrices}
\begin{aligned}
\mathbf{A}_{\eta \to \eta_0 } & =J_{\eta \to \eta_0 } \big(\nabx \bfPsi_{\eta \to \eta_0 }^{-1}\circ \bfPsi_{\eta \to \eta_0 }\big)\big( \nabx \bfPsi_{\eta \to \eta_0 }^{-1}\circ \bfPsi_{\eta \to \eta_0 } \big)^\intercal,
\\[0.4em]
\mathbf{B}_{\eta \to \eta_0 } & =J_{\eta \to \eta_0 } \left(\nabx \bfPsi_{\eta \to \eta_0 }^{-1}\circ \bfPsi_{\eta \to \eta_0 }\right)^\intercal,
\end{aligned}
\end{equation*}
where $J_{\eta \to \eta_0 } = \mathrm{det}(\nabla\bfPsi_{\eta \to \eta_0 })$.   \\

We next reduce to homogeneous boundary conditions by setting  
\[ \underline{\bu} = \underline{\bv} - \mathcal{E}_{\eta_0}(\partial_t\eta\bn),\]
where $\mathcal{E}_{\eta_0}$ is the extension operator from \eqref{eq:ext}.
For convenience, we denote by $\mathcal{L}$ the Lam\'e operator 
\[
\mathcal{L}\,\underline{\bu} := \mu\Delta\underline{\bu} + (\lambda + \mu)\nabla\Div\underline{\bu},  
\] 
and further introduce  the operators 
\begin{equation*}
\begin{aligned}
\underline{\mathsf{B}}\, \underline{\bu} & :=  - J_{\eta \to \eta_0 } \underline{\rho} \nabla\underline{\bu} \cdot \partial_t \bfPsi_{\eta \to \eta_0 }^{-1}\circ \bfPsi_{\eta \to \eta_0 }   -  \underline{\rho} \underline{\bu}  \big( \nabla\mathcal{E}_{\eta_0}(\partial_t\eta\bn) \colon  \mathbf{B}_{\eta \to \eta_0 }    \big)  -   \underline{\rho} \mathcal{E}_{\eta_0}(\partial_t\eta\bn)  \big( \nabla \underline{\bu} \colon \mathbf{B}_{\eta \to \eta_0 } \big),
\end{aligned}
\end{equation*}
and 
\begin{equation*}
\begin{aligned}
\underline{\mathsf{f}} & :=  - J_{\eta \to \eta_0 } \underline{\rho} \partial_t \mathcal{E}_{\eta_0}(\partial_t\eta\bn)  -  J_{\eta \to \eta_0 } \underline{\rho} \nabla\mathcal{E}_{\eta_0}(\partial_t\eta\bn) \cdot \partial_t \bfPsi_{\eta \to \eta_0 }^{-1}\circ \bfPsi_{\eta \to \eta_0 }  - \underline{\rho} \underline{\bu} \big( \nabla \underline{\bu} \colon \mathbf{B}_{\eta \to \eta_0 } \big)  
\\[0.4em]
& \quad \;\,  - \underline{\rho} \mathcal{E}_{\eta_0}(\partial_t\eta\bn) \big( \nabla\mathcal{E}_{\eta_0}(\partial_t\eta\bn)  \colon\mathbf{B}_{\eta \to \eta_0 }  \big)  - \mathbf{B}_{\eta \to \eta_0 } \nabla p(\underline{\rho},\underline\vartheta)
\\[0.4em]
& \quad \;\, + \Div\Bigg[  \mu \left(\mathbf{A}_{\eta \to \eta_0 }  - \mathbb{I}_{3 \times 3}   \right) \nabla\underline{\bu}  + (\lambda + \mu)\biggl( \Bigl(  \left(\mathbf{B}_{\eta \to \eta_0 } - \mathbb{I}_{3 \times 3} \right) \colon \nabla\underline{\bu} \Bigr)  \mathbb{I}_{3 \times 3}  
\\[0.4em]
& \qquad \qquad + \left( \mathbf{B}_{\eta \to \eta_0 } \colon \nabla\underline{\bu}  \right)  \left( \dfrac{1}{J_{\eta \to \eta_0 }} \mathbf{B}_{\eta \to \eta_0 } -  \mathbb{I}_{3 \times 3}  \right)                 \biggr)      \Bigg]      
\\[0.4em]
& \quad \;\, + \Div\Bigg[ \mu  \mathbf{A}_{\eta \to \eta_0 }  \nabla\mathcal{E}_{\eta_0}(\partial_t\eta\bn) +   \dfrac{(\lambda + \mu)}{J_{\eta \to \eta_0}}   \Big( \mathbf{B}_{\eta \to \eta_0 } \colon \nabla \mathcal{E}_{\eta_0}(\partial_t\eta\bn)     \Big)     \mathbf{B}_{\eta \to \eta_0 }     \Bigg].
\end{aligned}
\end{equation*}
It is shown in \cite[Section 4, Lemma 4.1]{MN} that $\underline{\bu}$ solves the problem
\begin{equation}\label{eq:CauchyPb-u}
\tag{$\mathtt{CP}$}\left\{
\begin{aligned}
\partial_t\underline{\bu} & =  \mathsf{A} \underline{\bu} + \mathsf{B}\underline{\bu}  + \mathsf{f}  && \text{ in } I_* \times \Omega_{\eta_0},
\\[0.4em]
\underline{\bu}(0) & = \underline{\bu}_0   && \text{ in } \Omega_{\eta_0}.  
\end{aligned}\right.
\end{equation}
 and satisfies the estimate
\begin{equation}\label{eq:uMRestimatePrior}
\begin{aligned}
&  \Vert \partial_t \underline{\bu} \Vert^2_{L^2\left(I_*;  W^{2,2}(\Omega_{\eta_0})   \right)} + \Vert \underline{\bu} \Vert^2_{L^2\left(I_*;  W^{4,2}(\Omega_{\eta_0})   \right)}  
\\[0.4em]
& \quad \lesssim  \Vert \bu_0 \Vert^2_{W^{3,2}(\Omega_{\eta_0}) } +  \Vert  \mathsf{B}\underline{\bu} \Vert^2_{L^2\left(I_*;  W^{2,2}(\Omega_{\eta_0})   \right)} +  \Vert  \mathsf{f} \Vert^2_{L^2\left(I_*;  W^{2,2}(\Omega_{\eta_0})   \right)} , 
\end{aligned}
\end{equation}
\begin{equation}\label{eq:BuEstim}
\begin{aligned}
\Vert  \mathsf{B}\underline{\bu}  \Vert^2_{L^2\left(I_*;  W^{2,2}(\Omega_{\eta_0})   \right)} & \lesssim \kappa \int_{I_*} \Vert \underline{\bu} \Vert_{W^{4,2}(\Omega_{\eta_0})}^2  \dt  +  c(\kappa) \int_{I_*} \Vert \underline{\bu} \Vert_{W^{2,2}(\Omega_{\eta_0})}^2 \Vert \partial_t \eta \Vert_{W^{3,2}(\omega)}^2  \dt      +   {\mathlarger{\mathtt{E}}}_{\mathrm{acc}} 
\end{aligned}
\end{equation}
and 
\begin{equation}\label{eq:fEstim}
\begin{aligned}
&  \Vert  \mathsf{f} + \mathcal{E}_{\eta_0}(\partial_t\eta\bn)  \Vert^2_{L^2\left(I_*;  W^{2,2}(\Omega_{\eta_0})   \right)} 
\\[0.4em]
& \quad  \lesssim    \sup\limits_{I_*} \Vert \underline{\bu} \Vert_{W^{3,2}(\Omega_{\eta_0})}^2   \int_{I_*}  \Vert \underline{\bu} \Vert_{W^{2,2}(\Omega_{\eta_0})}^2  \dt  + \int_{I_*} \Vert p(\underline{\rho}) \Vert_{W^{3,2}(\Omega_{\eta_0})}^2 \dt    + \kappa \sup\limits_{I_*} \Vert \partial_t \eta \Vert_{W^{5/2,2}(\omega)}^2    +   {\mathlarger{\mathtt{E}}}_{\mathrm{acc}} 
 \\[0.4em]
 & \qquad     +  T_*^{1/2}  \int_{I_*} \Vert \underline{\bu} \Vert_{W^{4, 2}(\omega)}^2 \dt    + T_* \int_{I_*} \Vert \Dely^3\eta \Vert_{L^2(\omega)}^2 \dt   + \kappa \int_{I_*} \Vert \partial_t\Dely^2\eta \Vert_{L^2(\omega)}^2 \dt 
 \\[0.4em]
 &\qquad  +  c(\kappa)  \int_{I_*} \Vert \eta \Vert_{W^{9/2,2}(\omega)}^2 \Vert \partial_t \eta \Vert_{W^{3,2}(\omega)}^2   \dt .
\end{aligned}
\end{equation}
The only difference here lies in the estimate for the contribution of the pressure that is
\begin{align*}
{\mathlarger{\mathfrak{S}}}_4:=\left(J_{\eta \to \eta_0}\underline{\rho} \right)^{-1} \mathbf{B}_{\eta \to \eta_0 } \nabla p(\underline{\rho},\underline\vartheta)=\left(J_{\eta \to \eta_0} \right)^{-1} \mathbf{B}_{\eta \to \eta_0 }\nabla\underline\vartheta+\left(J_{\eta \to \eta_0}\underline{\rho} \right)^{-1} \mathbf{B}_{\eta \to \eta_0 }\vartheta\nabla\underline\varrho.
\end{align*}
We obtain
\begin{align*}
\int_{I_*}\|{\mathlarger{\mathfrak{S}}}_4\|^2_{W^{2,2}(\Omega_{\eta_0})}\dt&\lesssim \int_{I_*}\|\underline\rho\|^2_{W^{2,2}(\Omega_{\eta_0})}\|\eta\|^2_{W^{3,2}(\omega)}\|\underline\vartheta\|^2_{W^{3,2}(\Omega_{\eta_0})}\dt\\
&+\int_{I_*}\|\underline\rho^{-1}\|^2_{W^{2,2}(\Omega_{\eta_0})}\|\eta\|^2_{W^{3,2}(\omega)}\|\underline\vartheta\nabla\underline\rho\|^2_{W^{2,2}(\Omega_{\eta_0})}\dt\\
&\lesssim \int_{I_*}\|\underline\vartheta\|^2_{W^{3,2}(\Omega_{\eta_0})}\dt+\int_{I_*}\|\underline\vartheta\nabla\underline\rho\|^2_{W^{2,2}(\Omega_{\eta_0})}\dt,
\end{align*}
where, by assumptions \ref{A1} and \ref{A4} as well as \cref{lem:MainResult},
\begin{align*} 
&\int_{I_*}\|\underline\vartheta\nabla\underline\rho\|^2_{W^{2,2}(\Omega_{\eta_0})}\dt
\\
&\lesssim \int_{I_*}\|\underline\vartheta\nabla\underline\rho\|^2_{L^{2}(\Omega_{\eta_0})}\dt+\int_{I_*}\|\nabla^2(\underline\vartheta\nabla\underline\rho)\|^2_{L^{2}(\Omega_{\eta_0})}\dt\\
&\lesssim 1+\int_{I_*}\|\underline\vartheta\nabla^3\underline\rho\|^2_{L^{2}(\Omega_{\eta_0})}\dt+\int_{I_*}\|\nabla\underline\vartheta\nabla^2\underline\rho\|^2_{L^{2}(\Omega_{\eta_0})}\dt+\int_{I_*}\|\nabla^2\underline\vartheta\nabla\underline\rho\|^2_{L^{2}(\Omega_{\eta_0})}\dt\\
&\lesssim 1+\int_{I_*}\|\underline\rho\|^2_{W^{3,2}(\Omega_{\eta_0})}\dt+\int_{I_*}\|\nabla\underline\vartheta\|^2_{L^4}\|\nabla^2\underline\rho\|^2_{L^{4}(\Omega_{\eta_0})}\dt+\int_{I_*}\|\nabla^2\underline\vartheta\|_{L^4(\Omega_{\eta_0})}^2\|\nabla\underline\rho\|^2_{L^{4}(\Omega_{\eta_0})}\dt\\
&\lesssim 1+\int_{I_*}\|\underline\rho\|^2_{W^{3,2}(\Omega_{\eta_0})}\dt+\int_{I_*}\|\nabla\underline\vartheta\|_{L^{\infty}(\Omega_{\eta_0})}\|\nabla\underline\vartheta\|_{L^{2}(\Omega_{\eta_0})}\|\nabla\underline\rho\|_{L^{\infty}(\Omega_{\eta_0})}\|\nabla\underline\rho\|_{W^{2,2}(\Omega_{\eta_0})}\dt 
\\
&+ \int_{I_*}\|\underline\vartheta\|_{W^{4,2}(\Omega_{\eta_0})}\|\underline\rho\|_{W^{2,2}(\Omega_{\eta_0})}\dt
\\
&\lesssim 1+\int_{I_*}\Big(1+\|\nabla\underline\rho\|_{L^{\infty}(\Omega_{\eta_0})}^2+\|\nabla\underline\vartheta\|^2_{L^{\infty}(\Omega_{\eta_0})}\Big)\|\underline\rho\|^2_{W^{3,2}(\Omega_{\eta_0})}\dt+\kappa\int_{I_*}\|\underline\vartheta\|_{W^{4,2}}^2\dt. 
\end{align*}
Note that by assumption \ref{A4} one can apply a Gronwall argument to the first term.
Finally, we argue as in \eqref{eq:31.03B} which gives (setting $(\eta_1,\overline\bu_1,\rho_1,\vartheta_1)=(\eta,\overline\bu,\rho,\vartheta)$ and $(\eta_2,\overline\bu_2,\rho_2,\vartheta_2)=0$)
\begin{align}\label{eq:theend}
\int_I\int_\Omega |\partial_t^2\overline\vartheta|^2\dxt+\int_I\|\overline\vartheta\|_{W^{4,2}_x}^2\dt\lesssim T_* {\tt E}_{\tt high}.
\end{align}
Note that the regularity of $\eta$ obtained from \eqref{eq:MRforLemma41} and Theorem \ref{theo:MainResult} is just ciritical for this estimate. This finishes the proof of \eqref{eq:VelocityCondEstimate}.
Eventually, we can allow $\sigma\leq 5/2$ in \eqref{eq:MRforLemma41}. 
This completes the proof of Theorem \ref{theo:MainResult2}. 
\end{proof}


 \section*{Acknowledgements}

This work was funded by the Deutsche Forschungsgemeinschaft (DFG) -- Projektnummer 543675748.

\section*{Compliance with Ethical Standards}
\smallskip
\par\noindent
{\bf Conflict of Interest}. The authors declare that they have no conflict of interest.

\smallskip
\par\noindent
{\bf Data Availability}. Data sharing is not applicable to this article as no datasets were generated or analysed during the current study.

\end{document}